\documentclass[11pt]{article}
\usepackage{hhline}
\usepackage{booktabs}   
\usepackage{makecell}   
\usepackage{multirow}  
\usepackage{amssymb}    
\usepackage[table]{xcolor}
\usepackage{caption}
\usepackage{graphicx,subfig}
\usepackage{amsmath}
\usepackage{float}
\usepackage{geometry}   
\usepackage{fancyhdr}   
\usepackage{listings} 
\usepackage{xcolor} 
\usepackage{amsthm}
\usepackage{algpseudocode}
\usepackage{algorithm}
\usepackage[colorlinks,
linkcolor=blue,
anchorcolor=blue,
citecolor=blue]{hyperref}
\usepackage{comment}
\usepackage[round]{natbib}
\usepackage{enumitem}
\usepackage{cleveref}

\newtheorem{theorem}{Theorem}
\newtheorem{lemma}{Lemma}

\newtheorem{proposition}{Proposition}
\newtheorem{corollary}{Corollary}
\newtheorem{assumption}{Assumption}

\newtheorem{example}{Example}

\definecolor{lightblue}{RGB}{220,235,255}
\newcommand\fro[1]{\| #1 \|_{\rm{F}}}
\newcommand\op[1]{\left\| #1 \right\|}
\newcommand\bop[1]{\big\| #1 \big\|}
\newcommand\psione[1]{\left\| #1 \right\|_{\psi_1}}
\newcommand\psitwo[1]{\left\| #1 \right\|_{\psi_2}}

\newcommand\lzero[1]{\| #1 \|_{0}}

\newcommand{\mat}[1]{\begin{bmatrix}#1 \\ \end{bmatrix}}
\newcommand{\inp}[2]{\langle #1,#2\rangle}

\makeatletter
\def\blfootnote{\gdef\@thefnmark{}\@footnotetext}

\def\calE{{\mathcal E}}
\def\calF{{\mathcal F}}
\def\calG{{\mathcal G}}
\def\calH{{\mathcal H}}

\def\calL{{\mathcal L}}

\def\calP{{\mathcal P}}

\def\calR{{\mathcal R}}

\def\calX{{\mathcal X}}

\def\EE{{\mathbb E}}

\def\NN{{\mathbb N}}
\def\OO{{\mathbb O}}
\def\PP{{\mathbb P}}

\def\RR{{\mathbb R}}
\def\SS{{\mathbb S}}

\def\g{{\boldsymbol g}}

\def\q{{\boldsymbol q}}

\def\s{{\boldsymbol s}}

\def\u{{\boldsymbol u}}
\def\v{{\boldsymbol v}}
\def\w{{\boldsymbol w}}
\def\x{{\boldsymbol x}}
\def\y{{\boldsymbol y}}
\def\z{{\boldsymbol z}}

\def\A{{\boldsymbol A}}
\def\B{{\boldsymbol B}}
\def\C{{\boldsymbol C}}
\def\D{{\boldsymbol D}}
\def\E{{\boldsymbol E}}

\def\G{{\boldsymbol G}}
\def\H{{\boldsymbol H}}
\def\I{{\boldsymbol I}}
\def\J{{\boldsymbol J}}

\def\M{{\boldsymbol M}}
\def\N{{\boldsymbol N}}
\def\O{{\boldsymbol O}}
\def\P{{\boldsymbol P}}
\def\Q{{\boldsymbol Q}}
\def\R{{\boldsymbol R}}
\def\S{{\boldsymbol S}}
\def\T{{\boldsymbol T}}
\def\U{{\boldsymbol U}}
\def\V{{\boldsymbol V}}
\def\W{{\boldsymbol W}}
\def\X{{\boldsymbol X}}
\def\Y{{\boldsymbol Y}}
\def\Z{{\boldsymbol Z}}

\def\bDelta{{\boldsymbol \Delta}}
\def\bLa{{\boldsymbol \Lambda}}
\def\bGa{{\boldsymbol \Gamma}}

\def\tr{\textsf{tr}}
\def\snr{\textsf{SNR}}
\def\hat{\widehat}
\def\tilde{\widetilde}
\def\bXi{{\boldsymbol \Xi}}
\def\hjive{{\sf HeteroJIVE}}

\newcommand\brac[1]{\left(#1\right)}
\newcommand\bbrac[1]{\big(#1\big)}
\newcommand\ebrac[1]{\left\{#1\right\}}
\newcommand\sqbrac[1]{\left[#1\right]}
\newcommand\ab[1]{\left|#1\right|}

\begin{document}
\title{Spectral Joint Subspace Estimation for Heterogeneous Multi-View Data: Geometry and Reweighting}\blfootnote{Author names are sorted alphabetically.}

\author{Jingyang Li\thanks{Fudan University. Email: \texttt{jjyyli@fudan.edu.cn}.}
\and Zhongyuan Lyu\thanks{The University of Sydney. Email: \texttt{zhongyuan.lyu@sydney.edu.au}.}
}
\maketitle

\begin{abstract}
Many modern datasets consist of multiple related matrices measured on a common set of units, with the goal of recovering a shared low-dimensional subspace. The Angle-based Joint and Individual Variation Explained (AJIVE) framework addresses this problem through equal-weight aggregation, which can be suboptimal when views exhibit statistical heterogeneity in signal-to-noise ratios and dimensions, as well as structural heterogeneity from individual components. 
For equal-weight AJIVE, we show that the previously identified ``non-diminishing'' error barrier is geometry dependent: under near-orthogonal deterministic loading orientations, the second-order term is reduced, whereas under sign-symmetric random loadings, it is centered and averages out, yielding a $K^{-1/2}$-type rate without iterative refinement. Under a majority sign-alignment condition in rank-one setting, a bias at the squared single-view perturbation scale can persist. For general weights, we establish error bounds that disentangle the two layers of heterogeneity, and propose \hjive, the weighted AJIVE estimator using an explicit weight that is optimal whenever its identifiability gap is constant. We also provide a data-driven plug-in implementation of \hjive, together with an optional geometry-adaptive extension of this data-driven procedure. Simulations and analyses of multi-omics and image data illustrate the practical benefits of \hjive.
\end{abstract}
\section{Introduction}

In modern applications, data are often recorded as collections of matrices on a common set of units. Examples include computational biology, where the same tissue samples are profiled by multiple assays in TCGA \citep{cancer2012comprehensive, lock2013joint, zhang2022joint, sergazinov2024spectral}, chemometrics, where the same chemicals are characterized by multiple instruments \citep{smilde2003framework, lofstedt2013global}, and autonomous driving, where synchronized sensors (cameras, LiDAR, radar) are recorded for the same time-stamped scenes \citep{geiger2013vision,caesar2020nuscenes}. In all these examples, the data can be represented by $K$ matrices $\ebrac{\A_k}_{k=1}^K$, and each $\A_k$ is of dimension $n$ by $d_k$, where $n$ is the number of units and $d_k$ is the view-specific dimension. 
It is often interesting to identify the joint latent structure among $\A_k$'s while allowing view-specific structures to persist. For instance, on TCGA multi-omics, the joint low-dimensional space captures subtype structure while individual parts retain assay-specific effects (gene expression, miRNA, genotype, protein/RPPA), and ignoring the individual structure  can lead to severe information loss \citep{lock2013joint,gaynanova2019structural}. 

The Joint and Individual Variation Explained (JIVE) framework \citep{lock2013joint} formalizes this by assuming data matrices $\ebrac{\A_k}_{k=1}^K$ admit the following decomposition: 
\begin{align}\label{eq:JIVE-model}
	\A_k = \U\V_k^\top + \U_k\W_k^\top + \E_k\in\RR^{n\times d_k},\qquad k\in[K].
\end{align}
Here $\U\in\OO_{n,r}$ and $\U_k\in \OO_{n,r_k}$ are  the \emph{joint} and \emph{individual} subspaces, respectively, where  $\OO_{p,q}$ denotes the space of orthonormal matrices of dimension $p$ by $q$, $\V_k\in \RR^{d_k\times r}$ and $\W_k\in \RR^{d_k\times r_k}$  are the loading matrices, and $\E_k$ is the random noise. It is commonly assumed that $r+r_k\ll n\wedge d_k$ for all $k\in[K]$, so that one can leverage the spectral method to exploit the low-rankness of $\U\V_k^\top + \U_k\W_k^\top$. Among the most popular algorithms for this task, Angle-based JIVE (AJIVE) \citep{feng2018angle} first estimates the joint and individual subspaces $\mat{\U& \U_k}$ by performing SVD on each $\A_k$ individually, denoted by $\tilde \U_k$, then aggregates the results to estimate the shared subspace $\U$ via a second SVD, i.e.,
\begin{align*}
    \hat\U=\text{Eigen}_r\brac{\sum_{k=1}^K\tilde \U_k\tilde \U_k^\top},
\end{align*}
where $\text{Eigen}_r(\A)$ denotes the leading-$r$ eigenvectors of matrix $\A$. Though equipped with  different names, this two-stage spectral algorithm has  been widely used in many integrative analyses  \citep{fan2019distributed, arroyo2021inference,li2024federated,tang2025mode} due to its computational efficiency. 

The standard AJIVE algorithm implicitly assumes that all views contribute equally to the estimation of the joint subspace. Nonetheless, a pervasive feature of multi-view data in practice is heterogeneity across views, i.e., the SNR and view-specific dimensions $\ebrac{d_k}_{k=1}^K$ of $\A_k$'s may vary substantially from one view to another. We term this \emph{statistical heterogeneity}. In such cases, equal-weight aggregation in the second step of AJIVE can be suboptimal, as it treats all views as equally informative. We illustrate this limitation with the following toy example.
\begin{example}\label{exmp:intro}
    Consider $K=2$, $\A_1=\U\V^\top+2\U_1\W^\top+\E_1$ and $\A_2=100(\U\V^\top+2\U_2\W^\top)+\E_2$, where $\E_1,\E_2\overset{\rm i.i.d.}{\sim} N(0,0.1^2)$. 
   In this case, the performance of $\hat\U:=\text{Eigen}_r\brac{\tilde \U_1\tilde \U_1^\top+\tilde \U_2\tilde \U_2^\top}$ is largely deteriorated by the low-SNR of $\A_1$. At the same time, direct SVD on $\A_2$ alone fails to identify $\U$ due to the factor $2$ in front of the individual component $\U_2\W^\top$. 
\end{example}
\begin{figure}[!th]
    \centering
    \includegraphics[width=0.6\textwidth]{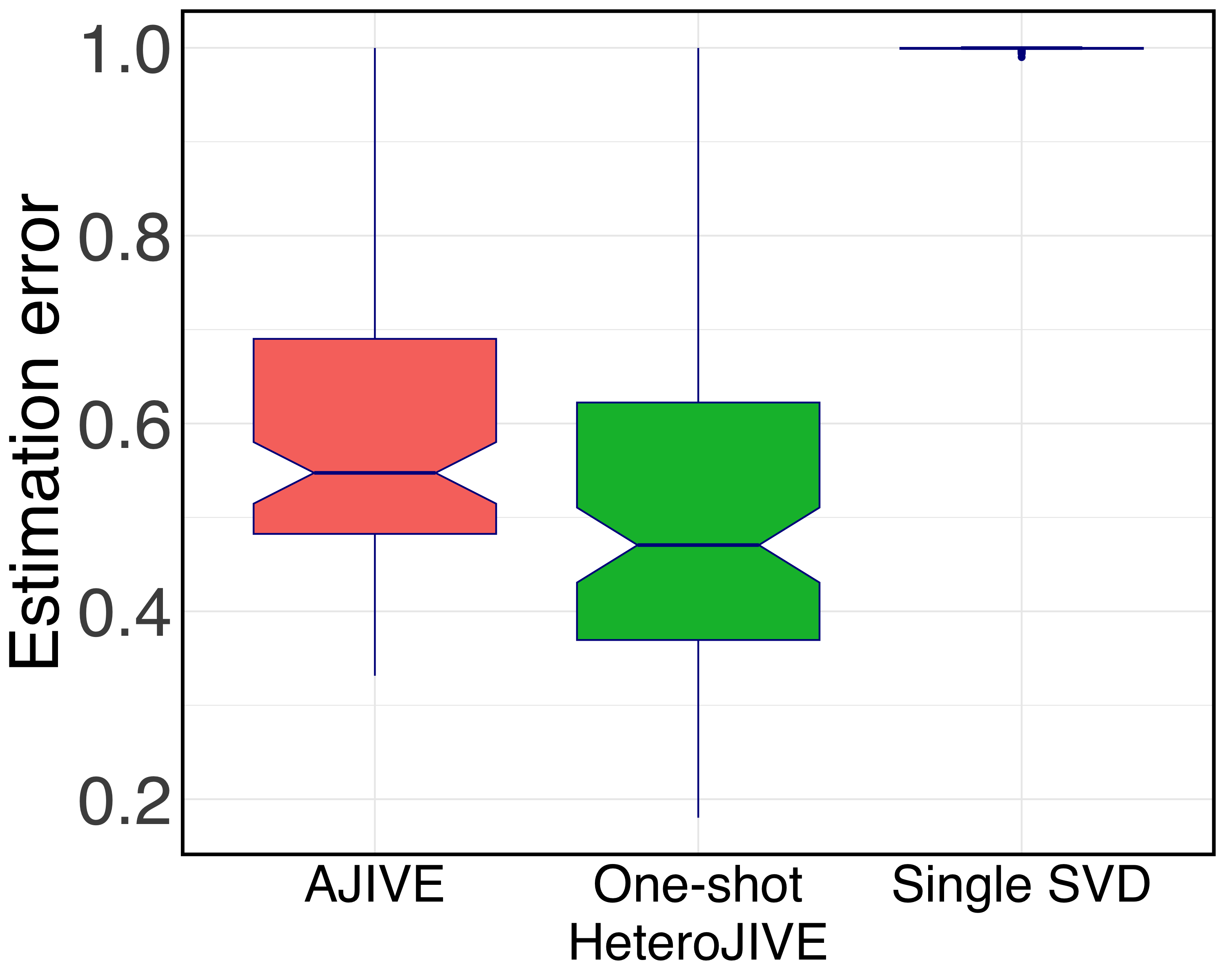}
\hfill
    \caption{Estimation error of $\|\hat\U\hat\U^\top-\U\U^\top\|$ across methods for \Cref{exmp:intro} averaged over $100$ replicates. The average weights for one-shot \hjive{} are given by $(0.99995,0.00005)$. }
    \label{fig:exmp}
\end{figure}
Beyond statistical heterogeneity, the JIVE model introduces a more fundamental challenge. Unlike standard multi-view low-rank models which assume data consist solely of a low-rank component and noise (i.e., $\A_k = \U\V^\top + \E_k$) \citep{hong2023optimally,ma2024optimal,baharav2025stacked}, the JIVE framework explicitly incorporates view-specific individual variations $\{\U_k \W_k^\top\}_{k=1}^K$. 
We term the interaction between these individual subspaces and the joint subspace as \emph{structural heterogeneity}.

To address these challenges, we propose \hjive, defined as the weighted AJIVE estimator using the explicit weights given in \Cref{sec:weights}. By down-weighting views with lower SNR or higher interference, we prevent noisy views from contaminating the joint signal. 
Although the exact minimizer of the fully geometry-dependent criterion need not admit a general closed form, we derive an  explicit weighting rule and prove that it is optimal relative to the oracle criterion whenever it preserves a constant identifiability gap. The resulting weighted estimator is \hjive. We further provide a computationally efficient data-driven plug-in implementation and an optional geometry-adaptive extension. For illustration, we consider the setup in \Cref{exmp:intro}, and \Cref{fig:exmp} demonstrates that
\hjive{} significantly outperforms both AJIVE and single SVD (on $\A_2$). Notably, even with a negligible weight ($0.00005$) assigned to the noisy view $\A_1$, integrating both views yields superior identification compared to using $\A_2$ alone.

To connect the weighting problem with the baseline AJIVE geometry, we first analyze equal-weight AJIVE. Our analysis identifies the geometry of the second-order term and provides the decomposition used later for general weights.
Even for AJIVE, its statistical guarantee has only recently been clarified, especially under the large-$K$ regime. In particular, \cite{yang2025estimating} 
establishes 
fundamental limits of AJIVE, identifying a ``non-diminishing" barrier in the estimation error: 
in low-SNR regimes, a second-order bias term does not vanish even as $K$ increases.

To demystify the phenomenon, in this paper we show that the non-diminishing term is governed by the geometry between the joint and individual loadings. When the their interactions are small, or when their signs are symmetrically randomized across views, this second-order contribution is centered and averages out, yielding a $K^{-1/2}$-type rate. In contrast, under a majority sign-alignment condition (see \Cref{ass:lb-majority}) in the rank-one setting, we prove an algorithm-specific lower bound showing that the equal-weight AJIVE output can retain a second-order bias at the scale of the squared single-view perturbation. Thus, the same AJIVE map can exhibit either view-wise averaging or persistent second-order bias, depending on the loading geometry. Therefore, our results sharpen the existing upper bounds and clarify the precise regimes in which the non-vanishing second-order term can be intrinsic under certain geometric conditions.


\subsection{Main contributions}
Our contributions span the geometry of equal-weight AJIVE, sharp guarantees for deterministic and random loadings, general-weight theory, and a data-driven reweighting procedure.
\begin{table}[!htbp]
\centering
\caption{Comparison of theoretical guarantees of our framework under the equal-weight case (i.e., AJIVE) with the results of \cite{yang2025estimating} and \cite{tang2025mode}. For deterministic loadings, explicit forms of the error rates are deferred to \Cref{sec:theory}. For random loadings, our results in the table assume that the individual subspaces $\ebrac{\U_k}_{k=1}^K$ are also random, to facilitate comparison with \cite{tang2025mode}. The distributional setting of \cite{tang2025mode} is more general, and the table compares only the intersection of their assumptions with ours. Here,  $\varepsilon$ denotes the subspace estimation error of the initial SVD on a single view, and $\theta$ characterizes the alignment of individual subspaces $\ebrac{\U_k}_{k=1}^K$.}
\vspace{0.2cm}

\resizebox{\textwidth}{!}{%
\begin{tabular}{ll|cc|cc}
\hline
\multicolumn{2}{c|}{} & \multicolumn{2}{c|}{\rule{0pt}{1.2em}Deterministic}& \multicolumn{2}{c}{Random}\\[0.3em]
\hhline{~~|----}
\multicolumn{2}{c|}{} & \rule{0pt}{1.2em}\cite{yang2025estimating} & \multicolumn{2}{c}{\cellcolor{lightblue}\textbf{This paper}} & \cite{tang2025mode} \\[0.3em]
\hline
\multicolumn{2}{c|}{\textbf{Model}} 
    & \makecell[c]{JIVE\\with equal $\sigma_k$'s} 
    &\multicolumn{2}{c}{\cellcolor{lightblue} General JIVE}
    & \makecell[c]{JIVE\\with $\W_k = \W$} \\

\hline
\multicolumn{2}{c|}{\textbf{Algorithm}} 
    & \makecell[c]{Two-stage spectral} 
    &\multicolumn{2}{c}{\cellcolor{lightblue} Two-stage spectral}
    & \makecell[c]{Two-stage spectral followed\\  by iterative refinement} \\

\hline
\multirow{4}{*}{\textbf{Error rate}} 
  & \multicolumn{1}{|c|}{w.r.t. $K$} 
  & \makecell[c]{Non-vanishing bound\\as $K\rightarrow\infty$} 
  & \cellcolor{lightblue}{\makecell[c]{Sharper upper bound; \\vanishing under \\loading-angle cancellation; \\$\Omega(\varepsilon^2)$ lower bound \\under majority sign alignment}}
  & \multirow{4}{*}{\cellcolor{lightblue}\makecell[c]{\vspace{-2cm}\\Matches minimax lower bound\\in the overlapping regime} } 
  & \multirow{4}{*}{\makecell[c]{\vspace{-1.2cm}\\Initialization error \\decay at $O(\varepsilon^2)$; \\Refinement error \\decay at $\tilde O(K^{-1/2})$}} \\
\hhline{~|---| >{\arrayrulecolor{lightblue}}- >{\arrayrulecolor{black}}~}
  & \multicolumn{1}{|c|}{w.r.t. $\varepsilon$}
  & \makecell[c]{First-order suboptimal\\ when $K\theta\lesssim 1$} 
  & \cellcolor{lightblue}{First-order optimal} 
  & \cellcolor{lightblue}{\makecell[c]{\vspace{-2cm}\\decay at $\tilde O(\varepsilon K^{-1/2})$}}
  & \\
\hline
\multicolumn{2}{c|}{\rule{0pt}{1.5em}\textbf{Condition on $K$}} 
    & None
    & \cellcolor{lightblue}\makecell[c]{None for the upper bound;\\
    $K\gtrsim\log n$ for the lower bound}
    & \cellcolor{lightblue}{$K \gtrsim \log n$} 
    & \makecell[c]{$K\gtrsim \varepsilon^{-4}\log n$} \\[0.5em]
\hline
\end{tabular}%
}
\label{tab:comparison}
\end{table}

First, for equal-weight AJIVE with deterministic loadings, we derive a sharper upper bound that explicitly tracks the joint--individual loading interactions. We also prove an algorithm-specific lower bound under a majority sign-alignment condition in the rank-one setting, showing that a second-order bias that does not vanish with $K$ can persist. Together, these results identify when the non-diminishing term is generated by the AJIVE algorithm rather than by a loose analysis.

Second, under sign-symmetric random loadings, the second-order interaction averages across views, yielding a $K^{-1/2}$-type rate. When the individual subspaces are also random, our rate matches the minimax lower bound up to logarithmic factors, which refines the recent result in \cite{tang2025mode} in an overlapping regime summarized in \Cref{tab:comparison}.

Third, we extend the analysis to general weights and obtain non-asymptotic bounds that separate statistical heterogeneity from structural heterogeneity. Based on these bounds, we derive an explicit weight, which gives optimality under a constant identifiability gap condition. When individual components are absent, this weight agrees to the recent optimality results for SVD-Stack established by \cite{baharav2025stacked}.

Finally, we provide a one-shot data-driven plug-in implementation of \hjive, together with an optional geometry-adaptive extension of this data-driven procedure. The resulting method is computationally simple and empirically effective. Extensive simulations and real-data analyses on TCGA-BRCA multi-omics and Caltech101-20 image data illustrate the practical benefits of \hjive.

\subsection{Related Work}
Our work is closely related to recent developments in AJIVE, heterogeneous PCA, and distributed/federated PCA. Regarding statistical guarantees for JIVE, \cite{ma2024optimal} analyze Stack-SVD for the special case of $K=2$. However, their result requires prior knowledge of the column indices of $\U$ within the singular vectors of the stacked signal matrix. In a broader setting, \cite{baharav2025stacked} systematically compare Stack-SVD, SVD-Stack, and their weighted variants, though their asymptotic analysis assumes the absence of individual components (i.e., $\U_k\W_k^\top = 0$). The most relevant work to ours is \cite{yang2025estimating}, which analyzes AJIVE under a homogeneous setting. They show that AJIVE attains an $O(K^{-1/2})$ rate in a high-SNR regime, but the error bound contains a non-vanishing term as $K$ increases when the SNR is low.  In parallel, \cite{tang2025mode} study a multi-view spiked covariance model (where $d_k = d$ and $\W_k = \W$) and propose the MOP-UP algorithm, which combines an AJIVE-type initialization with alternating projections. In the overlapping regimes, our analysis shows that an AJIVE-type initialization can already attain a $K^{-1/2}$-type rate matching the  minimax lower bound up to logarithmic factors.

Beyond the JIVE model, a large multi-view literature also studies the extraction of shared information from multiple data sets. One strand is based on cross-view dependence, including classical canonical correlation analysis (CCA) and partial least squares (PLS) estimate low-dimensional representations through cross-view correlation or covariance \citep{Hotelling1936,WoldSjostromEriksson2001}, and recent high-dimensional work has characterized detection and recovery thresholds in noisy multi-view or correlated spike models \citep{BykhovskayaGorin2023,KeupZdeborova2025,MergnyZdeborova2025}. These works clarify when a common signal can be recovered from multiple noisy views, but their
targets are correlation- or covariance-based latent directions rather than a decomposition into joint and view-specific low-rank components. A second, more directly related, strand is decomposition-based methods, such as D-CCA, D-GCCA, and SLIDE,  which are able to separate common, distinctive, or partially shared variation across data sets \citep{ShuWangZhu2020,ShuQuZhu2022,gaynanova2019structural}. These methods are close in modeling motivation, but they do not directly address simultaneous noise heterogeneity and view-specific subspace interference. Finally, nonlinear fusion methods based on alternating diffusion and common-manifold learning extract common variables or shared geometry by combining diffusion operators across modalities \citep{TalmonWu2019,KatzTalmonLoWu2019}. In this literature, shared structure is encoded by the geometry of a common latent manifold, rather than by an additive low-rank component. This flexibility makes the framework well suited for nonlinear data integration, but it also makes it less transparent how to obtain explicit finite-sample error rates for the recovered common structure, or how the accuracy should improve as information is pooled across multiple views.

A related literature studies PCA under heterogeneity. One line of work mainly models statistical heterogeneity \citep{oba2007heterogeneous, zhang2022heteroskedastic, arroyo2021inference, hong2023optimally, yan2024inference}, while another focuses on heterogeneous individual components \citep{lock2013joint, zhou2015group, shi2024personalized}. Most of these methods encode noise heterogeneity at the modeling level and then demonstrate a certain robustness to heterogeneity, either theoretically or empirically, but do not explicitly adjust the estimation procedure to exploit it. This is likely to be suboptimal, especially when a few high-SNR views carry most of the information.  Two exceptions are \cite{oba2007heterogeneous, hong2023optimally}, both of which propose reweighting an empirical loss to account for heteroskedastic noise and derive optimal weights in their settings. However, they only handle the heterogeneity in noise, and how to extend their weighting schemes to two-layer heterogeneity remains unclear.

Finally, we note that our proposed \hjive{} can be computed in a distributed fashion across local machines: each site computes its own local low-rank approximation and sends only a low-dimensional summary to a central server, where the weighted aggregation step is performed. This architecture is closely aligned with the literature on distributed PCA \citep{liang2014improved, fan2019distributed, zheng2022limit} and federated PCA \citep{grammenos2020federated,li2024federated}, and makes \hjive{} suitable for large-scale or privacy-constrained multi-view data where raw samples cannot be pooled.

\subsection{Notation}
For any positive integer $n$, we denote the set of integers $\{1, \dots, n\}$ by $[n]$. We use bold-face capital letters (e.g. $\A,\E$) to denote matrices, bold-face lower-case letters (e.g. $\x, \y$) for vectors. We denote by $\fro{\cdot}$ the Frobenius norm  and $\op{\cdot}$ the operator norm of matrices. We denote by $\lambda_i(\A)$ the $i$-th largest singular value of $\A$. 
When $n=r$, we simply write $\mathbb{O}_n$ for the orthogonal group. For two real numbers $a, b$, we use $a \wedge b = \min(a, b)$ and $a \vee b = \max(a, b)$.  
Finally, let $c, C, C_1, \dots$ denote unspecified positive constants whose values may change from line to line. We write $a\lesssim b$, $b\gtrsim a$ or $a=O(b)$ if $a\le C b$ for a universal constant $C>0$ and $a\asymp b$ if $a\lesssim b$ and $b\lesssim a$ hold simultaneously. Additionally, we use $\tilde{O}(\cdot)$ to denote order notation ignoring logarithmic factors.

\section{Preliminaries for JIVE}
\subsection{The JIVE Decomposition}
We consider the multi-view setting where data are collected as a collection of $K$ matrices $\{\A_{k}\}_{k=1}^{K}$ on a common set of $n$ units. We assume the data admit the following JIVE decomposition:
\begin{align}\label{eq:jive-model}
    	\A_k = \U\V_k^\top + \U_k\W_k^\top + \E_k\in\RR^{n\times d_k},\qquad k\in[K].
\end{align}
Here $\U\in\OO_{n,r}$ and $\U_k\in \OO_{n,r_k}$ are  the joint and individual subspaces with $\U^\top\U_k =0$ for all $k\in[K]$  following \cite{lock2013joint,yang2025estimating}. This orthogonality constraint entails no loss of generality, as any decomposition in \eqref{eq:jive-model} can be equivalently reparameterized to enforce orthogonality. The matrices $\V_{k}\in\mathbb{R}^{d_{k}\times r}$ and $\W_{k}\in\mathbb{R}^{d_{k}\times r_{k}}$ denote the view-specific loadings for the joint and individual components, respectively. We study both deterministic and random loading models. 
The random model allows us to isolate the identifiability difficulty arising from the overlap of the individual subspaces $\{\U_k\}_{k=1}^K$. Finally, we assume that $\E_1,\ldots,\E_K$ are independent across views, where $\E_k\in\RR^{n\times d_k}$ represents the additive noise containing i.i.d. $N(0,\sigma_k^2)$ entries and $\sigma_k$'s are heteroskedastic noise levels across views. 

\subsection{Angle-Based Estimation}
To estimate the joint subspace $\U$, we start from the angle-based JIVE (AJIVE) framework proposed by \cite{feng2018angle}. The procedure operates in two stages:
\begin{enumerate}
    \item \textbf{(Signal Extraction)} For each view $k$, perform an SVD on $\A_k$ and retain the top $r+r_k$ left singular vectors, denoted as $\tilde{\U}_k$. 
    \item \textbf{(Subspace Aggregation)} Estimate the joint subspace $\U$ by aggregating the projection matrices of these local subspaces as  $\sum_{k=1}^K\tilde{\U}_k\tilde{\U}_k^\top$.
\end{enumerate}
The original AJIVE algorithm \citep{feng2018angle} and recent theoretical analyses \citep{yang2025estimating} use equal-weight aggregation. We generalize the idea by defining the weighted AJIVE estimator $\hat\U(\w)$ as the top $r$ eigenvectors of $\sum_{k=1}^{K}w_k\tilde\U_k\tilde\U_k^{\top}$ for any fix nonnegative weight vector $\w=(w_1,\ldots,w_K)$ satisfying $\sum_{k=1}^K w_k=1$. We write its output simply as $\hat\U$ when the weight vector is clear. This estimator family, summarized in Algorithm~\ref{alg:weighted-ajive}, contains standard AJIVE ($w_k=1/K$) as a special case. The proposed \hjive{} estimator is the member obtained with the explicit weights $\w^\circ$ developed in \Cref{sec:weights}, and \Cref{alg:data-driven-hjive} gives its plug-in implementation.

A seemingly natural alternative is to stack the data and directly compute the top $r$ eigenvectors of the pooled covariance $\sum_{k=1}^{K} w_k \A_k \A_k^\top$. However, this approach assumes that individual components are uniformly weak compared to the joint component. When this assumption is violated, directions dominated by strong individual variation can be mistakenly identified as joint (see Example \ref{exmp:intro} or the discussion in \cite{yang2025estimating}). AJIVE circumvents this by decoupling subspace orientation from signal magnitude. By stripping singular values to normalize subspace scales, the algorithm aggregates via principal angles rather than variances.
This ensures the joint subspace estimation is robust to heterogeneous signal strengths (see Algorithm \ref{alg:weighted-ajive}).

\begin{algorithm}
	\caption{Weighted AJIVE with Fix Weights}
	\begin{algorithmic}
		\State{\textbf{Input:} $\{\A_k\}_{k=1}^K$, $r, \{r_k\}_{k=1}^K$, $\{w_k\}_{k=1}^K$}
		\For{$k = 1,\cdots,K$}
		\State{Let $\tilde\U_k$ be the top $r+r_k$ left singular vectors of $\A_k$.}
		\EndFor
            \State{Let $\hat \U(\w)$ be the top $r$ eigenvectors of $\sum_{k=1}^Kw_k\tilde\U_k\tilde\U_k^\top$.}
            \State{\textbf{Output:} $\hat\U(\w)$}
	\end{algorithmic}
	\label{alg:weighted-ajive}
\end{algorithm}

\section{\texorpdfstring{Error Bounds for Equal and General Weights}{Error Bounds for Equal and General Weights}}\label{sec:theory}
A primary challenge in deriving a sharp estimation error bound for $\hat \U$ stems from the algorithm's reliance on two nested SVD steps. Consequently, the final estimator $\hat \U$ is a highly nonlinear functional of the data, constructed from $\tilde \U_k$ which are themselves nonlinear. To rigorously control this compounded error, we apply the explicit formula for spectral projector developed in \cite{xia2021normal}.

\subsection{\texorpdfstring{Setup and Key Quantities}{Setup and Key Quantities}}
For each $k\in[K]$, we denote $\bar r_k:=r+r_k$ and set $\bar r:=\max_{k\in[K]}\bar r_k$. We also write $r_{\rm avg}:=K^{-1}\sum_{k=1}^K r_k$ for the average individual rank. The signal strength $\lambda_{k,\min}:= \lambda_{\bar r_k}(\U\V_k^\top + \U_k\W_k^\top)$ is the $\bar r_k$-th largest singular value.
In addition, we  define the condition number and the signal-to-noise ratio (SNR) for $k$-th view  as $\kappa_k:= \lambda_{k,\min}^{-1}\op{\U\V_k^\top + \U_k\W_k^\top}$ and  $\snr_k:= \lambda_{k,\min}/\sigma_k$, respectively.

Recall from single-view PCA that the unilateral subspace $\bar\U_k=\mat{\U&\U_k}$ has estimation error in operator norm of order \citep{cai2018rate}:
\begin{align}\label{eq:def-uk}
    \op{\tilde\U_k\tilde\U_k^\top - \bar\U_k\bar\U_k^\top}\lesssim \snr_k^{-1}\sqrt{n}+\snr_k^{-2}\sqrt{n d_k}=: \varepsilon_k. 
\end{align}
For the weighted AJIVE estimator, the joint subspace error will appear as a polynomial in $\{\varepsilon_k\}_{k=1}^K$. 
Within each view, leakage from individual to joint is governed by the angle between column spaces spanned by $\V_k$ and $\W_k$, measured by 
\begin{align}\label{deltak}
    \delta_k := \op{(\V_k^\top\V_k)^{-1/2}\V_k^\top \W_k(\W_k^\top\W_k)^{-1/2}}.
\end{align}
Notice that, when $\V_k$ and $\W_k$ have full column rank, $\delta_k<1$ is equivalent to $\lambda_{k,\min}>0$; this condition ensures $\text{span}(\U)\subset \cap_{k=1}^K\text{span}(\U\V_k^\top + \U_k\W_k^\top)$ for identifiability, which is called the faithfulness assumption in \cite{yang2025estimating}.
We will see how $\delta_k$ influences the final estimation.

When $\V_k$ and $\W_k$ are treated as random, the observed features can be viewed as realizations from a population driven by latent structures. This modeling choice treats the orientations of the view-specific projections as non-adversarial, thereby avoiding systematic sign alignment across views.
For this random scenario, we impose the following conditional random-loading assumption:
\begin{assumption}\label{assump:VW:random}
Let
$\mathcal H:=\sigma\bigl(\U,\{\U_j,\W_j\}_{j=1}^K\bigr)$.
Conditional on $\mathcal H$, the loading matrices $\V_1,\ldots,\V_K$ are independent across views, and for each $k\in[K]$,
\[
(\V_k\mid \mathcal H)\stackrel{d}{=}(-\V_k\mid \mathcal H).
\]
Moreover, for each $k\in[K]$, there exists $\lambda_{k,\min}>0$ such that, with probability exceeding $1-p_k$,
the smallest singular value of $\mat{\V_k& \W_k}$ is lower bounded by $\lambda_{k,\min}$.
\end{assumption}

This assumption holds for a broad class of symmetric distributions, such as i.i.d. symmetric sub-Gaussian matrices, where random matrix theory guarantees sufficient signal strength via the lower bound on the smallest singular value. 

\subsection{Equal Weights: Deterministic Loadings}\label{subsec:equal-weights}
We begin with the equal weights case. Because the joint subspace is present in every population signal space, identifiability is governed by the extent to which the individual subspaces overlap across views.
The misalignment is measured by 
\begin{align}\label{def:theta:equal}
    \theta:= 1 - \op{\frac{1}{K}\sum_{k=1}^K \U_k\U_k^\top}. 
\end{align}
In general, $\theta\in[0,1]$. If at least one individual rank is positive, then $\theta\leq1-1/K$; if all individual ranks are zero, then $\theta=1$. The quantity $\theta$ quantifies the difficulty recovering $\U$ with the presence of individual components. When $\theta =0$, 
 $\bigcap_{k=1}^K\text{span}(\U_k)$ is non-empty and the joint subspace $\text{span}(\U)$ is not identifiable. When all individual ranks are positive, the endpoint $\theta=1-1/K$ is attained when the individual subspaces $\ebrac{\U_k}_{k=1}^K$ are mutually orthogonal. Intuitively, separating joint from individual subspaces becomes harder as $\theta$ approaches $0$.

For the equal-weight case, we consider a homogeneous setting in which all views \(\A_k\) have comparable signal-to-noise ratios and dimensions. By \textit{homogeneous}, we mean that the \(\snr_k\) across the views are comparable, denoted as \(\snr_k \asymp \snr\), and that the second dimensions of each matrix \(\A_k\) satisfy \(d_k \asymp d\), for all \(k\).  As a result, we have $\varepsilon_k\asymp \varepsilon$, where $\varepsilon: = n^{1/2}\snr^{-1} +(nd)^{1/2}\snr^{-2}$.
However, we make no assumption on the scale of $\delta_k$.
The following result offers a natural baseline for our algorithm in the equal-weight case with \(w_k = 1/K\). 
\begin{theorem}\label{thm:equal-weight-deterministic}
Assume, for a sufficiently small absolute constant $c_0>0$ and a sufficiently large absolute constant $C_0>0$,
\begin{align*}
    \kappa_k\snr^{-1}n^{1/2}+\snr^{-2}(nd)^{1/2}\le  c_0, \quad \forall k\in[K],
    \qquad
    \theta \ge C_0\brac{\varepsilon^2+\sqrt{\frac{\log n}{K}}\varepsilon}.
\end{align*} Then we have with probability exceeding $1-O\bbrac{Ke^{-n\wedge d}}-O\bbrac{n^{-10}}$, 
the output of Algorithm \ref{alg:weighted-ajive} with the choice $w_k = 1/K$ for $k\in[K]$ satisfies
\begin{align*}
    \op{\hat\U\hat\U^\top - \U\U^\top}\lesssim  \varepsilon\sqrt{\frac{\bar r\log n}{K}(n^{-1}\theta^{-1}r_{\rm avg}+1)}+\varepsilon^2\theta^{-1}\bigg(\sqrt{\frac{\log n}{K}}+\frac{\sum_{k=1}^{K}\big(\delta_k(1-\delta_k^2)^{-1}\wedge 1\big)}{K} \bigg),
\end{align*}
where $r_{\rm avg}$ is defined above.
\end{theorem}

\paragraph{Connection with \cite{yang2025estimating}.}
For comparison, we first introduce the main results in \cite[Theorem 1]{yang2025estimating}.
For simplicity and to facilitate a clearer comparison, we focus on the setting where $n \asymp d$ and $r, r_k \asymp 1$.
In particular, they consider exact homogeneous noise case where $\sigma_k=\sigma$ for all $k$.
Under this setting, the estimation error is a polynomial of $\varepsilon \asymp \sqrt{n}\snr^{-1}$. 
They establish the following bound for $\bop{\hat\U\hat\U^\top - \U\U^\top}$:
\begin{align}\label{eq:yang-bound}
\varepsilon\sqrt{\frac{1}{K} + \frac{1}{Kn\theta}} \log^{5/2}n + \varepsilon^{2}\bigg(\frac{1}{\theta}+\frac{1}{K\theta^2}\bigg)\log^{5/2}n
\end{align}
Moreover, they established the minimax lower bound of order
\begin{align}\label{lower-bound}
     \varepsilon\sqrt{\frac{1}{K} + \frac{1}{Kn\theta}}.
\end{align}
In contrast, \Cref{thm:equal-weight-deterministic} provides the following bound for $\bop{\hat\U\hat\U^\top - \U\U^\top}$ under the same assumption on SNR and $\theta$:  
\begin{align}\label{eq:our-bound}
   \underbrace{\varepsilon \sqrt{\frac{1}{K} + \frac{1}{Kn\theta}} \log^{1/2} n}_{\textrm{first-order}} \hspace{0.1cm}+\hspace{0.1cm}\underbrace{\varepsilon^2\sqrt{\frac{\log n}{K\theta^2}} \hspace{0.1cm}+\hspace{0.1cm}\varepsilon^2\cdot \frac{\sum_{k=1}^{K}\big(\delta_k(1-\delta_k^2)^{-1}\wedge 1\big)}{K\theta}}_{\textrm{second-order}}.
\end{align}
In light of \eqref{eq:yang-bound} and \eqref{eq:our-bound}, our result offers a sharper upper bound.
For the first-order term, both results match the minimax lower bound \eqref{lower-bound} up to logarithmic factors, with our bound achieving a slightly tighter logarithmic dependence.
More importantly, the improvement lies in the second-order term.
Their result has a term $\theta^{-1}\log^{5/2} n\cdot \varepsilon^2$ that does not vanish with $K$, as their analysis does not explicitly track the loading-angle parameter between $\V_k$ and $\W_k$. In contrast, our bound incorporates a specific dependence on $\delta_k$, which can be viewed as a more refined characterization of the high-order terms. In \Cref{sec:rate-comparison} we further verify that, these terms are no larger than the corresponding terms in \eqref{eq:yang-bound} and have smaller logarithmic factors  under the same perturbative conditions.
A special case of particular interest occurs when $\delta_k$ are small (e.g., the orthogonal case when $\V_k^\top\W_k = 0$), such that $\sum_{k=1}^{K} \delta_k(1-\delta_k^2)^{-1}\lesssim K^{1/2}$. In this case, the second-order term reduces to $(K\theta^2)^{-1/2}\log^{1/2}n\cdot \varepsilon^2$, which vanishes as $K$ goes to infinity.
Finally, in the supplementary material we provide the algebraic comparison in the general $n,d$ setting, where additional cross-terms make a direct visual comparison less transparent.

\subsection{\texorpdfstring{A Lower Bound for AJIVE}{A Lower Bound for Equal-Weight AJIVE}}\label{subsec:rank-one-limitation}

The preceding comparison shows that the deterministic-loading upper bound contains a second-order term depending on the joint-individual loading angles $\{\delta_k\}_{k=1}^K$. Unlike the first-order term, this component need not average out over views when the deterministic loading interactions are aligned in sign. This raises a natural question: 

\emph{Is the non-vanishing second-order term merely an artifact of the proof, or is it intrinsic to AJIVE itself?}

We answer this question by proving a complementary algorithm-specific lower bound. 
To fix ideas, we consider a rank-one homogeneous setting: $\A_k=\A_k^*+\E_k$ with $$\A_k^*=\u\v_k^\top+\u_k\w_k^\top$$ for $k\in[K],$ where $\u$ is the joint direction, $\u_k$ is the $k$-th individual direction, $\op{\u}=\op{\u_k}=1$, and $\u^\top\u_k=0$. The entries of $\E_k$ are independent $N(0,\sigma^2)$ variables, all signal quantities are deterministic, equal weights are used throughout, and {we use the signal-strength notation $\lambda_{k,\min}$ from the general setup.} In this lower-bound subsection, we impose the exact homogeneous normalization $\lambda_{k,\min}=\lambda$ for all $k\in[K]$. Hence $\snr_k=\lambda/\sigma=:\snr$. In addition, we assume $d_k\asymp n$. Under this homogeneous scaling, the single-view perturbation size satisfies $\varepsilon_k\asymp\varepsilon$, where $\varepsilon:=\sqrt n\,\sigma/\lambda$.

We impose two simple deterministic conditions: one controls the loading angle and scale within each view, and the other requires majority alignment of the loading interactions across views.

\begin{assumption}[Loading angle and scale]\label{ass:lb-loading}
There exist fixed constants $\delta\in(0,1)$ and $C\geq1$ such that, for every $k\in[K]$,
\begin{equation*}
    \delta_k:=\frac{\ab{\v_k^\top\w_k}}{\op{\v_k}\op{\w_k}}=\delta,
    \qquad
    C^{-1}\leq \frac{\op{\v_k}}{\op{\w_k}}\leq C.
\end{equation*}
\end{assumption}
Assumption~\ref{ass:lb-loading} fixes the loading angle $\delta$ at a positive constant.
{As reflected in the second-order term of \Cref{thm:equal-weight-deterministic}, this permits a geometry contribution of order $(n\sigma^2\lambda^{-2})=\varepsilon^2$.} Assumption~\ref{ass:lb-majority}
below is what prevents this term from cancelling across views.

\begin{assumption}[Majority sign alignment]\label{ass:lb-majority}
Let $\z\in \u^\perp$ be a unit top eigenvector of $K^{-1}\sum_{k=1}^{K}\u_k\u_k^\top$:
\begin{equation*}
    \frac1K\sum_{k=1}^{K}\u_k\u_k^\top\z=(1-\theta)\z,
    \qquad
    \op{\z}=1.
\end{equation*}
Orient each pair $(\u_k,\w_k)$ so that $a_k:=\inp{\u_k}{\z}\geq0$. Define
\begin{equation*}
    s_k:=\operatorname{sign}(\v_k^\top\w_k),
    \qquad
    m_k:=\frac{a_k}{\op{\v_k}\op{\w_k}}.
\end{equation*}
There exist $\tau\in\{-1,1\}$ and $\gamma\in(0,1/2]$ such that
\begin{equation}
    \sum_{k:\tau s_k=1}m_k
    \geq
    \brac{\frac12+\gamma}\sum_{k=1}^K m_k.
    \label{eq:lb-majority}
\end{equation}
\end{assumption}
Assumption~\ref{ass:lb-majority} says that, after orienting the individual directions toward the dominant nuisance direction $\z$, a strict majority of the effective confounding mass has the same loading sign. 
Thus, it measures how strongly view $k$ can reinforce the aligned second-order bias. If the $m_k$'s are roughly comparable, the assumption reduces to the simple statement that more than half of the views share the same sign of $\v_k^\top\w_k$.

Recall that $\tilde\U_k$ denotes the rank-two left singular subspace estimator from $\A_k$, and let $\hat\u$ be the top eigenvector of $K^{-1}\sum_{k=1}^K\tilde\U_k\tilde\U_k^\top$.
The resulting lower bound is as follows.
\begin{proposition}[Algorithmic lower bound of AJIVE under majority sign alignment]\label{prop:lb-ajive}
Suppose the rank-one homogeneous setting above holds, and suppose Assumptions \ref{ass:lb-loading}--\ref{ass:lb-majority} hold.
Assume further that $n$ is sufficiently large and that, for some fixed
$c_0\in(0,1/2)$, the following conditions hold:
\begin{equation}
    c_0\leq\theta\leq1-c_0,
    \qquad
    K\geq C_1\log n,
    \qquad
    \varepsilon\sqrt{\frac{\log n}{K}}+\varepsilon^2\leq c_1\theta,
    \qquad
    K\exp\{-c_2n\}+n^{-10}\leq c_1\varepsilon^2.
    \label{eq:lb-smallness}
\end{equation}
Here $c_2>0$ is universal, whereas $C_1>0$ is chosen sufficiently
large and $c_1>0$ sufficiently small, depending only on the fixed constants
$c_0,\delta,\gamma$ and $C$ in Assumptions~\ref{ass:lb-loading}--\ref{ass:lb-majority}.
Then there exists a constant $c>0$, depending only on the fixed constants $c_0,\delta,\gamma$ and $C$, such that
\begin{equation*}
    \EE\op{\hat\u\hat\u^\top-\u\u^\top}
    \geq c\varepsilon^2.
\end{equation*}
\end{proposition}
The assumptions in Proposition~\ref{prop:lb-ajive} deliberately focus on a simple homogeneous regime with fixed $\delta$ and $c_0 \le \theta \le 1-c_0$. 
Fixing $\delta$ away from zero excludes the trivial case where the joint-individual loading interaction disappears, while keeping $\theta$ bounded away from zero places the problem in a stable identifiable regime. Therefore, the lower bound is not caused by a vanishing loading angle or by deteriorating identifiability. Rather, Proposition~\ref{prop:lb-ajive} shows that, even in this favorable setting, the bias of the expected AJIVE output projection has a non-vanishing lower bound of order $\varepsilon^2$. This indicates that the non-vanishing second-order term is generated by the equal-weight AJIVE map on the deterministic subclass satisfying \Cref{ass:lb-majority}, rather than being an artifact of our analysis. 

\paragraph{Relation to the oracle-aided lower bound of \cite{yang2025estimating}.}
This algorithmic lower bound is complementary to, and should be distinguished from, the oracle-aided lower-bound construction in \cite{yang2025estimating}. Their oracle spectral estimator uses side information about the target joint space when estimating the individual components, and their analysis shows that a non-diminishing second-order term can persist for that oracle-aided estimator. Proposition~\ref{prop:lb-ajive} addresses a different, algorithm-specific question: for the actual equal-weight AJIVE output, the expectation of the output projection itself can retain an $\varepsilon^2$-level bias under a majority sign-alignment condition in rank-one setting. Thus the two lower-bound results support the same qualitative message from different angles.

\subsection{\texorpdfstring{Equal Weights: Random Loadings}{Equal Weights: Sign-Symmetric Random Loadings}}\label{subsec:equal-weight-random}

The preceding subsection isolates a deterministic regime with majority sign alignment in which the second-order bias can align across views. We now consider the complementary random-design case, which yields a vanishing bound in terms of $K$.

\begin{theorem}\label{thm:equal-weight-random}
Suppose Assumption \ref{assump:VW:random} holds and, for a sufficiently small absolute constant $c_0>0$ and a sufficiently large absolute constant $C_0>0$,
\begin{align*}
    \kappa_k\snr^{-1}n^{1/2}+\snr^{-2}(nd)^{1/2}\le  c_0, \quad \forall k\in[K],
    \qquad
    \theta \ge C_0\brac{\varepsilon^2+\varepsilon\sqrt{\frac{\log n}{K}}}.
\end{align*} Then we have with probability exceeding $1-O\bbrac{Ke^{-n\wedge d}}-\sum_{k=1}^Kp_k-O\bbrac{n^{-10}}$, 
the output of Algorithm \ref{alg:weighted-ajive} with the choice $w_k = 1/K$ for $k\in[K]$ satisfies
\begin{align*}
    \op{\hat\U\hat\U^\top - \U\U^\top} 
    \lesssim \varepsilon\sqrt{\frac{\bar r \log n}{K}(n^{-1}\theta^{-1}r_{\rm avg}+1)}+\varepsilon^2\theta^{-1}\sqrt{\frac{\log n}{K}},
\end{align*}
where $r_{\rm avg}$ is defined above.
\end{theorem}

When $\theta$ is bounded away from zero and the ranks are constants, the bound in \Cref{thm:equal-weight-random} decreases at the $K^{-1/2}$ and essentially matches the minimax lower bound \eqref{lower-bound} established in \cite{yang2025estimating} up to  logarithmic factor. In this regime, the estimation error for the joint subspace is comparable to the scenario where no individual components exist. This suggests that the presence of individual components does not impose additional difficulty in recovering the joint subspace.
We now discuss a special case to illustrate this phenomenon.

\begin{proposition}\label{prop:random-Uk}
Under the setting of Theorem \ref{thm:equal-weight-random}, for each $k$, let $\U_k$ be an independent random matrix sampled uniformly from the set of orthogonal matrices of dimension $n\times r_k$ contained entirely within $\text{span}(\U)^{\perp}$. 
Also assume $\varepsilon_k\asymp \varepsilon$, $n\geq 4r_{\rm avg}+r$, $K\geq C\log n$. Then with the choice $w_k = 1/K$ for $k\in[K]$, we have with probability exceeding $1-O\bbrac{Ke^{-n\wedge d}}-\sum_{k=1}^Kp_k-O\bbrac{n^{-10}}$,
\begin{align*}
    \op{\hat\U\hat\U^\top - \U\U^\top} &\lesssim \varepsilon\sqrt{\frac{\bar r\log n}{K}}.
\end{align*}
\end{proposition}
Thus, when the random individual subspaces keep the misalignment bounded away from zero, equal-weight AJIVE is sharp up to logarithmic factors and attains the same rate as if the individual components were absent.
\paragraph{Connection with \cite{tang2025mode}.}
Recently, \cite{tang2025mode} proposed a framework for a multi-view spiked covariance model, which can be viewed as a special case of JIVE model \eqref{eq:jive-model} with 
$\W$ also shared across views (implying a common individual rank across views and $d_k = d$).
While their data generating process and ours are different and neither strictly subsumes the other, we can compare the theoretical guarantees within the intersection of ours and theirs. In particular, we consider 
$$\A_k = \U\V_k^\top + \U_k\W^\top + \E_k.$$
Here, we assume $\ebrac{\U_k}_{k=1}^K$ are generated under the setting of Proposition \ref{prop:random-Uk}, and $\V_k$ are i.i.d.  such that Assumption \ref{assump:VW:random} is satisfied. We further consider $\sigma_k=\sigma$ for $k\in[K]$,  $n\asymp d$  and $r,r_{\rm avg}\asymp 1$ for simplicity.

In this context, the {MOP-UP} algorithm in \cite{tang2025mode} consists of two stages. The initialization stage is equivalent to \Cref{alg:weighted-ajive} with equal weights (i.e., AJIVE) and they obtain a bound of $O(\varepsilon^2)$. They subsequently refine this initial estimator via alternating projection to derive a final bound of order $(K^{-1}\log n)^{1/2}$, provided that $K \gtrsim \varepsilon^{-4}\log n$.
In comparison, our upper bound is of order $\varepsilon\cdot(K^{-1}\log n)^{1/2}$, requiring only $K \gtrsim \log n$. Thus, the condition on $K$ is milder; whenever their condition holds, our bound is also smaller than their initialization error. Moreover, unlike their refined bound, our bound vanishes as either $K\to \infty$ or $\snr \to \infty$ and matches the minimax lower bound up to logarithmic factors. This suggests that in this scenario iterative refinement may be theoretically redundant.

\subsection{\texorpdfstring{General Weights}{General Weights}}\label{subsec:general-weights}

The equal-weight choice analyzed in \Cref{subsec:equal-weights} yields a clean bound in the homogeneous regime. However, in many applications, the view dimensions $d_k$ and the signal-to-noise ratios $\snr_k$ may vary substantially across $k$, making equal weighting potentially inefficient. Moreover, equal weights need not be optimal even under homogeneity.

A key difference from the equal weight analysis is the notion of misalignment under general weights. For $\w=(w_1,\ldots,w_K)\in\Delta_K$ where $\Delta_K:=\big\{ \tilde\w \in \mathbb{R}^K : \tilde w_k \ge 0, \sum_{k=1}^K \tilde w_k = 1 \big\}$, we define
\begin{align*}
    \theta(\w):= 1- \op{\sum_{k=1}^{K} w_k\U_k\U_k^\top}. 
\end{align*}
The choice of $\w$ directly influences the alignment level of the individual subspaces, offering flexibility to increase $\theta(\w)$ by down-weighting closely aligned views. In the following, we will drop the dependence on $\w$ in the subscript when the context is clear. The following example highlights the importance of weight choice.
\begin{example}\label{exmp:hetero}
Let $r_k = 1$ for all $k\in[K]$, and $\u_1 = \cdots = \u_{K-1} = \v$, $\u_{K}$ is orthogonal to $\v$. Then with equal weights, we have $\theta(\w) = K^{-1}$. However, if we choose $w_1 =w_K= \frac{1}{2}$ and $w_2 = \cdots = w_{K-1} =0$, $\theta(\w) = \frac{1}{2}$. 
\end{example}

We  denote the eigen-decomposition of $\I - \sum_{k=1}^{K} w_k \bar\U_k\bar\U_k^\top = \U_{\perp}\bLa\U_{\perp}^\top$ and also recall $\varepsilon_k = n^{1/2}\lambda_{k,\min}^{-1}\sigma_k +(nd_k)^{1/2}\lambda_{k,\min}^{-2}\sigma_k^2$.
The following theorem quantifies how reweighting interacts with $\theta$ and can improve on equal weights.
\begin{theorem}\label{thm:adaptive-weights-deterministic}
Fix $\w\in\Delta_K$. 
Assume
    \begin{align}\label{theta:deterministic}
            &\kappa_k\snr_k^{-1}n^{1/2} + \snr_k^{-2}(nd_k)^{1/2}\le c_0, \quad\forall k\in[K], \notag\\
        &\theta\ge C_0\cdot\max\bigg\{\sum_{k=1}^{K} w_k\varepsilon_k^2, \bigg(\sum_{k=1}^{K}w_k^2\varepsilon_k^2\log n\bigg)^{1/2}\bigg\},
    \end{align}
for a sufficiently small absolute constant $c_0>0$ and a sufficiently large absolute constant $C_0>0$. Then we have with probability exceeding $1-O\bbrac{\sum_{k=1}^Ke^{-n\wedge d_k}}-O\bbrac{n^{-10}}$, 
\begin{align*}
    \op{\hat\U\hat\U^\top - \U\U^\top} &\lesssim \sqrt{\sum_{k=1}^{K}n^{-1}w_k^2 \big(\tr(\M_k)+\bar r_k\op{\M_k}\big)\varepsilon_k^{2}}\sqrt{\log (n)}\\
&\quad +\sqrt{\sum_{k=1}^{K}w_k^2\varepsilon_k^{4}\theta^{-2}}\sqrt{\log n}+\sum_{k=1}^{K} w_k\big(\delta_k(1-\delta_k^2)^{-1}\wedge 1\big)\varepsilon_k^2\theta^{-1},
\end{align*}
where $\M_k:= \bar\U_{k\perp}^\top\U_{\perp}\bLa^{-2}\U_{\perp}^\top\bar\U_{k\perp}$.
\end{theorem}
The condition in \Cref{thm:adaptive-weights-deterministic} depends both on the weighted alignment of the individual subspaces, through $\theta$, and on the view-specific perturbation levels $\{\varepsilon_k\}_{k=1}^K$. Hence, one can face regimes with small $\theta$ (well-aligned individual subspaces) but uneven view quality where some $\varepsilon_k$ are small and others are large. In such cases, equal weighting may fail the condition, as illustrated by \Cref{exmp:hetero}, while choosing proper weights can adjust $\theta$ to meet the condition \eqref{theta:deterministic} as well as improving the overall error bound.

The bound in \Cref{thm:adaptive-weights-deterministic} separates the two sources of heterogeneity. The term $\varepsilon_k$ captures the statistical heterogeneity induced by view-specific $\snr_k$ and $d_k$, while the matrix term $\M_k$ encapsulates the structural heterogeneity driven by the orientation of the individual subspace $\U_k$ relative to the null space $\U_\perp$. Unlike heterogeneous PCA, where optimal weights are often determined by statistical precision alone \citep{baharav2025stacked}, the exact minimizer here depends nonlinearly on the weights through $\theta(\w)$ and $\M_k(\w)$. Nevertheless, \Cref{sec:weights} shows that a simple geometry-free inverse-cost rule is constant-factor optimal under a constant-gap condition, and the geometry-adaptive iteration is only an optional refinement when that condition is doubtful.

Next, we consider the sign-symmetric random-loading setting of Assumption \ref{assump:VW:random}. 
\begin{theorem}\label{thm:adaptive-weights-random}
Fix $\w\in\Delta_K$.  Suppose Assumption \ref{assump:VW:random} holds, and
    \begin{align*}
           &\kappa_k\snr_k^{-1}n^{1/2} + \snr_k^{-2}(nd_k)^{1/2}\le c_0, \quad\forall k\in[K],\\
        &\theta\ge C_0\cdot\max\bigg\{\sum_{k=1}^{K} w_k\varepsilon_k^2, \bigg(\sum_{k=1}^{K}w_k^2\varepsilon_k^2\log n\bigg)^{1/2}\bigg\},
    \end{align*}
for a sufficiently small absolute constant $c_0>0$ and a sufficiently large absolute constant $C_0>0$. Then we have with probability exceeding $1-O\bbrac{\sum_{k=1}^Ke^{-n\wedge d_k}} - \sum_{k=1}^{K} p_k-O\bbrac{n^{-10}}$, 
\begin{align*}
    \op{\hat\U\hat\U^\top - \U\U^\top} &\lesssim\sqrt{\sum_{k=1}^{K}n^{-1}w_k^2 \big(\tr(\M_k)+\bar r_k\op{\M_k}\big)\varepsilon_k^{2}}\sqrt{\log n}+
\sqrt{\sum_{k=1}^{K}w_k^2\varepsilon_k^{4}\theta^{-2}}\sqrt{\log n}.
\end{align*}
\end{theorem}
As with the equal-weight case, conditional sign symmetry centers the deterministic loading-interaction term in \Cref{thm:adaptive-weights-deterministic}, leaving only the corresponding root-sum-of-squares fluctuation term. 

\section{\texorpdfstring{\hjive: Weight Selection and Data-Driven Implementation}{HeteroJIVE: Weight Selection and Data-Driven Implementation}}\label{sec:weights}

The bounds in \Cref{sec:theory} characterize the weighted AJIVE estimator for any fixed weight vector $\w=(w_1,\ldots,w_K)\in\Delta_K$, but they do not by themselves prescribe how the weights should be selected. We now complete the definition of \hjive{} by specifying an explicit inverse-cost rule. When these candidate weights preserve a constant identifiability gap, the rule is constant-factor optimal relative to the unrestricted oracle criterion. We use it as the default choice and then provide a data-driven plug-in implementation; when the estimated gap is small, we also describe an optional extension of the data-driven procedure.

\subsection{\texorpdfstring{Explicit Weights and Optimality}{HeteroJIVE Weights and Constant-Factor Optimality}}\label{subsec:explicit-weights}

We focus on the random-loading setting of \Cref{thm:adaptive-weights-random}, for which
the deterministic bias term involving $\{\delta_k\}_{k=1}^K$ is absent. Let $\mathcal D:=\{\w\in\Delta_K:\theta(\w)>0\}$. For $\w\in\mathcal D$, define
\begin{align}\label{eq:oracle-criterion}
J(\w)
:=\sum_{k=1}^K w_k^2
\big\{q_k(\w)\varepsilon_k^2+\theta(\w)^{-2}\varepsilon_k^4\big\},
\end{align}
where
$q_k(\w):=n^{-1}\{\tr(\M_k(\w))+\bar r_k\op{\M_k(\w)}\}$.
We set $J(\w):=+\infty$ on $\Delta_K\setminus\mathcal D$ and write
$J^\ast:=\inf_{\w\in\Delta_K}J(\w)$. $J(\w)^{1/2}$ is equivalent to the sum of the two nonnegative terms in the bound of \Cref{thm:adaptive-weights-random} up to logarithmic terms. Thus, minimizing $J$ directly targets the theorem's stochastic upper bound. The key observation is that  $J$ admits a quadratic
envelope:
\begin{align*}
\Phi(\w)\leq J(\w)\leq\theta(\w)^{-2}\Phi(\w),
\qquad \w\in\mathcal D,
\end{align*}
where 
$\Phi(\w):=\sum_{k=1}^K\xi_k w_k^2$ and $\xi_k:=\varepsilon_k^2(1+\varepsilon_k^2)$. Therefore, the minimizer of $\Phi$ over the simplex admits the explicit inverse-cost form:
\begin{align}\label{eq:constant-gap-explicit-weight}
w_k^\circ
:=\frac{\xi_k^{-1}}{\sum_{j=1}^K\xi_j^{-1}},
\qquad k\in[K].
\end{align}
It worths noting that \eqref{eq:oracle-criterion} is tailored to the random-loading bound and need not be sharp for every deterministic loading configuration. However, the inverse-cost rule \eqref{eq:constant-gap-explicit-weight} remains a natural choice for the fluctuation terms if the signed second-order interactions partially or completely cancel, e.g., when \Cref{ass:lb-majority} fails in the rank-one setting.
\begin{proposition}
\label{prop:constant-gap-explicit-weight}
Assume $\varepsilon_k>0$ for all $k\in[K]$. {If $\theta(\w^\circ)>0$, then} ${J(\w^\circ)
\leq \theta(\w^\circ)^{-2}J^\ast.}$
{In particular, if $\theta(\w^\circ)\geq c_0$ for some absolute constant $c_0\in(0,1)$, then $J(\w^\circ)\leq c_0^{-2}J^\ast$, so $\w^\circ$ attains the same order as the unrestricted oracle infimum.}
\end{proposition}

The proof is deferred to \Cref{pf-prop:constant-gap-explicit-weight}. The fact that $\w^\circ$ depends only on $\{\varepsilon_k\}_{k=1}^K$ does not impose structural homogeneity. Rather, when $\theta(\w^\circ)$ is bounded away from zero, the geometry-dependent factors $q_k(\w^\circ)$ and $\theta(\w^\circ)^{-2}$ are uniformly bounded, so structural heterogeneity changes the criterion only by a constant factor and statistical heterogeneity determines the explicit rule. The constant-gap condition
$\theta(\w^\circ)\geq c_0$ is satisfied, for example, when the individual subspaces are sufficiently dispersed and the candidate weights do not
concentrate on a nearly aligned collection of views. It mainly excludes near-identifiability
settings in which a common individual direction receives almost all of the candidate weight.

Accordingly, we define \hjive{} as the weighted AJIVE estimator in Algorithm~\ref{alg:weighted-ajive} with the explicit weights $\w^\circ$ in \eqref{eq:constant-gap-explicit-weight}.

\begin{corollary}[No individual components]\label{cor:no-individual-exact-weight}
If $r_k=0$ for every $k\in[K]$ (so there are no individual components), then $\theta(\w)=1$ and $q_k(\w)=1$ for $\w\in\Delta_K$. Hence $\w^\circ$ defined in
\eqref{eq:constant-gap-explicit-weight} is the exact minimizer of $J$.
\end{corollary}
In the regime $\max_k\varepsilon_k=o(1)$, the exact no-individual-component weight
satisfies
$w_k^\circ
={\varepsilon_k^{-2}}\bbrac{\sum_{j=1}^K\varepsilon_j^{-2}}^{-1}\sqbrac{1+o(1)}$. This agrees to first order with the asymptotically optimal SVD-Stack weights of
\cite{baharav2025stacked} in the rank-one proportional regime $d_k/n\asymp1$. Their
result is asymptotic in that regime, whereas our error bounds for a fixed weight vector are
non-asymptotic and also allow $d_k$ to be much larger than $n$.

\subsection{One-shot \hjive}\label{subsec:one-shot-weights}

The explicit rule in \eqref{eq:constant-gap-explicit-weight} only requires estimates of the
single-view perturbation levels $\ebrac{\varepsilon_k}_{k=1}^K$. Since $\lambda_{k,\min}$
is the smallest nonzero singular value of the combined joint--individual signal in view
$k$, the quantities entering $\varepsilon_k$ can be estimated separately within each view.

For each $k\in[K]$, recall $\bar r_k=r+r_k$, let $\tilde\U_k$ be the top $\bar r_k$
left singular vectors of $\A_k$, and define the local rank-$\bar r_k$ approximation $\tilde\A_k:=\tilde\U_k\tilde\U_k^\top\A_k$. We estimate the smallest signal singular value and the noise level locally by
\begin{align*}
\hat\lambda_{k,\min}
&:=\lambda_{\bar r_k}\big(\tilde\A_k\big),
&
\hat\sigma_k
&:=\frac{\fro{\A_k-\tilde\A_k}}
{\sqrt{nd_k-\bar r_k(n+d_k-\bar r_k)}}.
\end{align*}
The corresponding local SNR and plug-in perturbation level are
\begin{align*}
\widehat{\snr}_{k}
&:={\hat\lambda_{k,\min}}/
{\hat\sigma_k},
&
\hat\varepsilon_k
&:=\sqrt{n}\widehat{\snr}_{k}^{-1}
+\sqrt{nd_k}\widehat{\snr}_{k}^{-2}.
\end{align*}
This construction estimates the quality of the combined signal space
$\operatorname{span}(\U,\U_k)$ and does not require separating its joint and individual
parts. The one-shot plug-in weights are then
\begin{align}\label{eq:plugin-explicit-weight}
\hat\xi_k
:=\hat\varepsilon_k^2(1+\hat\varepsilon_k^2),
\qquad
\hat w_k^\circ
:=\frac{\hat\xi_k^{-1}}{\sum_{j=1}^K\hat\xi_j^{-1}},
\qquad k\in[K].
\end{align}
All quantities used to form $\hat w_k^\circ$ are computed from $\A_k$ alone and can be
estimated in parallel across views. The local singular vectors $\tilde\U_k$ are then reused
in the weighted aggregation step. The resulting data-driven procedure is summarized below.

\begin{algorithm}[H]

\caption{One-shot \hjive}
\begin{algorithmic}
\State{\textbf{Input:} $\{\A_k\}_{k=1}^K$, $r$, $\{r_k\}_{k=1}^K$}
\For{$k=1,\ldots,K$}
\State{Compute the top $\bar r_k=r+r_k$ left singular vectors $\tilde\U_k$ of $\A_k$.}
\State{Compute $\hat\varepsilon_k$ from the local SNR estimate above.}
\EndFor
\State{Compute $\hat\w^\circ$ from \eqref{eq:plugin-explicit-weight}.}
\State{Let $\hat\U^{\rm H}$ be the top $r$ eigenvectors of
$\sum_{k=1}^K\hat w_k^\circ\tilde\U_k\tilde\U_k^\top$.}
\State{\textbf{Output:} $\hat\U^{\rm H}$ and $\hat\w^\circ$}
\end{algorithmic}
\label{alg:data-driven-hjive}
\end{algorithm}

\subsection{\texorpdfstring{Optional Refinement of One-shot \hjive{}}{Optional Extension of the Data-Driven HeteroJIVE Procedure}}\label{subsec:optional-refinement}

The one-shot \hjive{} is the default whenever its estimated
identifiability gap is bounded away from zero. The geometry needed to evaluate this gap is
constructed only after the local-SNR plug-in weights have been obtained. In particular, set
\begin{align*}
    \hat\U^{(0)}
    :=\hat\U(\hat\w^\circ)
    =\text{Eigen}_r\left(\sum_{k=1}^K
    \hat w_k^\circ\tilde\U_k\tilde\U_k^\top\right),
\end{align*}
which is exactly the data-driven \hjive{} estimate from
Algorithm~\ref{alg:data-driven-hjive}, rather than an equal-weight AJIVE estimate. For
$k\in[K]$, let $\hat\U_k^{(0)}$ be the top $r_k$ left singular vectors of $\big(\I-\hat\U^{(0)}\hat\U^{(0)\top}\big)\A_k$,
and define $\hat{\bar\U}_k^{(0)}:=\mat{\hat\U^{(0)}&\hat\U_k^{(0)}}$. For any $\w\in\Delta_K$, set $\hat\theta(\w)
:=1-\op{\sum_{k=1}^K w_k
\hat\U_k^{(0)}\hat\U_k^{(0)\top}}$. Whenever $\hat\theta(\w)>0$, let $\I-\sum_{j=1}^K w_j
\hat{\bar\U}_j^{(0)}\hat{\bar\U}_j^{(0)\top}
=\widehat{\U}_{\perp}(\w)\widehat{\bLa}(\w)
\widehat{\U}_{\perp}(\w)^\top$
be its rank-$(n-r)$ nonzero eigendecomposition. If
$\hat{\bar\U}_{k\perp}^{(0)}$ is an orthonormal basis for the orthogonal complement of
$\hat{\bar\U}_k^{(0)}$, define
\begin{align*}
\widehat{\M}_k(\w)
:=\hat{\bar\U}_{k\perp}^{(0)\top}
\widehat{\U}_{\perp}(\w)\widehat{\bLa}(\w)^{-2}
\widehat{\U}_{\perp}(\w)^\top
\hat{\bar\U}_{k\perp}^{(0)}.
\end{align*}
Thus, both the initial joint estimate and the estimated geometry are based on the
local-SNR plug-in weights $\hat\w^\circ$.
If the empirical gap $\hat\theta(\hat\w^\circ)$ is small, the variation of
$\theta(\w)$ and $\M_k(\w)$ may affect the oracle criterion beyond a constant factor. In
that case, one may optionally refine the explicit weights by keeping the estimated subspaces fixed
and evaluating the geometry-dependent costs at the current iterate. 

Starting from
$\w^{(0)}=\hat\w^\circ$, and provided the current estimated gap is positive, use
\begin{align*}
w_k^{(t+1)}
&:=\frac{\big[\hat c_k\{\w^{(t)}\}\big]^{-1}}
{\sum_{j=1}^K\big[\hat c_j\{\w^{(t)}\}\big]^{-1}},\\
\hat c_k(\w)
&:=\hat\varepsilon_k^4\hat\theta(\w)^{-2}
+n^{-1}\hat\varepsilon_k^2
\left[\tr\{\widehat{\M}_k(\w)\}+\bar r_k\op{\widehat{\M}_k(\w)}\right],
\qquad k\in[K].
\end{align*}
If an update produces a nonpositive estimated gap, the refinement is stopped at the last
well-defined iterate. After a prescribed number of updates, the last weight vector is used
in Algorithm~\ref{alg:weighted-ajive}; when no refinement is performed, the output
remains $\hat\U^{(0)}=\hat\U^{\rm H}$.

This geometry-adaptive frozen-cost heuristic is only an optional extension of the
data-driven \hjive{} procedure, rather than part of its default implementation. We do not
claim that the iteration converges or globally minimizes $J$. For completeness,
\Cref{app:oracle-refinement} records the oracle version of the iteration and shows that any
sufficiently regular interior fixed point has a small projected gradient in the perturbative
regime.

\section{Numerical Experiments}

We consider the following reparameterization of the JIVE model:
\begin{align*}
    \A_k
    =s_k\big(\U\V_k^\top+\gamma\U_k\W_k^\top\big)+\E_k,
    \qquad k\in[K],
\end{align*}
where the entries of $\E_k$ are independent $N(0,\sigma_k^2)$ random
variables.  Here $s_k$ controls the overall signal strength in view $k$,
while $\gamma$ controls the signal strength of the individual component
relative to the joint component.  Throughout this section, we set
$s=s_1=\cdots=s_K$, $r=r_1=\cdots=r_K$, and
$d=d_1=\cdots=d_K$, with $d>2r$.

To generate $\U$ and $\{\U_k\}_{k=1}^K$, we use a construction similar to
that of \citet{yang2025estimating}.  We first draw the joint subspace
$\U$ randomly from $\OO_{n,r}$.  We then set
\begin{align*}
    \U_k=\sqrt{1-\theta}\,\Z+\sqrt{\theta}\,\Z_k,
    \qquad k\in[K],
\end{align*}
where $\Z\in\RR^{n\times r}$ is a random orthonormal matrix in
$\operatorname{span}(\I_n-\U\U^\top)$ and
$\Z_k\in\RR^{n\times r}$ is a random orthonormal matrix in
$\operatorname{span}(\I_n-\U\U^\top-\Z\Z^\top)$.  For the loading
matrices $\V_k$ and $\W_k$, we consider the following four schemes:
\begin{itemize}
    \item \emph{Random scheme.}
    Draw $\V_k$ and $\W_k$ independently from $\OO_{d,r}$ for each
    $k\in[K]$.

    \item \emph{Shared scheme.}
    Draw $\V$ and $\W$ independently from $\OO_{d,r}$ and set
    $\V_k=\V$ and $\W_k=\W$ for every $k\in[K]$.

    \item \emph{Shared-orthogonal scheme.}
    Draw $\Q=[\q_1,\ldots,\q_d]$ from $\OO_d$ and set
    $\V_k=[\q_1,\ldots,\q_r]$ and
    $\W_k=[\q_{r+1},\ldots,\q_{2r}]$ for every $k\in[K]$.

    \item \emph{Random-orthogonal scheme.}
    Independently draw
    $\Q_k=[\q_{k,1},\ldots,\q_{k,d}]$ from $\OO_d$ and set
    $\V_k=[\q_{k,1},\ldots,\q_{k,r}]$ and
    $\W_k=[\q_{k,r+1},\ldots,\q_{k,2r}]$ for each $k\in[K]$.
\end{itemize}

Unless stated otherwise, performance is measured by the operator-norm
distance between the estimated and true joint projection matrices $\big\|\hat\U\hat\U^\top-\U\U^\top\big\|$.
We evaluate $\sqrt K\in\{2,5,10,15,20,30,40,50\}$.  For every configuration and
every value of $K$, we use $100$ independent Monte Carlo replications.
The plotted curves are the replicate means, and the error bars extend
two Monte Carlo standard errors on either side of the mean.

\subsection{Homogeneous Setting}

We first consider a homogeneous setting with identical noise levels
across views.  We set $s=1$, $\sigma_1=\cdots=\sigma_K=0.1$, $r=2$, and
$n=d=20$.  We use AJIVE, namely weighted AJIVE with weights $1/K$, to
isolate the effect of loading geometry.

The left panel of \Cref{fig:homo-1} compares AJIVE under the four loading
schemes with $(\theta,\gamma)=(0.5,0.5)$.  The shared scheme behaves
qualitatively differently from the other three schemes.  Its average
error decreases from approximately $0.31$ at $\sqrt K=2$ to $0.086$ at
$\sqrt K=10$, but then remains around $0.065$--$0.076$ over the larger
values of $K$.  In contrast, under the random, shared-orthogonal, and
random-orthogonal schemes, the error continues to decrease at an
approximately $K^{-1/2}$ rate.  At $\sqrt K=50$, their average errors
range from $0.0116$ to $0.0138$.  Thus, the non-vanishing low-SNR
behavior is associated with the specific configuration in which the
loadings are shared and non-orthogonal; either randomness across views
or orthogonality between joint and individual loadings removes this
persistent error in the present simulations.

The right panel of \Cref{fig:homo-1} compares AJIVE with Stack-SVD under
the random loading scheme and $(\theta,\gamma)=(0.5,1.5)$.  Stack-SVD
horizontally concatenates the views and applies PCA to the concatenated
matrix.  AJIVE's average error decreases from $0.353$ at $\sqrt K=2$ to
$0.0152$ at $\sqrt K=50$.  By contrast, the Stack-SVD error remains
close to one throughout the grid.  This separation is consistent with
Stack-SVD's inability to disentangle joint and individual variation
when the individual signal is relatively strong.

\begin{figure}[!tbp]
    \centering
    \includegraphics[width=0.47\textwidth]{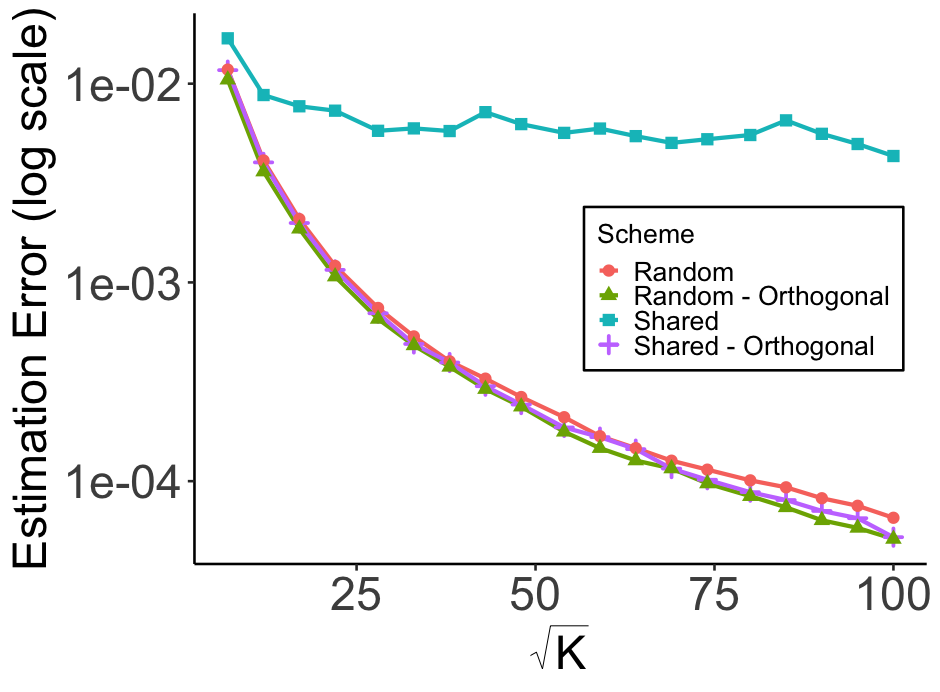}
    \includegraphics[width=0.45\textwidth]{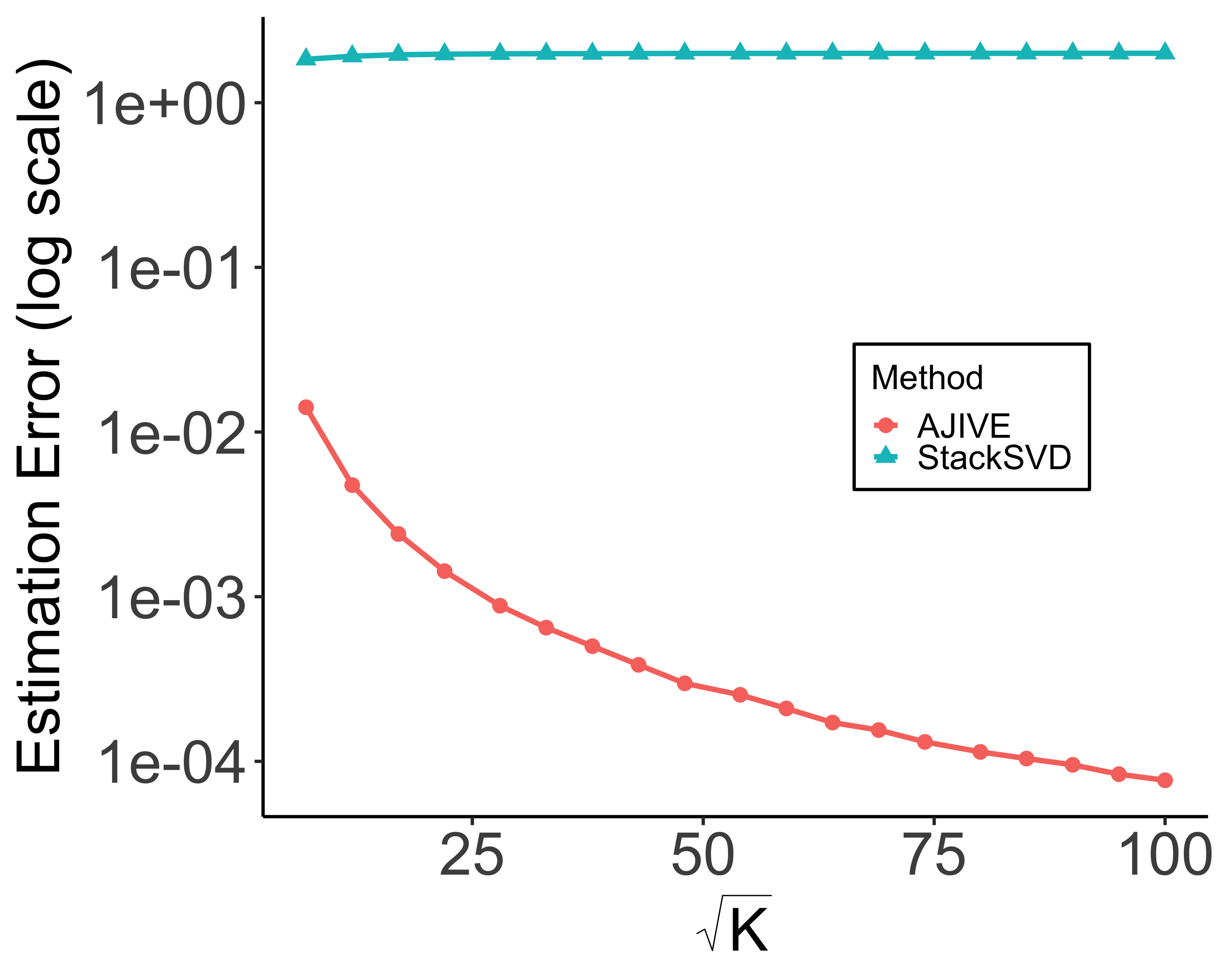}
    \hfill
    \caption{Homogeneous-noise simulations.  The vertical axis is the
    operator-norm error
    $\|\hat\U\hat\U^\top-\U\U^\top\|$ and the horizontal
    axis is $\sqrt K$.  Left: AJIVE under the four loading schemes with
    $(\sigma,\theta,\gamma)=(0.1,0.5,0.5)$.  Right: AJIVE and Stack-SVD
    under the random loading scheme with
    $(\sigma,\theta,\gamma)=(0.1,0.5,1.5)$.  Each point is the mean of
    $100$ independent replications; error bars show two Monte Carlo
    standard errors.}
    \label{fig:homo-1}
\end{figure}

\subsection{Heterogeneous Setting}

We next consider view-specific noise levels.  We set $s=10$,
$\gamma=2$, $\theta=0.6$, $r=3$, $n=40$, and $d=50$, and independently
draw $\sigma_k\sim\operatorname{Unif}(0.1,2)$ for $k\in[K]$ within each replication.  We consider the random and shared-orthogonal
loading schemes and compare four estimators: (i) Stack-SVD; (ii) AJIVE, using weights $1/K$; (iii) the default one-shot data-driven \hjive{} estimator in \Cref{alg:data-driven-hjive}; and (iv) the optional refined estimator, denoted \textnormal{\hjive{}-refined} in  \Cref{subsec:optional-refinement}.
For the refined estimator,
we use at most $25$ updates, stopping
early when successive weights differ by at most $10^{-6}$. If a
candidate update has a numerically nonpositive estimated gap, the last
well-defined iterate is retained.  

\Cref{fig:hetero-1} shows that the one-shot and refined \hjive{}
estimators outperform both AJIVE and Stack-SVD at every reported value
of $K$ under both loading schemes.  The Stack-SVD error remains close to
one, whereas the AJIVE and \hjive{} errors decrease as $K$ grows.
Down-weighting the noisier views yields a substantial improvement over
AJIVE\@.  For example, at $\sqrt K=50$, the mean errors of AJIVE,
one-shot \hjive{}, and refined \hjive{} are respectively $0.0215$,
$0.00738$, and $0.00742$ under the random scheme, and $0.0207$,
$0.00708$, and $0.00713$ under the shared-orthogonal scheme.

The one-shot and refined curves are generally very close.  The most
visible separation occurs at $\sqrt K=5$, after which the two curves
nearly overlap.  This behavior is consistent with the role of the
refinement as a geometry-adaptive sensitivity analysis rather than as
the default estimator.  

\begin{figure}[!ht]
    \centering
    \captionsetup{font=small,skip=4pt}
    \includegraphics[width=0.40\textwidth]{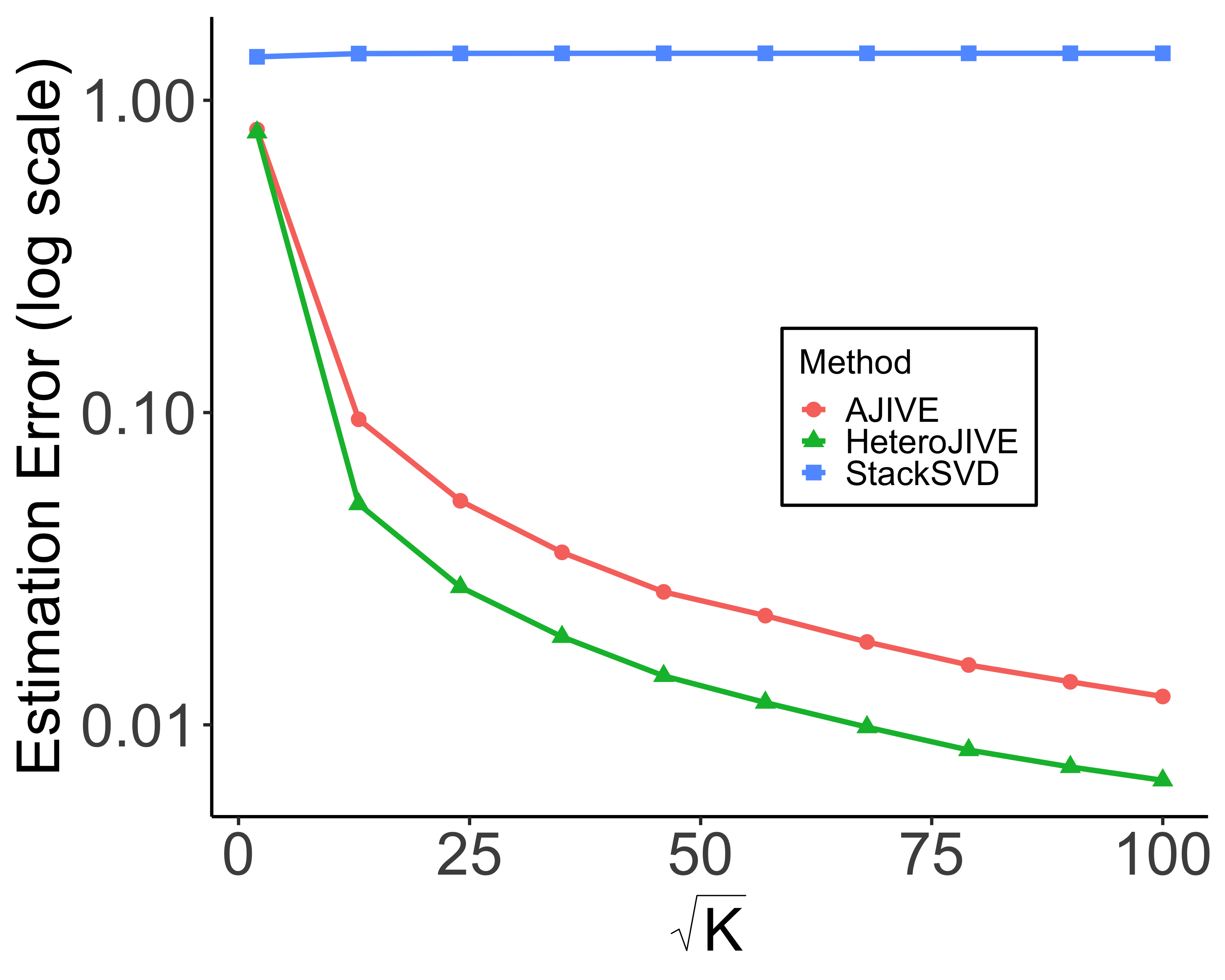}
    \hfill
    \includegraphics[width=0.40\textwidth]{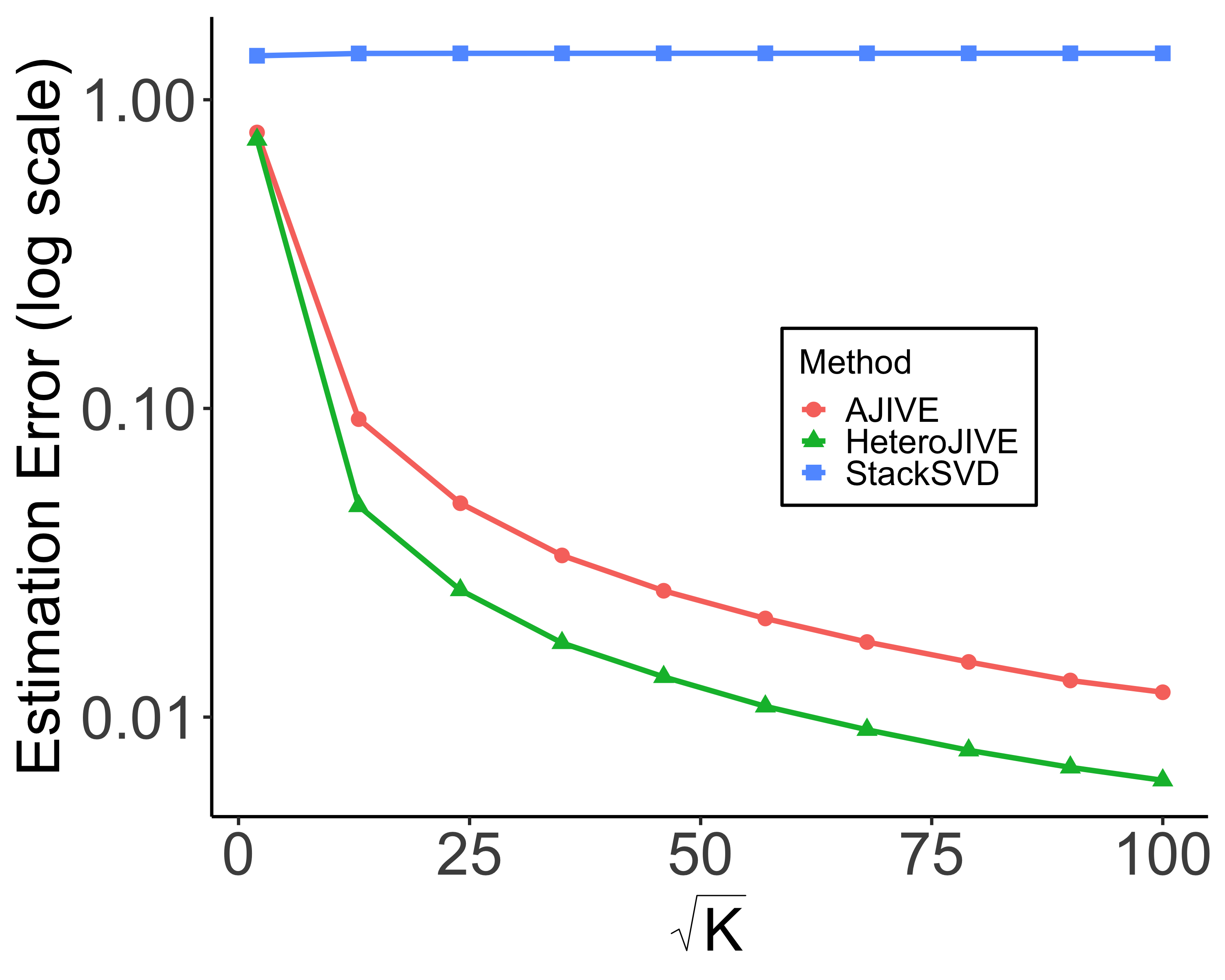}
    \caption{Heterogeneous-noise simulations comparing Stack-SVD,
    equal-weight AJIVE, the default one-shot data-driven \hjive{}, and
    its optional frozen-geometry refinement.  The vertical axis is the
    operator-norm error
    $\|\hat\U\hat\U^\top-\U\U^\top\|$ and the horizontal
    axis is $\sqrt K$.  The random loading scheme is shown on the left
    and the shared-orthogonal scheme on the right.  Here
    $\sigma_k\overset{\mathrm{i.i.d.}}{\sim}\operatorname{Unif}(0.1,2)$,
    $(s,\theta,\gamma)=(10,0.6,2)$, $r=3$, $n=40$, and $d=50$.
    Each point is the mean of $100$ independent replications; error bars
    show two Monte Carlo standard errors.}
    \label{fig:hetero-1}
\end{figure}

\section{Real-data Analysis}

\subsection{TCGA-BRCA multi-omics data}
We illustrate \hjive{} on the TCGA-BRCA multi-omics dataset with four views: gene expression (GE), DNA methylation (ME), miRNA expression (miRNA), and reverse phase protein array (RPPA). Following the preprocessing procedure in \cite{lock2013bayesian}, the four aligned matrices contain $n=348$ tumor samples and have dimensions $d_1=645$, $d_2=574$, $d_3=423$, and $d_4=171$.

We fit AJIVE and the data-driven one-shot \hjive{} estimator with $(r,r_1,r_2,r_3,r_4)=(3,31,28,27,19)$, using the same ranks and input matrices for both methods.   These ranks are determined by R package \texttt{r.jive} \citep{o2016r} using its default permutation-based procedure. This yields AJIVE estimates $(\hat\U^{\rm A},\{\hat\U_k^{\rm A}\}_{k=1}^K)$ and \hjive{} estimates $(\hat\U^{\rm H},\{\hat\U_k^{\rm H}\}_{k=1}^K)$ for the  joint and individual subspaces. The resulting weights for one-shot \hjive{} for GE, ME, miRNA, and RPPA are $(0.307, 0.202, 0.326, 0.164)$.  

We evaluate each estimated individual subspace against its corresponding TCGA cluster labels: PAM50 mRNA for GE, methylation clusters for ME, miRNA clusters for miRNA, and RPPA clusters for RPPA \citep{cancer2012comprehensive}. All 348 samples and all label categories are retained, and labels are used only for post-hoc evaluation. In addition, we calculate the SWISS score \citep{cabanski2010swiss}, which measures how well Euclidean distances in the low-dimensional representation align with the a priori labels (lower is better). Table~\ref{tab:tcga-individual} reports SWISS and label-matched clustering error. 

\begin{table}[H]
\centering
\caption{TCGA-BRCA post-hoc evaluation of the assay-specific individual subspaces. ``Overall'' is the unweighted mean over the four tasks.}
\label{tab:tcga-individual}
\small
\setlength{\tabcolsep}{5pt}
\renewcommand{\arraystretch}{1.08}
\begin{tabular}{@{}l c cc cc@{}}
\toprule
& & \multicolumn{2}{c}{SWISS} & \multicolumn{2}{c}{Clustering error} \\
\cmidrule(lr){3-4}\cmidrule(lr){5-6}
Task & Clusters & AJIVE & \makecell{One-shot\\\hjive{}} & AJIVE & \makecell{One-shot\\\hjive{}} \\
\midrule
GE / PAM50 mRNA & 5 & 0.9765 & \textbf{0.9764} & 0.7011 & \textbf{0.6839} \\
ME / methylation & 5 & \textbf{0.9561} & 0.9566 & 0.7040 & \textbf{0.6782} \\
miRNA / miRNA cluster & 7 & \textbf{0.9505} & 0.9507 & 0.7586 & \textbf{0.7557} \\
RPPA / RPPA cluster & 7 & 0.8921 & \textbf{0.8886} & 0.5718 & \textbf{0.4943} \\
\midrule
Overall & -- & 0.9438 & \textbf{0.9431} & 0.6839 & \textbf{0.6530} \\
\bottomrule
\end{tabular}
\end{table}

\subsection{Caltech101-20 multi-view image data}
As a second real-data example, we analyze the Caltech101-20 multi-view image dataset, which contains $n=2386$ images from 20 object classes. Each image is represented by six aligned views that capture complementary aspects of appearance: Gabor and wavelet-moment (WM) features describe frequency- and texture-related patterns; CENTRIST uses census-transform histograms to summarize local intensity comparisons; histograms of oriented gradients (HOG) encode edge orientations; GIST summarizes global spatial layout; and local binary patterns (LBP) represent local texture. Their dimensions are $d_1=48$, $d_2=40$, $d_3=254$, $d_4=1984$, $d_5=512$, and $d_6=928$, respectively. Within each view, we center and scale the columns and remove constant or near-constant features. Class labels are used only for post-hoc evaluation.

We apply the permutation criterion in \texttt{r.jive} with the explicit constraint $\bar r_k\leq 20$. The procedure selects $(r,r_1,\ldots,r_6)=(6,6,7,14,14,14,14)$. The same matrices and ranks are used for AJIVE and one-shot \hjive{}. The resulting weights for Gabor, WM, CENTRIST, HOG, GIST, and LBP are $(0.111, 0.061, 0.035, 0.360, 0.193, 0.240)$, respectively. Note that these weights should be interpreted as reliability for estimating the joint subspace. In particular, HOG, LBP, and GIST have the three smallest $\hat\varepsilon_k$ values (0.222, 0.269, 0.298) and together receive 79.4\% of the total weight. Intuitively, their edge-orientation, local-texture, and global-layout summaries provide the most stable evidence for the shared low-rank object structure in these data.

To quantify the performance, we evaluate the joint subspace shared by all six views against the 20 object-class labels. AJIVE and one-shot \hjive{} yield SWISS values of 0.4284 and 0.4012, respectively, and label-matched clustering errors of 0.6023 and 0.5826, respectively. Both metrics indicate a modest but consistent improvement of \hjive{} over AJIVE on this  dataset.

\bibliographystyle{plainnat}
\bibliography{reference.bib}

\newpage
\appendix

\newpage
{\centering\section*{\LARGE Supplementary Material to ``Spectral Joint Subspace Estimation for Heterogeneous Multi-View Data: Geometry and Reweighting"}}
This Supplementary Material  collects the detailed technical proofs and  derivations that support the main results in ``Spectral Joint Subspace Estimation for Heterogeneous Multi-View Data: Geometry and Reweighting".
\vspace{0.5cm}

\section{Summary}
The proofs are organized as follows. Theorem 5 provides the general upper-bound result from which Theorems 1--4 follow. Section B.1 records the equal-weight specialization. The proofs of Theorems 3 and 4 are given in Sections B.2 and B.3, and Theorem 5 is proved in Section B.4. The algorithm-specific lower bound, Proposition 1, is proved in Section B.5. The proofs of Propositions 2 and 3 are given in Sections B.6 and B.7, respectively. Section B.8 records the optional oracle fixed-point refinement and proves Proposition 4.
Finally, proofs of the primary lemmas used in \Cref{sec:pf:thm:prop} are provided in \Cref{pf:lemmas}, while additional technical lemmas are deferred to \Cref{sec:pf:technicallemmas}.

We briefly highlight the technical novelty of our analysis. Our primary tool is the explicit representation of the spectral projector developed by \cite{xia2021normal}. While \cite{yang2025estimating} utilizes the same tool, their analysis is restricted to the first-order expansion. Consequently, their derived error bound remains non-vanishing with respect to $K$, even in the ideal scenario where $\V_k^\top \W_k = 0$. In contrast, our analysis systematically accounts for the full expansion to all orders, which is technically more involved and leads to sharper bounds that can vanish as $K \to \infty$ under near-orthogonal loading-orientation conditions or a sign-symmetric random-loading model.

\section{Proofs of Theorems and Propositions}\label{sec:pf:thm:prop}
{
Before stating the general theorem, we collect some notation used in the perturbation expansion. Let
\begin{align*}
    \calG:=\sigma\brac{\U,\ebrac{\U_k,\V_k,\W_k}_{k=1}^{K}},
    \qquad
    \mat{\G_k\\ \H_k}
    :=\mat{\bar\U_k^\top\E_k\\ \bar\U_{k\perp}^\top\E_k},
    \quad k\in[K].
\end{align*}
Here $\bar\R_k$ denotes the right singular vectors of $\U\V_k^\top+\U_k\W_k^\top$ associated with its nonzero singular values. We define the local good event $\calE_k:=\calE_{k,1}\cap\calE_{k,2}$ by
\begin{equation}
\label{eq:local-good-event}
\begin{aligned}
    \calE_{k,1}
    &:=\ebrac{\op{\G_k\bar\R_k}\le C\sqrt{n\wedge d_k}\sigma_k}
    \cap\ebrac{\op{\H_k\bar\R_k}\le C\sqrt{n}\sigma_k}, \\
    \calE_{k,2}
    &:=\ebrac{\op{\G_k\G_k^\top-d_k\sigma_k^2\I}\le C\sqrt{(n\wedge d_k)d_k}\sigma_k^2}
    \cap\ebrac{\op{\G_k\H_k^\top}\le C\sqrt{nd_k}\sigma_k^2} \\
    &\qquad\cap\ebrac{\op{\H_k\H_k^\top-d_k\sigma_k^2\I}\le C\brac{\sqrt{nd_k}+n}\sigma_k^2}.
\end{aligned}
\end{equation}
By standard Gaussian operator-norm and Wishart concentration bounds, the constant $C$ can be chosen so that, conditionally on $\calG$,
\begin{align*}
    \PP(\calE_k^c\mid\calG)\leq C\exp\{-c(n\wedge d_k)\}.
\end{align*}}
Finally, define $\R_k:=\U^\top\bar\U_k\bar\U_k^\top
    \brac{\bar\U_k\bar\U_k^\top-\tilde\U_k\tilde\U_k^\top}\bar\U_k\bar\U_k^\top\U_\perp\bLa^{-1}$.

We now state the general theorem, whose proof can be found in \Cref{proof:thm:main}.
Recall the misalignment is defined as $\theta = 1- \op{\sum_{k=1}^{K} w_k\U_k\U_k^\top}$, and $\I - \sum_{k=1}^{K} w_k\bar\U_k\bar\U_k^\top = \U_{\perp}\bLa\U_{\perp}^\top$
and $\varepsilon_k = n^{1/2}\lambda_{k,\min}^{-1}\sigma_k +(nd_k)^{1/2}\lambda_{k,\min}^{-2}\sigma_k^2$.
\begin{theorem}\label{thm:main}
Assume, for a sufficiently small absolute constant $c_0>0$ and a sufficiently large absolute constant $C_0>0$,
    \begin{align}
    &\kappa_k\snr_k^{-1}n^{1/2} + \snr_k^{-2}(nd_k)^{1/2}\le c_0, \quad\forall k\in[K], \label{theta-max}\\
        &\theta\ge C_0\max\bigg\{\sum_{k=1}^{K} w_k\varepsilon_k^2, \bigg(\sum_{k=1}^{K}w_k^2\varepsilon_k^2\log n\bigg)^{1/2}\bigg\}.\label{theta-average}  
    \end{align}
Then we have with probability exceeding $1-C\sum_{k=1}^Ke^{-n\wedge d_k}-Cn^{-10}$,
\begin{align*}
    \op{\hat\U\hat\U^\top - \U\U^\top} &\leq C\sqrt{\sum_{k=1}^{K}n^{-1}w_k^2 \big(\tr(\M_k)+\bar r_k\op{\M_k}\big)\varepsilon_k^{2}}\sqrt{\log (n)}\\
& +C\sqrt{\sum_{k=1}^{K}w_k^2\varepsilon_k^{4}\theta^{-2}}\sqrt{\log n}
+ C\op{\sum_{k=1}^Kw_k\EE_{\calG}(\R_k\cdot 1(\calE_k))},
\end{align*}
where $\M_k = \bar\U_{k\perp}^\top\U_{\perp}\bLa^{-2}\U_{\perp}^\top\bar\U_{k\perp}$.
\end{theorem}


{
\subsection{Equal-weight specialization}

We record the reduction of the first term in Theorem~\ref{thm:main}
to the form used in Theorems~\ref{thm:equal-weight-deterministic}
and~\ref{thm:equal-weight-random}. Let $w_k=K^{-1}$ and define, on
$\operatorname{span}(\U)^\perp$,
\begin{align*}
    \B
    :=\frac1K\sum_{k=1}^K
    \U_\perp^\top\U_k\U_k^\top\U_\perp,
    \qquad
    \widetilde\P_k
    :=\I-\U_\perp^\top\U_k\U_k^\top\U_\perp.
\end{align*}
Then
\begin{align*}
    \bLa=\I-\B,
    \qquad
    \frac1K\sum_{k=1}^K\widetilde\P_k=\bLa.
\end{align*}
By cyclicity of the trace,
\begin{align*}
    \frac1K\sum_{k=1}^K\tr(\M_k)
    &=\frac1K\sum_{k=1}^K
      \tr\brac{\bLa^{-2}\widetilde\P_k}
      =\tr(\bLa^{-1}).
\end{align*}
If $\mu_1,\ldots,\mu_{n-r}$ are the eigenvalues of $\B$, then
$0\leq\mu_j\leq1-\theta$ and
$\sum_j\mu_j=\tr(\B)=r_{\rm avg}$. Hence
\begin{align*}
    \tr(\bLa^{-1})
    &=\sum_{j=1}^{n-r}\frac1{1-\mu_j}
      =(n-r)+\sum_{j=1}^{n-r}\frac{\mu_j}{1-\mu_j}\\
    &\leq n+\theta^{-1}r_{\rm avg}.
\end{align*}
Since $\op{\M_k}\leq\tr(\M_k)$ and
$\bar r_k\leq\bar r$, it follows that
\begin{align*}
&\sum_{k=1}^K
  n^{-1}K^{-2}
  \brac{\tr(\M_k)+\bar r_k\op{\M_k}}\\
&\qquad\leq
C\frac{\bar r}{K}
\brac{1+\frac{r_{\rm avg}}{n\theta}}.
\end{align*}
Because $\varepsilon_k\asymp\varepsilon$, the first term in
Theorem~\ref{thm:main} is therefore bounded by
\begin{align*}
    C\varepsilon
    \sqrt{
      \frac{\bar r\log n}{K}
      \brac{1+\frac{r_{\rm avg}}{n\theta}}
    }.
\end{align*}
Combining this calculation with
Lemma~\ref{lemma:expectation-upper-bound} gives
Theorem~\ref{thm:equal-weight-deterministic}; combining it with the
conditional centering argument in the proof of
Theorem~\ref{thm:adaptive-weights-random} gives
Theorem~\ref{thm:equal-weight-random}.
}

\subsection{Proof of Theorem \ref{thm:adaptive-weights-deterministic}}
\label{proof:thm:adaptive-weights-deterministic}
The assumptions in \Cref{thm:adaptive-weights-deterministic} imply those of \Cref{thm:main}. Hence it remains only to control the last term in \Cref{thm:main}. By the triangle inequality and \Cref{lemma:expectation-upper-bound},
\begin{align*}
    \op{\sum_{k=1}^{K}w_k\EE_{\calG}\bbrac{\R_k\cdot 1(\calE_k)}}
    &\leq \sum_{k=1}^{K}w_k\op{\EE_{\calG}\bbrac{\R_k\cdot 1(\calE_k)}} \\
    &\leq C\sum_{k=1}^{K}w_k\brac{\delta_k(1-\delta_k^2)^{-1}\wedge 1}\varepsilon_k^2\theta^{-1}.
\end{align*}
Substituting this bound into \Cref{thm:main} gives the asserted result.

\subsection{Proof of Theorem \ref{thm:adaptive-weights-random}}\label{proof:thm:adaptive-weights-random}
{
For the random-loading setting, we apply \Cref{thm:main} conditionally on
\begin{align*}
    \calG:=\sigma\brac{\U,\ebrac{\U_j,\V_j,\W_j}_{j=1}^{K}},
    \qquad
    \calH:=\sigma\brac{\U,\ebrac{\U_j,\W_j}_{j=1}^{K}},
\end{align*}
and write $\EE_{\calG}(\cdot):=\EE(\cdot\mid\calG)$. For each $k\in[K]$, let
\begin{align*}
    \calF_k:=\ebrac{\lambda_{r+r_k}\brac{\mat{\V_k&\W_k}}\geq \lambda_{k,\min}}.
\end{align*}
By \Cref{assump:VW:random}, $\PP(\calF_k)\geq 1-p_k$. Define the conditional bias term
\begin{align*}
    \T_k:=\EE_{\calG}\bbrac{\R_k\cdot 1(\calE_k)}.
\end{align*}
On $\calF_k$, by \Cref{lemma:expectation-upper-bound}, we have 
\begin{align}\label{eq:random-Tk-bound}
    \op{\T_k}
    &\leq C\varepsilon_k^2\brac{\delta_k(1-\delta_k^2)^{-1}\wedge 1}\theta^{-1}
    \leq C\varepsilon_k^2\theta^{-1}.
\end{align}

We next show that $\T_k\cdot 1(\calF_k)$ is centered with respect to the randomness of $\V_k$, conditionally on $\calH$. Fix $\calH$ and set
\begin{align*}
    \Pi_k:=\begin{bmatrix}\I_r&0\\0&-\I_{r_k}\end{bmatrix}.
\end{align*}
Under the sign flip $\V_k\mapsto -\V_k$,
\begin{align*}
    \mat{(-\V_k)^\top\\ \W_k^\top}\mat{-\V_k&\W_k}
    =\Pi_k\bGa_k\Pi_k.
\end{align*}
Moreover, $\calF_k$ is invariant under this sign flip, since $\mat{\V_k&\W_k}$ and $\mat{-\V_k&\W_k}$ have the same singular values. Let
\begin{align*}
    \B_k(\V_k,\W_k)
    :=\EE_{\calG}\bbrac{\bar\U_k^\top\brac{\bar\U_k\bar\U_k^\top-\tilde\U_k\tilde\U_k^\top}\bar\U_k\cdot 1(\calE_k)}.
\end{align*}
The Gaussian law of the noise blocks and the event $\calE_k$ are invariant under the corresponding block sign flip. Therefore, the equivariance argument used in the proof of \Cref{lemma:expectation-upper-bound} gives
\begin{align*}
    \B_k(-\V_k,\W_k)=\Pi_k\B_k(\V_k,\W_k)\Pi_k.
\end{align*}
Consequently, the $(1,2)$ block is odd in $\V_k$:
\begin{align*}
    [\B_k(-\V_k,\W_k)]_{12}=-[\B_k(\V_k,\W_k)]_{12}.
\end{align*}
Since
\begin{align*}
    \T_k
    =\begin{bmatrix}\I_r&0\end{bmatrix}\B_k(\V_k,\W_k)
    \begin{bmatrix}0\\ \U_k^\top\U_\perp\end{bmatrix}\bLa^{-1},
\end{align*}
and $\U_k,\U_\perp,\bLa$ do not depend on the sign of $\V_k$, we have
\begin{align*}
    \T_k(-\V_k,\W_k)=-\T_k(\V_k,\W_k).
\end{align*}
By \Cref{assump:VW:random}, conditionally on $\mathcal H$, $\V_k\stackrel{d}{=}-\V_k$, and hence
\begin{align}\label{eq:random-Tk-centered}
    \EE\bbrac{\T_k\cdot 1(\calF_k)\mid \calH}=0.
\end{align}

Now consider $\sum_{k=1}^{K}w_k\T_k\cdot 1(\calF_k)$. Conditionally on $\calH$, the summands are independent {by \Cref{assump:VW:random}} and centered by \eqref{eq:random-Tk-centered}. Moreover, \eqref{eq:random-Tk-bound} implies
\begin{align*}
    \op{w_k\T_k\cdot 1(\calF_k)}\leq Cw_k\varepsilon_k^2\theta^{-1}.
\end{align*}
Applying the matrix Hoeffding inequality to the self-adjoint dilation of the sum yields, with probability at least $1-O(n^{-10})$,
\begin{align}\label{eq:random-bias-sum-bound}
    \op{\sum_{k=1}^{K}w_k\T_k\cdot 1(\calF_k)}
    \leq C\brac{\sum_{k=1}^{K}w_k^2\varepsilon_k^4\theta^{-2}\log n}^{1/2}.
\end{align}
On $\cap_{k=1}^{K}\calF_k$, the indicators in \eqref{eq:random-bias-sum-bound} can be removed. Since
\begin{align*}
    \PP\brac{\bigcap_{k=1}^{K}\calF_k}
    \geq 1-\sum_{k=1}^{K}p_k,
\end{align*}
we obtain
\begin{align}\label{eq:random-conditional-bias-bound}
    \op{\sum_{k=1}^{K}w_k\EE_{\calG}\bbrac{\R_k\cdot 1(\calE_k)}}
    \leq C\brac{\sum_{k=1}^{K}w_k^2\varepsilon_k^4\theta^{-2}\log n}^{1/2}
\end{align}
with probability at least $1-\sum_{k=1}^{K}p_k-O(n^{-10})$.

Combining \eqref{eq:random-conditional-bias-bound} with \Cref{thm:main} applied conditionally on $\calG$ gives
\begin{align*}
    \op{\hat\U\hat\U^\top-\U\U^\top}
    &\leq C\sqrt{\sum_{k=1}^{K}n^{-1}w_k^2\brac{\tr(\M_k)+\bar r_k\op{\M_k}}\varepsilon_k^2}\sqrt{\log n} \\
    &\quad +C\sqrt{\sum_{k=1}^{K}w_k^2\varepsilon_k^4\theta^{-2}}\sqrt{\log n}.
\end{align*}
The probability of the intersection of the above events is at least $1-O\bbrac{\sum_{k=1}^{K}e^{-n\wedge d_k}}-\sum_{k=1}^{K}p_k-O\bbrac{n^{-10}}$. This completes the proof.
}

\subsection{Proof of Theorem \ref{thm:main}}\label{proof:thm:main}
Condition on $\calG$ throughout this proof. Under the assumption \eqref{theta-average}, we have $\theta>0$ and thus 
$\op{\sum_{k=1}^{K}w_k\U_k\U_k^\top}<1$. 
Since $\U_k^\top\U = 0$, we can conclude $$\I - \sum_{k=1}^{K}w_k\bar\U_k\bar\U_k^\top = \I - \U\U^\top - \sum_{k=1}^{K}w_k\U_k\U_k^\top$$ has rank $n-r$ and has $\U_{\perp}\R_0$ as its top $n-r$ left singular vectors, where $\R_0\in\OO_{n-r}$ is orthogonal matrix, and its minimal non-zero eigenvalue is $\theta$. 
On the other hand, $\hat\U_{\perp}$ is the top $n-r$ left singular vectors of $\I - \sum_{k=1}^{K}w_k\tilde\U_k\tilde\U_k^\top$. 

Let the rank $n-r$ decomposition of $\I - \sum_{k=1}^{K}w_k\bar\U_k\bar\U_k^\top$ be $\I - \sum_{k=1}^{K}w_k\bar\U_k\bar\U_k^\top = \U_{\perp}\bLa\U_{\perp}^\top$ (notice here $\bLa$ is not necessarily diagonal). 
And we denote $\Q_{\perp} = \I - \U_{\perp}\U_{\perp}^\top = \U\U^\top$, $\Q^{-s}=\U_{\perp}\bLa^{-s}\U_{\perp}^\top$ for $s>0$, and we denote 
\begin{align}\label{eq:delta}
    \bDelta = (\I -\sum_{k=1}^{K}w_k\tilde\U_k\tilde\U_k^\top) - (\I - \sum_{k=1}^{K}w_k\bar\U_k\bar\U_k^\top) =\sum_{k=1}^{K}w_k\underbrace{(\bar\U_k\bar\U_k^\top-\tilde\U_k\tilde\U_k^\top)}_{:=\bDelta_k}.
\end{align}
Using the notation introduced before \Cref{thm:main}, define
\begin{align*}
    \calE_{(i)}:=\bigcap_{k=1}^{K}\calE_{k,i},\qquad i=1,2.
\end{align*}
Then
\begin{align*}
    \PP\brac{\calE_{(1)}\cap\calE_{(2)}}
    \ge 1-C\sum_{k=1}^{K}e^{-n\wedge d_k}.
\end{align*}

We now bound the operator norms $\bop{\sum_{k=1}^{K} w_k\U^\top\bDelta_k\U_{\perp}\cdot 1(\calE_k)}$, $\bop{\sum_{k=1}^{K} w_k\U^\top\bDelta_k\U\cdot 1(\calE_k)}$, and $\bop{\sum_{k=1}^{K}w_k\U_{\perp}^\top\bDelta_k\U_{\perp}\cdot 1(\calE_k)}$, which are building blocks of our analysis.
\begin{lemma}\label{lemma:general}
We have
\begin{align*}
    \op{\sum_{k=1}^{K}w_k\U^\top\bDelta_k\U\cdot 1(\calE_k)} \leq C\sum_{k=1}^{K} w_k\varepsilon_k^2.
\end{align*}
Conditionally on $\calG$, with probability exceeding $1-4n^{-10}$, we have
\begin{align*}
\op{\sum_{k=1}^{K}w_k\U^\top\bDelta_k\U_{\perp}\cdot 1(\calE_k)}&\leq C\sum_{k=1}^{K}w_k\varepsilon_k^2 + C\bigg(\sum_{k=1}^{K}w_k^2\varepsilon_k^{2}\log (n)\bigg)^{1/2},\\
 \op{\sum_{k=1}^{K}w_k\U_{\perp}^\top\bDelta_k\U_{\perp}\cdot 1(\calE_k)}
    &\leq C\sum_{k=1}^{K}w_k\varepsilon_k^2 + C\bigg(\sum_{k=1}^{K}w_k^2\varepsilon_k^{2}\log (n)\bigg)^{1/2}.
\end{align*}
As a result, we have
\begin{align}\label{eq:sumkwkDeltak}
    \op{\sum_{k=1}^{K} w_k \bDelta_k\cdot 1(\calE_k)} \leq C\sum_{k=1}^{K}w_k\varepsilon_k^2 + C\bigg(\sum_{k=1}^{K}w_k^2\varepsilon_k^{2}\log (n)\bigg)^{1/2}.
\end{align}
Moreover, conditionally on $\calG$, with probability exceeding $1-2n^{-10}$, we have
\begin{align*}
&\op{\sum_{k=1}^{K}w_k\U^\top\bDelta_k\U_{\perp}\bLa^{-1}\cdot 1(\calE_k)}\leq C\sqrt{\sum_{k=1}^{K}n^{-1}w_k^2 \big(\tr(\M_k)+\bar r_k\op{\M_k}\big)\varepsilon_k^{2}}\sqrt{\log (n)}\\
& \hspace{2.5cm}+ C\sqrt{\sum_{k=1}^{K}w_k^2\varepsilon_k^{4}\theta^{-2}}\sqrt{\log n}+C\op{\sum_{k=1}^Kw_k\EE_{\calG}(\R_k\cdot 1(\calE_k))}.
\end{align*}
\end{lemma}

Let $\calE_{\rm agg}$ be the intersection of
$\calE_{(1)}\cap\calE_{(2)}$ and the two conditional concentration events
in Lemma~\ref{lemma:general}.  Conditionally on $\calG$,
\begin{align*}
\PP(\calE_{\rm agg}^c\mid\calG)
\leq C\sum_{k=1}^K e^{-n\wedge d_k}+Cn^{-10}.
\end{align*}
On $\calE_{\rm agg}$, the truncated sums equal their untruncated
counterparts, and \eqref{eq:sumkwkDeltak} together with
\eqref{theta-average} gives
\begin{align*}
    \op{\bDelta}\leq c_1\theta
\end{align*}
for an absolute $c_1<1/8$, after choosing $C_0$ sufficiently large.  Here $\theta$ is the minimal nonzero eigenvalue of $\I - \sum_{k=1}^{K} w_k \bar\U_k\bar\U_k^\top$.
Then from \cite[Theorem 1]{xia2021normal}, we can expand $\hat\U\hat\U^\top - \U\U^\top$ as follows ($\bDelta$ defined in \eqref{eq:delta}):
\begin{align*}
    &\quad  \U\U^\top-\hat\U\hat\U^\top =(\I - \hat\U\hat\U^\top) - (\I - \U\U^\top)\\
    &= \sum_{l\geq 1}\sum_{\s\in\SS_l}(-1)^{\lzero{\s}+1}\M(s_1)\bLa^{-s_1}\underline{\M(s_1)^\top\bDelta\M(s_2)}\cdots\underline{\M(s_l)^\top\bDelta\M(s_{l+1})}\bLa^{-s_{l+1}}\M(s_{l+1})^\top,
\end{align*}
where $\M(0) = \U$, $\M(s) = \U_{\perp}$ for $s>0$,
$\SS_l = \big\{\s = (s_1,\cdots,s_{l+1})\in\NN_0^{l+1}: s_1 + \cdots + s_{l+1} = l\big\}$, where $\NN_0=\{0,1,2,\ldots\}$,
and $\bLa^{-0} = \I_r$ with slight abuse of notation.

On $\calE_{\rm agg}$,
\begin{align*}
    \theta^{-1}\max_{s_1,s_2}\op{\M(s_1)^\top \bDelta\M(s_2)}\leq c_1.
\end{align*}

Then for each $l\geq 1, \s\in\SS_l$, $$\M(s_1)\bLa^{-s_1}\underline{\M(s_1)^\top\bDelta\M(s_2)}\cdots\underline{\M(s_l)^\top\bDelta\M(s_{l+1})}\bLa^{-s_{l+1}}\M(s_{l+1})^\top$$ has the following property. Since $\s\in\SS_l$, the sequence contains at least one zero and at least one positive entry. Hence there exists an adjacent pair $(s_{l_0},s_{l_0+1})$ with exactly one entry equal to zero, which guarantees the presence of a cross block of the form $\U^\top\bDelta\U_{\perp}\bLa^{-1}$ or its transpose. So

\begin{align*}
    &\quad \op{\M(s_1)\bLa^{-s_1}\underline{\M(s_1)^\top\bDelta\M(s_2)}\cdots\underline{\M(s_l)^\top\bDelta\M(s_{l+1})}\bLa^{-s_{l+1}}\M(s_{l+1})^\top}\\
&\leq\op{\U^\top\bDelta\U_{\perp}\bLa^{-1}}\bigg(\max_{s_1,s_2}\op{\M(s_1)^\top \bDelta\M(s_2)}\theta^{-1}\bigg)^{l-1}\\
    &\leq c_1^{l-1}\op{\U^\top\bDelta\U_{\perp}\bLa^{-1}},\qquad\text{on }\calE_{\rm agg}.
\end{align*}

And we can apply the upper bound in Lemma \ref{lemma:general}. Together with the fact that $|\SS_l|\leq 4^l$, we conclude

\begin{align*}
    \op{\hat\U\hat\U^\top - \U\U^\top}&\lesssim \sqrt{\sum_{k=1}^{K}n^{-1}w_k^2 \big(\tr(\M_k)+\bar r_k\op{\M_k}\big)\varepsilon_k^{2}}\sqrt{\log (n)}\\
&\hspace{-2.5cm}+\sqrt{\sum_{k=1}^{K}w_k^2\varepsilon_k^{4}\theta^{-2}}\sqrt{\log n}+C\op{\sum_{k=1}^Kw_k\EE_{\calG}(\R_k\cdot 1(\calE_k))},
\qquad\text{on }\calE_{\rm agg}.
\end{align*}

Together with the probability bound for $\calE_{\rm agg}$, this proves the asserted bound conditionally on $\calG$ with probability at least $1-C\sum_{k=1}^{K}e^{-n\wedge d_k}-Cn^{-10}$. The constants are uniform in the conditioned signal quantities, so averaging over $\calG$ gives the stated probability bound.

\subsection{Proof of Proposition \ref{prop:lb-ajive}}\label{pf-prop:lb-ajive}

{We use the notation introduced in \Cref{subsec:rank-one-limitation}. Set
$\S:=K^{-1}\sum_{k=1}^{K}\u_k\u_k^\top$ and write
$\I-\u\u^\top-\S=\U_\perp\bLa\U_\perp^\top$. All expectations below are
over the Gaussian noise blocks. Let $\calE_k$ be the local good event in
\eqref{eq:local-good-event}. Define
\begin{equation}
    \bDelta_k:=\bar\U_k\bar\U_k^\top-
    \widetilde\U_k\widetilde\U_k^\top,
    \qquad
    \bDelta:=\frac1K\sum_{k=1}^{K}\bDelta_k,
    \label{eq:lb-Delta}
\end{equation}
and let
\begin{equation*}
    \widehat\M:=\frac1K\sum_{k=1}^{K}
    \widetilde\U_k\widetilde\U_k^\top.
\end{equation*}
Introduce the outer concentration event
\begin{equation*}
    \calE_\Delta
    :=\left\{\op{\bDelta}\leq
    C_\Delta\left(\varepsilon\sqrt{\frac{\log n}{K}}+\varepsilon^2\right)\right\},
    \qquad
    \calE_{\rm out}:=\calE_\Delta\cap\bigcap_{k=1}^{K}\calE_k,
\end{equation*}
where $C_\Delta>0$ is universal. Let
$\calE_*:=\bigcap_{k=1}^{K}\calE_k$. On $\calE_*$,
\begin{equation*}
    \bDelta=\frac1K\sum_{k=1}^{K}\bDelta_k1(\calE_k).
\end{equation*}
Hence the equal-weight specialization of \eqref{eq:sumkwkDeltak}, together
with $\varepsilon_k\asymp\varepsilon$, gives, for $C_\Delta$ sufficiently
large,
\begin{equation*}
    \PP(\calE_\Delta^c\cap\calE_*)\leq Cn^{-10}.
\end{equation*}
The local Gaussian bounds also give
$\PP(\calE_*^c)\leq CK\exp\{-cn\}$. Therefore, after decreasing the
universal constant $c_2$ if necessary,
\begin{equation}
    \PP(\calE_{\rm out}^c)
    \leq C\left\{K\exp\{-c_2n\}+n^{-10}\right\}.
    \label{eq:lb-global-good-prob}
\end{equation}
By \eqref{eq:lb-smallness}, $\op{\bDelta}\leq c\theta$ on
$\calE_{\rm out}$ after decreasing $c_1$ if necessary.

Define the projected bias
\begin{equation}
    \calL
    :=\u^\top\brac{\EE(\hat\u\hat\u^\top)-\u\u^\top}\U_\perp.
    \label{eq:lb-target}
\end{equation}
Because $\I-\widehat\M=(\I-\u\u^\top-\S)+\bDelta$, the first-order term in
the bottom-eigenspace projector expansion has a minus sign. Thus, on
$\calE_{\rm out}$,
\begin{equation}
    \u^\top(\hat\u\hat\u^\top-\u\u^\top)\U_\perp
    =-\u^\top\bDelta\U_\perp\bLa^{-1}+\mathfrak R_{\rm out}.
    \label{eq:lb-outer-expansion}
\end{equation}
Set
\begin{align}
    \calR_{\rm out}
    &:=\EE\ebrac{\mathfrak R_{\rm out}\,1(\calE_{\rm out})},
    \notag\\
    \calR_{\rm out,tail}
    &:=\EE\sqbrac{\brac{
      \u^\top(\hat\u\hat\u^\top-\u\u^\top)\U_\perp
      +\u^\top\bDelta\U_\perp\bLa^{-1}}
      1(\calE_{\rm out}^c)}.
    \label{eq:lb-Rout-Rtail-def}
\end{align}
This definition gives the exact identity
\begin{equation}
    \calL=-\u^\top\EE\bDelta\U_\perp\bLa^{-1}
    +\calR_{\rm out}+\calR_{\rm out,tail}.
    \label{eq:lb-outer}
\end{equation}

As in the upper-bound proof, write
\begin{align}
    \u^\top\bDelta_k\U_\perp\bLa^{-1}
    &=\u^\top\bar\U_k\bar\U_k^\top\bDelta_k
      \bar\U_k\bar\U_k^\top\U_\perp\bLa^{-1}
      +\u^\top\bar\U_k\bar\U_k^\top\bDelta_k
      \bar\U_{k\perp}\bar\U_{k\perp}^\top\U_\perp\bLa^{-1}
      \notag\\
    &:=\R_k+\Q_k.
    \label{eq:lb-Rk}
\end{align}
Conditionally on $\G_k=\bar\U_k^\top\E_k$, set
$\J_k:=\bar\U_k\bar\U_k^\top-
\bar\U_{k\perp}\bar\U_{k\perp}^\top$. Under
$\H_k=\bar\U_{k\perp}^\top\E_k\mapsto-\H_k$, equivariance of the local
spectral projector gives
$\bDelta_k(\G_k,-\H_k)=\J_k\bDelta_k(\G_k,\H_k)\J_k$.
Thus $\Q_k$ changes sign, whereas its conditional law is invariant. Hence
$\EE\Q_k=0$ without any
truncation, so the linear deterministic bias is exactly
$K^{-1}\sum_{k=1}^{K}\EE\R_k$.}

{Let $\SS_l$ be the index set in the local projector expansion, let
$\S_{k,l}(\s)$ be the order-$l$ term in the expansion of $\bDelta_k$, and set
\[
    \SS_l^{\rm ss}:=\{\s\in\SS_l:s_1>0,\ s_{l+1}>0\}.
\]
On $\calE_k$ the local series is
\begin{equation}
    \bDelta_k=\sum_{l\geq1}\sum_{\s\in\SS_l}\S_{k,l}(\s).
    \label{eq:lb-local-expansion}
\end{equation}
Only the indices in $\SS_l^{\rm ss}$ contribute after left and right multiplication by $\bar\U_k^\top$ and $\bar\U_k$, respectively.
At order two, the only such index is $\s=(1,0,1)$.  With the convention
$\bDelta_k=\bar\U_k\bar\U_k^\top-\widetilde\U_k\widetilde\U_k^\top$, its coefficient is positive; hence the matrix obtained after left and right multiplication by $\bar\U_k^\top$ and $\bar\U_k$ is exactly the positive semidefinite product displayed below.
Let
\begin{equation*}
    \X_k:=\mat{\v_k&\w_k}\in\RR^{d_k\times2},
    \qquad
    \G_k:=\bar\U_k^\top\E_k,
    \qquad
    \H_k:=\bar\U_{k\perp}^\top\E_k.
\end{equation*}
The part containing exactly two noise factors in this order-two matrix is
\begin{equation}
    \T_{k,{\rm lead}}
    :=\bGa_k^{-1}\X_k^\top\H_k^\top\H_k\X_k\bGa_k^{-1}.
    \label{eq:lb-S2-explicit}
\end{equation}
We therefore define, for every realization of the noise,
\begin{equation}
    \R_{k,{\rm lead}}
    :=\u^\top\bar\U_k\T_{k,{\rm lead}}
      \bar\U_k^\top\U_\perp\bLa^{-1},
    \label{eq:lb-Rlead-S}
\end{equation}
and
\begin{equation}
    \R_{k,{\rm rem}}:=\R_k-\R_{k,{\rm lead}},
    \qquad
    \R_k=\R_{k,{\rm lead}}+\R_{k,{\rm rem}}.
    \label{eq:lb-R-decomp}
\end{equation}
This definition is global and hence does not discard the complements of the
good events. To identify the terms contained in $\R_{k,{\rm rem}}$, note
that the complete order-two matrix obtained after compression by $\bar\U_k$ equals
\begin{equation*}
    \bGa_k^{-1}(\X_k^\top+\G_k)\H_k^\top\H_k
    (\X_k+\G_k^\top)\bGa_k^{-1}.
\end{equation*}
Consequently, the part of this block having noise degree at least three is
\begin{equation}
\begin{aligned}
    \T_{k,2+}:=\bGa_k^{-1}\big(&
       \X_k^\top\H_k^\top\H_k\G_k^\top
       +\G_k\H_k^\top\H_k\X_k\\
       &+\G_k\H_k^\top\H_k\G_k^\top\big)\bGa_k^{-1}.
\end{aligned}
\label{eq:lb-S2-higher-noise}
\end{equation}
On $\calE_k$, the exact remainder can thus be written as
\begin{equation}
\begin{aligned}
    \R_{k,{\rm rem}}
    =\u^\top\bar\U_k\Bigg\{
       \T_{k,2+}
       +\sum_{l\geq3}\sum_{\s\in\SS_l^{\rm ss}}
       \bar\U_k^\top\S_{k,l}(\s)\bar\U_k
       \Bigg\}\bar\U_k^\top\U_\perp\bLa^{-1}.
\end{aligned}
\label{eq:lb-Rrem-S}
\end{equation}
In particular, the cubic and quartic terms involving $\G_k$ are part of the
remainder rather than of the leading coefficient. Their full expectation
contains the quartic matrix $(n-2)d_k\sigma^4\bGa_k^{-2}$, which is of
order $\varepsilon^4$ under $d_k\asymp n$ and is controlled below.}

Define the deterministic geometry factor
\begin{equation}
    \B_K:=\frac1K\sum_{k=1}^{K}\u^\top\bar\U_k\bGa_k^{-1}\bar\U_k^\top \U_\perp\bLa^{-1}.
    \label{eq:lb-BK}
\end{equation}
The following lemmas give the leading coefficient, the deterministic geometry lower bound, and the local and outer remainder controls.

\begin{lemma}[Rank-one local leading coefficient]\label{lem:lb-coeff}
In the rank-one homogeneous Gaussian-noise regime of \Cref{subsec:rank-one-limitation},
\begin{equation}
    \frac1K\sum_{k=1}^{K}\EE \R_{k,{\rm lead}}=(n-2)\sigma^2\B_K.
    \label{eq:lb-coeff-exact}
\end{equation}
\end{lemma}

\begin{proof}
By \eqref{eq:lb-Rlead-S} and \eqref{eq:lb-S2-explicit},
\begin{equation*}
    \EE \R_{k,{\rm lead}}
    =\u^\top\bar\U_k\bGa_k^{-1}\mat{\v_k&\w_k}^\top\EE(\H_k^\top \H_k)\mat{\v_k&\w_k}\bGa_k^{-1}\bar\U_k^\top \U_\perp\bLa^{-1}.
\end{equation*}
Since $\bar\U_{k\perp}$ has $n-2$ columns and $\E_k$ has i.i.d. $N(0,\sigma^2)$ entries, $\EE(\H_k^\top \H_k)=(n-2)\sigma^2\I_{d_k}$. Using $\mat{\v_k&\w_k}^\top\mat{\v_k&\w_k}=\bGa_k$, we obtain
\begin{equation*}
    \EE \R_{k,{\rm lead}}=(n-2)\sigma^2\u^\top\bar\U_k\bGa_k^{-1}\bar\U_k^\top \U_\perp\bLa^{-1}.
\end{equation*}
Averaging over $k$ gives \eqref{eq:lb-coeff-exact}.
\end{proof}

\begin{lemma}[Geometry lower bound]\label{lem:lb-geometry}
Under Assumptions \ref{ass:lb-loading} and \ref{ass:lb-majority},
\begin{equation}
    \op{\frac1K\sum_{k=1}^{K}\u^\top\bar\U_k\bGa_k^{-1}\bar\U_k^\top \U_\perp\bLa^{-1}}
    \geq 2\gamma\frac{\delta}{C(1+\delta)}\frac{1-\theta}{\theta}\lambda^{-2}.
    \label{eq:lb-geom}
\end{equation}
\end{lemma}

\begin{proof}
The sign convention in Assumption \ref{ass:lb-majority} is without loss of generality: the rank-one individual term $\u_k\w_k^\top$ is unchanged under the simultaneous sign flip $(\u_k,\w_k)\mapsto(-\u_k,-\w_k)$.
Let $\x:=\U_\perp^\top\z$. Since $\z\perp\u$ and $\S\z=(1-\theta)\z$,
\begin{equation*}
    (\I-\u\u^\top-\S)\z=\theta\z,
    \qquad
    \U_\perp\bLa^{-1}\x=\theta^{-1}\z,
    \qquad
    \op{\x}=1.
\end{equation*}
Write $g_k:=[\bGa_k^{-1}]_{12}$. Then $\u^\top\bar\U_k\bGa_k^{-1}\bar\U_k^\top \U_\perp\bLa^{-1}=g_k\u_k^\top \U_\perp\bLa^{-1}$, and hence
\begin{equation*}
    \op{\frac1K\sum_{k=1}^{K}\u^\top\bar\U_k\bGa_k^{-1}\bar\U_k^\top \U_\perp\bLa^{-1}}
    \geq \theta^{-1}\ab{\frac1K\sum_{k=1}^{K}g_k\inp{\u_k}{\z}}.
\end{equation*}
In the rank-one case,
\begin{equation*}
    g_k=-\frac{\v_k^\top\w_k}{\op{\v_k}^2\op{\w_k}^2-(\v_k^\top\w_k)^2}
    =-\frac{\delta s_k}{\op{\v_k}\op{\w_k}(1-\delta^2)}.
\end{equation*}
Thus $g_k\inp{\u_k}{\z}=-\delta(1-\delta^2)^{-1}s_km_k$.
Assumption~\ref{ass:lb-majority} gives
\begin{align*}
\left|\sum_{k=1}^Ks_km_k\right|
&=\left|\tau\brac{
2\sum_{k:\tau s_k=1}m_k-\sum_{k=1}^Km_k
}\right|\geq2\gamma\sum_{k=1}^Km_k.
\end{align*}
Since
$|g_k|a_k=\delta(1-\delta^2)^{-1}m_k$, it follows that
\begin{align*}
\left|\sum_{k=1}^Kg_k\inp{\u_k}{\z}\right|
\geq2\gamma\sum_{k=1}^K|g_k|a_k.
\end{align*}
For the uniform scale bound, the smallest eigenvalue of $\bGa_k$ satisfies
\begin{align*}
\lambda^2=\lambda_{\min}(\bGa_k)
\geq(1-\delta)\min\{\op{\v_k}^2,\op{\w_k}^2\}.
\end{align*}
Together with
$C^{-1}\leq\op{\v_k}/\op{\w_k}\leq C$, this gives
\begin{align*}
\op{\v_k}\op{\w_k}
&\leq C\min\{\op{\v_k}^2,\op{\w_k}^2\}
\leq \frac{C}{1-\delta}\lambda^2,
\end{align*}
and therefore
\begin{align*}
|g_k|
=\frac{\delta}{\op{\v_k}\op{\w_k}(1-\delta^2)}
\geq \frac{\delta}{C(1+\delta)}\lambda^{-2}.
\end{align*}
Since $0\leq a_k\leq1$ and $a_k\geq a_k^2$,
\begin{align*}
    \ab{\frac1K\sum_{k=1}^{K}g_k\inp{\u_k}{\z}}
    &\geq 2\gamma\frac1K\sum_{k=1}^{K}|g_k|a_k^2 \\
    &\geq 2\gamma\frac{\delta}{C(1+\delta)}\lambda^{-2}
    \frac1K\sum_{k=1}^{K}\inp{\u_k}{\z}^2.
\end{align*}
Finally, $K^{-1}\sum_{k=1}^{K}\inp{\u_k}{\z}^2=\z^\top\S\z=1-\theta$. Combining the displays proves the claim.
\end{proof}

{
\begin{lemma}[Local high-order control]\label{lem:lb-local-high-order}
Under the homogeneous scaling, \eqref{eq:lb-smallness}, and
Assumption~\ref{ass:lb-loading},
\begin{equation}
    \op{\frac1K\sum_{k=1}^{K}
    \EE\ebrac{\R_{k,{\rm rem}}\,1(\calE_k)}}
    \leq C\theta^{-1}\varepsilon^3.
    \label{eq:lb-local-high-order}
\end{equation}
\end{lemma}

\begin{proof}
Because the loading angle and scale ratio are fixed and
$\lambda_{k,\min}=\lambda$, Assumption~\ref{ass:lb-loading} implies that the
condition number of $\X_k$ is bounded uniformly in $k$. The bounds defining
$\calE_k$, together with $d_k\asymp n$, therefore give
\begin{equation*}
    \op{\bGa_k^{-1}\X_k^\top\H_k^\top}
       \leq C\frac{\sqrt n\sigma}{\lambda}=C\varepsilon,
    \qquad
    \op{\bGa_k^{-1}\G_k\H_k^\top}
       \leq C\frac{n\sigma^2}{\lambda^2}=C\varepsilon^2.
\end{equation*}
Applying these two bounds to the three summands in
\eqref{eq:lb-S2-higher-noise} gives
\begin{equation*}
    \op{\T_{k,2+}}\leq C(\varepsilon^3+\varepsilon^4).
\end{equation*}
For completeness, every order-$l$ term in the local projector expansion
is a product of $l$ normalized perturbation blocks, each bounded by
$C\varepsilon$ on $\calE_k$. Since $|\SS_l^{\rm ss}|\leq4^l$, this gives
\begin{equation*}
    \op{\sum_{\s\in\SS_l^{\rm ss}}
    \bar\U_k^\top\S_{k,l}(\s)\bar\U_k}
    \leq (C\varepsilon)^l,
    \qquad l\geq3,
\end{equation*}
where $|\SS_l|\leq4^l$ has been absorbed into $C^l$. Since
$C\varepsilon<1$ after choosing $c_1$ sufficiently small, the local series
converges absolutely and uniformly on $\calE_k$. Since
$\op{\bLa^{-1}}\leq\theta^{-1}$, \eqref{eq:lb-Rrem-S} yields
\begin{equation*}
    \op{\R_{k,{\rm rem}}\,1(\calE_k)}
    \leq C\theta^{-1}\left(\varepsilon^3+
    \sum_{l\geq3}(C\varepsilon)^l\right)
    \leq C\theta^{-1}\varepsilon^3.
\end{equation*}
The last inequality follows from \eqref{eq:lb-smallness} after choosing $c_1$
sufficiently small. Averaging and taking expectations proves the claim.
\end{proof}
}

{
\begin{lemma}[Outer high-order control]\label{lem:lb-outer}
Under the homogeneous scaling and \eqref{eq:lb-smallness},
\begin{equation}
    \op{\calR_{\rm out}}
    \leq C_{\rm out}\theta^{-2}\left(
       \frac{\varepsilon^2\log n}{K}+\varepsilon^4\right).
    \label{eq:lb-outer-rem}
\end{equation}
\end{lemma}

\begin{proof}
On $\calE_{\rm out}$, the spectral projector expansion and
\eqref{eq:lb-outer-expansion} give
\begin{equation*}
    \op{\mathfrak R_{\rm out}}
    \leq C\theta^{-2}\op{\bDelta}^2.
\end{equation*}
Using the definition of $\calE_\Delta$ and
$(a+b)^2\leq2a^2+2b^2$ therefore yields
\begin{equation*}
    \op{\calR_{\rm out}}
    \leq C\theta^{-2}\EE\ebrac{\op{\bDelta}^2
      1(\calE_{\rm out})}
    \leq C_{\rm out}\theta^{-2}\left(
       \frac{\varepsilon^2\log n}{K}+\varepsilon^4\right).
\end{equation*}
\end{proof}

Define the local tail contribution
\begin{equation*}
    \calR_{\rm loc,tail}
    :=\frac1K\sum_{k=1}^{K}
      \EE\ebrac{\R_{k,{\rm rem}}\,1(\calE_k^c)}.
\end{equation*}
\begin{lemma}[Good-event complements]\label{lem:lb-tails}
Under the conditions of Proposition~\ref{prop:lb-ajive},
\begin{equation}
    \op{\calR_{\rm loc,tail}}+
    \op{\calR_{\rm out,tail}}
    \leq C\theta^{-1}\left\{K\exp\{-c_2n\}+n^{-10}\right\}.
    \label{eq:lb-tail-bound}
\end{equation}
\end{lemma}

\begin{proof}
First, $\op{\R_k}\leq\theta^{-1}$ because it is a deterministic compression
of a difference of two orthogonal projectors followed by $\bLa^{-1}$.
Moreover, the Gaussian fourth-moment bound applied to
\eqref{eq:lb-S2-explicit} gives
\begin{equation*}
    \left(\EE\op{\R_{k,{\rm lead}}}^2\right)^{1/2}
    \leq C\theta^{-1}\varepsilon^2.
\end{equation*}
Indeed,
$\H_k\X_k\bGa_k^{-1/2}\stackrel{d}{=}\sigma\Z_{k,0}$ for a matrix
$\Z_{k,0}\in\RR^{(n-2)\times2}$ with i.i.d. standard Gaussian entries, and
$\EE\op{\Z_{k,0}^\top\Z_{k,0}}^2\leq Cn^2$.
Since $\PP(\calE_k^c)\leq C\exp\{-cn\}$ under $d_k\asymp n$, the identity
$\R_{k,{\rm rem}}=\R_k-\R_{k,{\rm lead}}$ and Cauchy--Schwarz imply, after
decreasing the universal exponent $c_2$ if necessary,
\begin{align*}
    \EE\ebrac{\op{\R_{k,{\rm lead}}}1(\calE_k^c)}
    &\leq
    \brac{\EE\op{\R_{k,{\rm lead}}}^2}^{1/2}
    \PP(\calE_k^c)^{1/2}\\
    &\leq C\theta^{-1}\varepsilon^2\exp\{-cn/2\},\\
    \EE\ebrac{\op{\R_k}1(\calE_k^c)}
    &\leq\theta^{-1}\PP(\calE_k^c).
\end{align*}
Consequently,
\begin{equation*}
    \op{\calR_{\rm loc,tail}}\leq C\theta^{-1}\left\{K\exp\{-c_2n\}+n^{-10}\right\}.
\end{equation*}
For the outer tail, both projector differences and $\bDelta$, which is an
average of projector differences, have operator norm at most one. Hence
\eqref{eq:lb-Rout-Rtail-def} and \eqref{eq:lb-global-good-prob} give
\begin{equation*}
    \op{\calR_{\rm out,tail}}
    \leq(1+\theta^{-1})\PP(\calE_{\rm out}^c)
    \leq C\theta^{-1}\left\{K\exp\{-c_2n\}+n^{-10}\right\}.
\end{equation*}
\end{proof}
}

{
\begin{proof}[Proof of \Cref{prop:lb-ajive}]
Set
\begin{equation*}
    \calR_{\rm loc}:=\frac1K\sum_{k=1}^{K}
      \EE\ebrac{\R_{k,{\rm rem}}\,1(\calE_k)}.
\end{equation*}
The exact outer identity \eqref{eq:lb-outer}, the fact that $\EE\Q_k=0$,
and the global decomposition
$\R_{k,{\rm rem}}=\R_k-\R_{k,{\rm lead}}$ give
\begin{equation}
    \calL
    =-\frac1K\sum_{k=1}^{K}\EE\R_{k,{\rm lead}}
      -\calR_{\rm loc}-\calR_{\rm loc,tail}
      +\calR_{\rm out}+\calR_{\rm out,tail}.
    \label{eq:lb-exact-final-decomp}
\end{equation}
Thus no local or outer bad-event contribution is omitted.

By Lemmas~\ref{lem:lb-coeff}--\ref{lem:lb-geometry},
\begin{equation*}
    \op{\frac1K\sum_{k=1}^{K}\EE\R_{k,{\rm lead}}}
    \geq \frac{n-2}{n}\,
    \frac{2\gamma\delta}{C(1+\delta)}
    \frac{1-\theta}{\theta}\,\varepsilon^2.
\end{equation*}
Because $c_0\leq\theta\leq1-c_0$ and $n$ is sufficiently large, the
coefficient on the right is bounded below by a positive constant depending
only on $c_0,\delta,\gamma$ and $C$. Combining
\eqref{eq:lb-exact-final-decomp} with
Lemmas~\ref{lem:lb-local-high-order}, \ref{lem:lb-outer}, and
\ref{lem:lb-tails}, we obtain
\begin{align*}
    \op{\calL}
    \geq{}& \frac{n-2}{n}\,
    \frac{2\gamma\delta}{C(1+\delta)}
    \frac{1-\theta}{\theta}\,\varepsilon^2
    -C\theta^{-1}\varepsilon^3\\
    &\quad-C_{\rm out}\theta^{-2}\left(
       \frac{\varepsilon^2\log n}{K}+\varepsilon^4\right)
    -C\theta^{-1}\left\{K\exp\{-c_2n\}+n^{-10}\right\}.
\end{align*}
Choose $C_1$ in \eqref{eq:lb-smallness} sufficiently large and $c_1$
sufficiently small, in that order, depending only on
$c_0,\delta,\gamma$ and $C$. Then $K\geq C_1\log n$ absorbs the
$\varepsilon^2\log n/K$ term, the small-perturbation condition absorbs the
$\varepsilon^3$ and $\varepsilon^4$ terms, and
the last condition in \eqref{eq:lb-smallness} absorbs the two tail contributions.
Consequently, $\op{\calL}\geq c\varepsilon^2$. Finally,
$\op{\EE(\hat\u\hat\u^\top)-\u\u^\top}\geq\op{\calL}$, which proves the second inequality in the statement. The first inequality follows from Jensen's inequality.
\end{proof}
}

\subsection{Proof of Proposition \ref{prop:random-Uk}}\label{pf-prop:random-Uk}
    We first show with high probability, $\theta$ is lower bounded by $1/2$. 
    Recall $\theta = 1-  K^{-1}\op{\sum_{k=1}^K \U_k\U_k^\top}$. From the way $\U_k$ is generated, we have $\EE \U_k\U_k^\top = \frac{r_k}{n-r}(\I-\U\U^\top)$ in the ambient space. On $\operatorname{span}(\U)^\perp$, this is equivalent to $\frac{r_k}{n-r}\I_{n-r}$. And $\op{\U_k\U_k^\top - \frac{r_k}{n-r}(\I-\U\U^\top)}\leq 1$. On the other hand, on $\operatorname{span}(\U)^\perp$, $\op{\EE(\U_k\U_k^\top - \frac{r_k}{n-r}(\I-\U\U^\top))^2} \leq r_k/(n-r)$. Using matrix Bernstein inequality on $\operatorname{span}(\U)^\perp$, we obtain with probability exceeding $1-n^{-10}$, 
    \begin{align*}
        \op{\sum_{k=1}^K \frac{1}{K}\U_k\U_k^\top} \leq \frac{r_{\rm avg}}{n-r} + C\sqrt{\frac{r_{\rm avg}\log n}{(n-r)K}} + C\frac{\log n }{K}\leq 1/2,
    \end{align*}
    where in the last inequality we use $n\geq 4r_{\rm avg}+r$, $K\geq C\log n$.
    {Combining this event with Theorem~\ref{thm:equal-weight-random}, and setting $w_k = 1/K$, gives, with probability exceeding
    $1-O\bbrac{Ke^{-n\wedge d}}-\sum_{k=1}^Kp_k-O\bbrac{n^{-10}}$, the upper bound}
    The theorem's second-order term is
    $\varepsilon^2\sqrt{\log n/K}$; it is absorbed by the first-order term below because the standing perturbative condition gives $\varepsilon\leq c$ and $\bar r\geq1$.
    \begin{align*}
    \op{\hat\U\hat\U^\top - \U\U^\top} &\lesssim \sqrt{\frac{\bar r\log n}{K}} \varepsilon.
\end{align*}

\subsection{Proof of Proposition \ref{prop:constant-gap-explicit-weight}}
\label{pf-prop:constant-gap-explicit-weight}
\begin{proof}
Fix $\w\in\mathcal D$. On $\operatorname{span}(\U)^\perp$, the nonzero
eigenvalues of
\begin{align*}
\I-\sum_{j=1}^K w_j\bar\U_j\bar\U_j^\top
=
\I-\U\U^\top-\sum_{j=1}^K w_j\U_j\U_j^\top
\end{align*}
lie in $[\theta(\w),1]$. Since
$\operatorname{span}(\bar\U_k)^\perp\subseteq\operatorname{span}(\U)^\perp$,
the matrix $\M_k(\w)$ is a compression of the inverse square of this operator.
Therefore,
\begin{align*}
\I_{n-\bar r_k}
\preceq \M_k(\w)
\preceq \theta(\w)^{-2}\I_{n-\bar r_k}.
\end{align*}
Using $(n-\bar r_k)+\bar r_k=n$ in the definition of $q_k(\w)$ gives
\begin{align*}
1\leq q_k(\w)\leq\theta(\w)^{-2},
\qquad k\in[K].
\end{align*}
Define $\Phi(\w):=\sum_{k=1}^K\xi_k w_k^2$.
Since $\theta(\w)\leq1$, the preceding bounds imply
\begin{align*}
\Phi(\w)\leq J(\w)\leq\theta(\w)^{-2}\Phi(\w),
\qquad \w\in\mathcal D.
\end{align*}
By Cauchy--Schwarz,
\begin{align*}
1
=\left(\sum_{k=1}^K w_k\right)^2
\leq
\left(\sum_{k=1}^K\xi_k w_k^2\right)
\left(\sum_{k=1}^K\xi_k^{-1}\right),
\end{align*}
with equality at the weight vector $\w^\circ$ in
\eqref{eq:constant-gap-explicit-weight}. Hence
\begin{align*}
\Phi(\w^\circ)
=\min_{\w\in\Delta_K}\Phi(\w)
=\left(\sum_{k=1}^K\xi_k^{-1}\right)^{-1}.
\end{align*}
Using the preceding sandwich bound, we obtain
\begin{align*}
J(\w^\circ)
&\leq \theta(\w^\circ)^{-2}\Phi(\w^\circ)
=\theta(\w^\circ)^{-2}\min_{\w\in\Delta_K}\Phi(\w)\leq \theta(\w^\circ)^{-2}\inf_{\w\in\Delta_K}J(\w),
\end{align*}
where $J(\w)=+\infty$ on $\Delta_K\setminus\mathcal D$. If $\theta(\w^\circ)\geq c_0$, this further yields $J(\w^\circ)\leq c_0^{-2}J^\ast$.This proves the claim.
\end{proof}

\subsection{Optional oracle fixed-point refinement}\label{app:oracle-refinement}

For completeness, we record the oracle analogue of the optional refinement in
\Cref{subsec:optional-refinement}. Suppose that
$\{\varepsilon_k\}_{k=1}^K$, $\theta(\cdot)$, and
$\{\M_k(\cdot)\}_{k=1}^K$ are known, with $\varepsilon_k>0$. For
$\w\in\mathcal D$, define
\begin{align}\label{eq:c-def}
c_k(\w)
:=\varepsilon_k^4\theta(\w)^{-2}
+n^{-1}\varepsilon_k^2
\left[\tr\{\M_k(\w)\}+\bar r_k\op{\M_k(\w)}\right],
\qquad k\in[K],
\end{align}
and let $T:\mathcal D\to\Delta_K^\circ$ be given by
\begin{align}\label{eq:T-def}
T_k(\w)
:=\frac{c_k(\w)^{-1}}{\sum_{j=1}^Kc_j(\w)^{-1}},
\qquad k\in[K],
\end{align}
where $\Delta_K^\circ:=\{\w\in\Delta_K:w_k>0\text{ for every }k\}$. The iteration below is understood only for as long as its iterates remain in $\mathcal D$. The fixed-point statement concerns a point in $\mathcal D$ and is unaffected by this qualification.

\begin{algorithm}[H]
\caption{Oracle geometry-adaptive refinement}
\begin{algorithmic}
\State{\textbf{Input:} $\w^{(0)}\in\mathcal D$, $\{\varepsilon_k\}_{k=1}^K$,
$\{\M_k(\cdot)\}_{k=1}^K$, $\theta(\cdot)$, and $t_{\max}$}
\For{$t=0,\ldots,t_{\max}-1$}
\State{Compute $c_k\{\w^{(t)}\}$ from \eqref{eq:c-def} for $k\in[K]$.}
\State{Set $\w^{(t+1)}=T\{\w^{(t)}\}$ using \eqref{eq:T-def}.}
\EndFor
\State{\textbf{Output:} $\w^{(t_{\max})}$}
\end{algorithmic}
\label{alg:oracle}
\end{algorithm}

\begin{proposition}[Stationarity of a regular interior fixed point]
\label{prop:oracle-stationary}
Let $\w^\dagger\in\mathcal D$ be a fixed point of $T$ with
$w_k^\dagger>0$ for $k\in[K]$. Assume that (1) $\theta(\w^\dagger)>0$,
(2) all positive eigenvalues of
$\sum_{k=1}^Kw_k^\dagger\U_k\U_k^\top$ are simple, and (3)
$\lambda_{\max}\{\M_k(\w^\dagger)\}$ is simple for every $k\in[K]$.
Then $\w^\dagger$ is a $2L(\theta_0)$-stationary point of $J$ under the
simplex constraint in the sense that
\begin{align*}
\op{\operatorname{Proj}_{\mathcal T}\{\nabla J(\w^\dagger)\}}_\infty
&\leq2L(\theta_0),
\qquad
L(\theta_0)
:=4\theta_0^{-3}\max_{k\in[K]}
\left\{\varepsilon_k^4+3n^{-1}\varepsilon_k^2\bar r\right\},
\end{align*}
where $\theta_0:=\theta(\w^\dagger)/2$,
$\mathcal T:=\{\v\in\RR^K:\mathbf 1^\top\v=0\}$, and
\begin{align*}
\operatorname{Proj}_{\mathcal T}(\g)
:=\g-\mathbf 1\mathbf 1^\top\g/K.
\end{align*}
\end{proposition}

The proposition is conditional on the existence of a sufficiently regular interior fixed
point. It does not provide a convergence or global-optimality guarantee for the iteration.
When $L(\theta_0)=o(1)$, its conclusion says that such a fixed point has a small
projected gradient.

Define an $\ell_1$ ball around $\w^\dagger$ as
\begin{align*}
\Omega_{\theta_0}(\w^\dagger)
:=\left\{\w\in\Delta_K:\op{\w-\w^\dagger}_1\leq\theta_0\right\}.
\end{align*}
We use the following two lemmas, whose proofs are given below.

\begin{lemma}\label{lem:lip}
For any $\w,\w'\in\Omega_{\theta_0}(\w^\dagger)$, we have
\begin{align*}
|c_k(\w)-c_k(\w')|
&\leq\frac{L(\theta_0)}{2}\op{\w-\w'}_1,
\qquad k\in[K].
\end{align*}
\end{lemma}

\begin{lemma}\label{lem:diff}
Under the conditions of Proposition~\ref{prop:oracle-stationary}, there exists a
neighborhood $\Omega$ of $\w^\dagger$ such that
$\{c_k(\w)\}_{k=1}^K$ and $J(\w)$ are $C^1$ on $\Omega$.
\end{lemma}

\begin{proof}[Proof of Proposition~\ref{prop:oracle-stationary}]
Since $\w^\dagger$ is a fixed point, \eqref{eq:T-def} gives
\begin{align*}
w_k^\dagger
=\frac{c_k(\w^\dagger)^{-1}}
{\sum_{j=1}^Kc_j(\w^\dagger)^{-1}},
\qquad k\in[K].
\end{align*}
Define
$\alpha:=\big(\sum_{j=1}^Kc_j(\w^\dagger)^{-1}\big)^{-1}$, so that
$w_k^\dagger c_k(\w^\dagger)=\alpha$. We thus have
\begin{align*}
\frac{\partial J(\w^\dagger)}{\partial w_k}
=2\alpha+\sum_{j=1}^K(w_j^\dagger)^2
\frac{\partial c_j(\w^\dagger)}{\partial w_k},
\qquad k\in[K].
\end{align*}

For any fixed $j,k,l\in[K]$, the interiority of $\w^\dagger$ implies that
$\w^\dagger+t(\mathbf e_k-\mathbf e_l)
\in\Omega_{\theta_0}(\w^\dagger)$ for all sufficiently small $|t|$.
Since $\op{t(\mathbf e_k-\mathbf e_l)}_1=2|t|$, Lemma~\ref{lem:lip} gives
\begin{align*}
\left|c_j\{\w^\dagger+t(\mathbf e_k-\mathbf e_l)\}
-c_j(\w^\dagger)\right|
\leq L(\theta_0)|t|.
\end{align*}
Dividing by $|t|$ and letting $t\to0$, Lemma~\ref{lem:diff} yields
\begin{align*}
\left|
\frac{\partial c_j(\w^\dagger)}{\partial w_k}
-\frac{\partial c_j(\w^\dagger)}{\partial w_l}
\right|
\leq L(\theta_0).
\end{align*}
Moreover, $\sum_{j=1}^K(w_j^\dagger)^2\leq1$. Therefore, for any
$k,l\in[K]$,
\begin{align*}
\left|
\frac{\partial J(\w^\dagger)}{\partial w_k}
-\frac{\partial J(\w^\dagger)}{\partial w_l}
\right|
&\leq\sum_{j=1}^K(w_j^\dagger)^2
\left|
\frac{\partial c_j(\w^\dagger)}{\partial w_k}
-\frac{\partial c_j(\w^\dagger)}{\partial w_l}
\right|\\
&\leq2L(\theta_0).
\end{align*}
Let $S:=K^{-1}\sum_{k=1}^K\partial J(\w^\dagger)/\partial w_k$. Then
\begin{align*}
\left|\frac{\partial J(\w^\dagger)}{\partial w_k}-S\right|
&\leq\frac1K\sum_{l=1}^K
\left|
\frac{\partial J(\w^\dagger)}{\partial w_k}
-\frac{\partial J(\w^\dagger)}{\partial w_l}
\right|
\leq2L(\theta_0),
\end{align*}
which proves the claim.
\end{proof}

\section{Proofs of Lemmas}\label{pf:lemmas}
\subsection{Proof of Lemma \ref{lemma:general}}    

We start with the expansion of $\bDelta_k$. 
Recall 	$\A_k = \underbrace{\U\V_k^\top + \U_k\W_k^\top}_{:=\A_k^*} + \E_k$, and $\bar\U_k = \mat{\U& \U_k}$ is the top $\bar r_k = r+r_k$ left singular vectors of $\A_k^*\A_k^{*\top}$, $\tilde\U_k$ is the top $\bar r_k$ left singular vectors of $\A_k\A_k^\top-d_k\sigma_k^2\I$. 
	Then $\A_k^*$ admits a rank $\bar r_k$ decomposition as 
	\begin{align*}
		\A_k^* = \mat{\U&\U_k}\mat{\D_1& 0\\ \D_3& \D_2}\mat{\R_1^\top\\\R_2^\top}:= \bar\U_k\bar\D_k\bar\R_k^\top,
	\end{align*}
	where $\D_1\in\RR^{r\times r},\D_2\in\RR^{r_k\times r_k}$ are not necessarily diagonal matrices, and $\bar\R_k$ has orthonormal columns. Then we have $\bGa_k = \mat{\V_k^\top\\ \W_k^\top}\mat{\V_k&\W_k} = 
    \bar\D_k\bar\D_k^\top$. 
We denote $\bXi_k = \A_k^*\E_k^\top + \E_k\A_k^{*\top} + \E_k\E_k^\top - d_k\sigma_k^2\I$. 
On $\calE_k$, the block bounds in \eqref{eq:local-good-event} imply
\begin{align*}
\op{\bXi_k}
&\leq C\left\{\kappa_k\lambda_{k,\min}\sqrt n\,\sigma_k
 +(\sqrt{nd_k}+n)\sigma_k^2\right\}\\
&\leq C\lambda_{k,\min}^2\left\{
 \kappa_k\snr_k^{-1}\sqrt n
 +\snr_k^{-2}\sqrt{nd_k}
 +n\snr_k^{-2}\right\}.
\end{align*}
Because $\kappa_k\geq1$ and
$n\snr_k^{-2}\leq(\kappa_k\sqrt n\,\snr_k^{-1})^2$, condition
\eqref{theta-max}, with $c_0$ sufficiently small, gives
$\op{\bXi_k}\leq c\lambda_{k,\min}^2$. Hence the projector series below
converges absolutely on $\calE_k$.

For $s\in\mathbb N_0$, define
\begin{align*}
\calP_{k}^{-s}
:=
\begin{cases}
\bar\U_k\bGa_k^{-s}\bar\U_k^\top, & s\geq1,\\
\bar\U_{k\perp}\bar\U_{k\perp}^\top, & s=0.
\end{cases}
\end{align*}
Then \cite[Theorem 1]{xia2021normal} gives
\begin{align}\label{eq:Deltak-xia}
\tilde \U_k\tilde \U_k^\top-\bar \U_k\bar\U_k^\top
&=\sum_{l\geq1}\sum_{\s\in\SS_l}
\underbrace{(-1)^{\lzero{\s}+1}
\calP_k^{-s_1}\bXi_k\calP_k^{-s_2}\bXi_k\cdots
\bXi_k\calP_k^{-s_{l+1}}}_{:=\widetilde{\S}_{k,l}(\s)}.
\end{align}
Under our convention
$\bDelta_k=\bar\U_k\bar\U_k^\top-\tilde\U_k\tilde\U_k^\top$, hence
\begin{align}\label{eq:Deltak}
\bDelta_k
&=-\sum_{l\geq1}\sum_{\s\in\SS_l}\widetilde{\S}_{k,l}(\s)
=: \sum_{l\geq1}\sum_{\s\in\SS_l}\S_{k,l}(\s),
\qquad
\S_{k,l}(\s):=-\widetilde{\S}_{k,l}(\s).
\end{align}
For coordinate calculations below, we also write
$\N_k(0)=\bar\U_{k\perp}$ and $\N_k(s)=\bar\U_k$ for $s>0$.
Thus, when $s_j>0$, the $j$th factor in \eqref{eq:Deltak-xia} is
$\N_k(s_j)\bGa_k^{-s_j}\N_k(s_j)^\top$; when $s_j=0$, it is
$\bar\U_{k\perp}\bar\U_{k\perp}^\top$.

Then the following matrices can be written in terms of $\G_k,\H_k$:
\begin{subequations} \label{eq:UxiU}
    \begin{align}
        \bar\U_k^\top\bXi_k\bar\U_k &= \bar\D_k\bar\R_k^\top\G_k^\top + \G_k\bar\R_k\bar\D_k^\top + \G_k\G_k^\top - d_k\sigma_k^2\I, \\
        \bar\U_k^\top\bXi_k\bar\U_{k\perp} &= \bar\D_k\bar\R_k^\top\H_k^\top + \G_k\H_k^\top, \\
        \bar\U_{k\perp}^\top\bXi_k\bar\U_{k\perp} &= \H_k\H_k^\top - d_k\sigma_k^2\I. 
    \end{align}
\end{subequations}
We begin with the upper bounds for $\op{(\bGa_k)^{-\frac{1}{2}}\bar\U_k^\top\bXi_{k}\bar\U_k(\bGa_k)^{-\frac{1}{2}}}$, $\op{\bar\U_{k\perp}^\top\bXi_{k}\bar\U_k(\bGa_k)^{-\frac{1}{2}}}$, and $\op{\bar\U_{k\perp}^\top\bXi_{k}\bar\U_{k\perp}}$. We have 
	\begin{align*}
		(\bGa_k)^{-\frac{1}{2}}\bar\U_k^\top\bXi_{k}\bar\U_k(\bGa_k)^{-\frac{1}{2}} = (\bar\D_k\bar\D_k^\top)^{-\frac{1}{2}}\big(\bar\D_k\bar\R_k^\top\G_k^\top + \G_k\bar\R_k\bar\D_k^\top + \G_k\G_k^\top - d_k\sigma_k^2\I\big) (\bar\D_k\bar\D_k^\top)^{-\frac{1}{2}},
	\end{align*}
	and thus under $\calE_k$ it is bounded by
	\begin{align*}
			\op{(\bGa_k)^{-\frac{1}{2}}\bar\U_k^\top\bXi_{k}\bar\U_k(\bGa_k)^{-\frac{1}{2}} } \lesssim (\bar r_k + (n\wedge d_k))^{1/2}\lambda_{k,\min}^{-1} (\sigma_k + d_k^{1/2}\lambda_{k,\min}^{-1}\sigma_k^2).
	\end{align*}
	And 
	\begin{align*}
		(\bGa_k)^{-\frac{1}{2}}\bar\U_{k}^\top\bXi_{k}\bar\U_{k\perp}= (\bar\D_k\bar\D_k^\top)^{-\frac{1}{2}}\big(\bar\D_k\bar\R_k^\top\H_k^\top + \G_k\H_k^\top\big),
	\end{align*}
	which is bounded by 
	\begin{align*}
		\op{(\bGa_k)^{-\frac{1}{2}}\bar\U_{k}^\top\bXi_{k}\bar\U_{k\perp}} \lesssim n^{1/2}\sigma_k +(nd_k)^{1/2}\lambda_{k,\min}^{-1}\sigma_k^2. 
	\end{align*}
	under $\calE_k$. 
	And 
	\begin{align*}
		\bar\U_{k\perp}^\top\bXi_{k}\bar\U_{k\perp}= \H_k\H_k^\top - d_k\sigma_k^2\I,
	\end{align*}
	which is bounded by 
	\begin{align*}
		\op{\bar\U_{k\perp}^\top\bXi_{k}\bar\U_{k\perp}} \lesssim\big((n\vee d_k)n\big)^{1/2}\sigma_k^2. 
	\end{align*}
		under $\calE_k$. Using the event bounds above and the crude dimension bound $\bar r_k\le n\wedge d_k$, we conclude that
	\begin{align}\label{defu}
		&\max\bigg\{\op{(\bGa_k)^{-\frac{1}{2}}\bar\U_k^\top\bXi_{k}\bar\U_k(\bGa_k)^{-\frac{1}{2}} }, \lambda_{k,\min}^{-1}\op{(\bGa_k)^{-\frac{1}{2}}\bar\U_{k}^\top\bXi_{k}\bar\U_{k\perp}}, \lambda_{k,\min}^{-2}\op{\bar\U_{k\perp}^\top\bXi_{k}\bar\U_{k\perp}}\bigg\}\notag\\
		&\lesssim n^{1/2}\lambda_{k,\min}^{-1}\sigma_k +(nd_k)^{1/2}\lambda_{k,\min}^{-2}\sigma_k^2 = \varepsilon_k 
	\end{align}
Here the term $n\sigma_k^2\lambda_{k,\min}^{-2}$ from $\bar\U_{k\perp}^\top\bXi_k\bar\U_{k\perp}$ is absorbed because
$n\snr_k^{-2}\leq\sqrt n\,\snr_k^{-1}$ whenever $\sqrt n\,\snr_k^{-1}\leq1$. Thus \eqref{defu} follows from \eqref{theta-max}.

    We summarize some important observations about $\bDelta_k$ in the following lemmas.

    \begin{lemma}\label{lemma:boundDeltak}
    On $\calE_k$,
    \begin{align*}
        \op{\bDelta_k}&\leq C\varepsilon_k,\\
        \op{\bar\U_k^\top\bDelta_k\bar\U_k}
        +\op{\bar\U_{k\perp}^\top\bDelta_k\bar\U_{k\perp}}
        &\leq C\varepsilon_k^2,\\
        \op{\bar\U_k^\top\bDelta_k\bar\U_{k\perp}}
        &\leq C\varepsilon_k.
    \end{align*}
    \end{lemma}

The proof of the following lemma is given in Section~\ref{sec:proof:lemma:meanzero}.

\begin{lemma}\label{lemma:meanzero}
Conditionally on $\calG$,
\begin{align*}
    \EE_{\calG}\bbrac{\bar\U_k^\top\bDelta_k\bar\U_{k\perp}\cdot 1(\calE_k)}=0.
\end{align*}
\end{lemma}

Now we are ready to prove the lemma. The proof of the lemmas involved can be found in the appendix. 

\noindent\textbf{Bounding $\op{\sum_{k=1}^{K}w_k\U^\top\bDelta_k\U_{\perp}\bLa^{-1}\cdot 1(\calE_k)}$.}
Since $\U^\top \bar\U_{k\perp} = 0$, we have the following decomposition:
\begin{align*}
    &\quad\sum_{k=1}^{K}w_k\U^\top\bDelta_k\U_{\perp}\bLa^{-1}\cdot 1(\calE_k) \\
    &= \sum_{k=1}^{K}w_k\underbrace{\U^\top\bar\U_k\bar\U_k^\top\bDelta_k\bar\U_k\bar\U_k^\top\U_{\perp}\bLa^{-1}}_{:=\R_k}\cdot 1(\calE_k) + \sum_{k=1}^{K}w_k\underbrace{\U^\top\bar\U_k\bar\U_k^\top\bDelta_k\bar\U_{k\perp}\bar\U_{k\perp}^\top\U_{\perp}\bLa^{-1}}_{:=\Q_k}\cdot 1(\calE_k).
\end{align*}
The following lemma provides the conditional expectation bound for $\R_k\cdot 1(\calE_k)$.
\begin{lemma}\label{lemma:expectation-upper-bound}
For all $k\in[K]$,
\begin{align*}
    \op{\EE_{\calG}\bbrac{\R_k\cdot 1(\calE_k)}}
    \le C\varepsilon_k^2\brac{\delta_k(1-\delta_k^2)^{-1}\wedge 1}\theta^{-1}.
\end{align*}
\end{lemma}
Next the following lemma controls the upper bound for  $$\op{\sum_{k=1}^{K}w_k\R_k\cdot 1(\calE_k) - \EE\sum_{k=1}^{K}w_k\R_k\cdot 1(\calE_k)}.$$ 
\begin{lemma}\label{lemma:concentrationR:general}
Conditionally on $\calG$,
\begin{align*}
\PP\bigg(\op{\sum_{k=1}^{K}w_k\Big\{\R_k\cdot 1(\calE_k)-\EE_{\calG}[\R_k\cdot 1(\calE_k)]\Big\}}\geq C\sqrt{\sum_{k=1}^{K}w_k^2\varepsilon_k^{4}\theta^{-2}}\sqrt{\log n}\,\Bigm|\,\calG\bigg)\leq n^{-10}.
\end{align*}
\end{lemma}

Finally, we consider the upper bound for $\op{\sum_{k=1}^{K}w_k\Q_k\cdot 1(\calE_k)}$. 

\begin{lemma}\label{lemma:concentrationUDU}
Conditionally on $\calG$,
\begin{align*}
\PP\bigg(\op{\sum_{k=1}^{K}w_k\Q_k\cdot 1(\calE_k)}\geq C\sqrt{\sum_{k=1}^{K}n^{-1}w_k^2 \big(\tr(\M_k)+\bar r_k\op{\M_k}\big)\varepsilon_k^{2}}\sqrt{\log (n)}\,\Bigm|\,\calG\bigg)\leq n^{-10},
\end{align*}
where $\M_k = \bar\U_{k\perp}^\top\U_{\perp}\bLa^{-2}\U_{\perp}^\top\bar\U_{k\perp}$.
\end{lemma}
Combining these lemmas and we get the desired bound for $\op{\sum_{k=1}^{K}w_k\U^\top\bDelta_k\U_{\perp}\bLa^{-1}\cdot 1(\calE_k)}$.

\noindent\textbf{Bounding $\op{\sum_{k=1}^{K}w_k\U^\top\bDelta_k\U\cdot 1(\calE_k)}$.}
We have 
\begin{align*}
    \op{\sum_{k=1}^{K}w_k\U^\top\bDelta_k\U\cdot 1(\calE_k)} &= \op{\sum_{k=1}^{K}w_k\U^\top\bar\U_k\cdot \underline{\bar\U_k^\top\bDelta_k\bar\U_k}\cdot \bar\U_k^\top\U\cdot 1(\calE_k)}
    \leq \sum_{k=1}^{K} w_k\varepsilon_k^2. 
\end{align*}

\noindent\textbf{Bounding $\op{\sum_{k=1}^{K}w_k\U_{\perp}^\top\bDelta_k\U_{\perp}\cdot 1(\calE_k)}$.}
It can be written as 
\begin{align*}
    &\quad\op{\sum_{k=1}^{K}w_k\U_{\perp}^\top\bDelta_k\U_{\perp}\cdot 1(\calE_k)} \\
    &\leq \op{\sum_{k=1}^{K}w_k\U_{\perp}^\top\bar\U_k\underline{\bar\U_k^\top\bDelta_k\bar\U_k}\bar\U_k^\top\U_{\perp}\cdot 1(\calE_k)}+\op{\sum_{k=1}^{K}w_k\U_{\perp}^\top\bar\U_{k\perp}\underline{\bar\U_{k\perp}^\top\bDelta_k\bar\U_{k\perp}}\bar\U_{k\perp}^\top\U_{\perp}\cdot 1(\calE_k)}\\
	    &\quad + \op{\sum_{k=1}^{K}w_k\U_{\perp}^\top\bar\U_k\underline{\bar\U_k^\top\bDelta_k\bar\U_{k\perp}}\bar\U_{k\perp}^\top\U_{\perp}\cdot 1(\calE_k)} + \op{\sum_{k=1}^{K}w_k\U_{\perp}^\top\bar\U_{k\perp}\underline{\bar\U_{k\perp}^\top\bDelta_k\bar\U_k}\bar\U_k^\top\U_{\perp}\cdot 1(\calE_k)}.
\end{align*}
For the first two terms on the right hand side, we have 
\begin{align*}
   \op{\sum_{k=1}^{K}w_k\U_{\perp}^\top\bar\U_k\bar\U_k^\top\bDelta_k\bar\U_k\bar\U_k^\top\U_{\perp}\cdot 1(\calE_k)} &\leq C\sum_{k=1}^{K}w_k \varepsilon_k^2, \\
    \op{\sum_{k=1}^{K}w_k\U_{\perp}^\top\bar\U_{k\perp}\bar\U_{k\perp}^\top\bDelta_k\bar\U_{k\perp}\bar\U_{k\perp}^\top\U_{\perp}\cdot 1(\calE_k)}&\leq C\sum_{k=1}^{K}w_k \varepsilon_k^2.
\end{align*}
For the last two distinct cross terms, we use Lemma \ref{lemma:concentration-general} directly and symmetrically for the transpose-type term, and, conditionally on $\calG$, we have with probability exceeding $1-2n^{-10}$, 
\begin{align*}
	    \op{\sum_{k=1}^{K}w_k\U_{\perp}^\top\bar\U_k\underline{\bar\U_k^\top\bDelta_k\bar\U_{k\perp}}\bar\U_{k\perp}^\top\U_{\perp}\cdot 1(\calE_k)} &\leq C\sqrt{\sum_{k=1}^{K}w_k^2\varepsilon_k^{2}}\sqrt{\log (n)},\\
    \op{\sum_{k=1}^{K}w_k\U_{\perp}^\top\bar\U_{k\perp}\underline{\bar\U_{k\perp}^\top\bDelta_k\bar\U_k}\bar\U_k^\top\U_{\perp}\cdot 1(\calE_k)}&\leq C\sqrt{\sum_{k=1}^{K}w_k^2\varepsilon_k^{2}}\sqrt{\log (n)}.
\end{align*}
In conclusion, we have 
\begin{align*}
 \op{\sum_{k=1}^{K}w_k\U_{\perp}^\top\bDelta_k\U_{\perp}\cdot 1(\calE_k)}
    &\leq C\sum_{k=1}^{K}w_k\varepsilon_k^2 + C\sqrt{\sum_{k=1}^{K}w_k^2\varepsilon_k^{2}}\sqrt{\log (n)}.
\end{align*}
Finally, notice that the upper bound for $\op{\sum_{k=1}^{K}w_k\U^\top\bDelta_k\U_{\perp}\cdot 1(\calE_k)}$ can be proved using Lemma \ref{lemma:expectation-upper-bound} - Lemma \ref{lemma:concentrationUDU} with $\bLa = \I$. In fact, from the lemmas above, we conclude 
\begin{align*}
    \op{\sum_{k=1}^{K}w_k\U^\top\bDelta_k\U_{\perp}\cdot 1(\calE_k)} &\leq C\sqrt{\sum_{k=1}^{K}w_k^2\varepsilon_k^{2}}\sqrt{\log n} + C\sum_{k=1}^{K}w_k\varepsilon_k^2.
\end{align*}

\subsection{Proof of Lemma \ref{lem:lip}}

We first show that for any $\w,\w'\in\Delta_K$,
\begin{align}\label{eq:theta-lip}
|\theta(\w)-\theta(\w')|\leq\op{\w-\w'}_1.
\end{align}
Denote $\B(\w):=\sum_{k=1}^Kw_k\U_k\U_k^\top$, then
$\theta(\w)=1-\lambda_{\max}\{\B(\w)\}$. Notice that
\begin{align*}
\op{\B(\w)-\B(\w')}
\leq\sum_{k=1}^K|w_k-w_k'|
=\op{\w-\w'}_1.
\end{align*}
By Weyl's inequality, we thus get
\begin{align*}
|\theta(\w)-\theta(\w')|
&=|\lambda_{\max}\{\B(\w)\}
-\lambda_{\max}\{\B(\w')\}|\\
&\leq\op{\B(\w)-\B(\w')}
\leq\op{\w-\w'}_1.
\end{align*}
Then for any $\w\in\Omega_{\theta_0}(\w^\dagger)$,
\begin{align*}
\theta(\w)
\geq\theta(\w^\dagger)-\op{\w-\w^\dagger}_{1}
\geq\theta_0>0.
\end{align*}
Next, we show that for any
$\w,\w'\in\Omega_{\theta_0}(\w^\dagger)$,
\begin{align}\label{eq:M-k-lip}
\op{\M_k(\w)-\M_k(\w')}
\leq2\theta_0^{-3}\op{\w-\w'}_1,\qquad k\in[K].
\end{align}
Let $\A(\w):=\I-\B(\w)$.  Since
$\lambda_{\min}\{\A(\w)\}=\theta(\w)$,
\begin{align}\label{eq:A-positive}
\A(\w)\succeq\theta_0\I,
\qquad
\op{\A(\w)^{-1}}\leq\theta_0^{-1}
\end{align}
on $\Omega_{\theta_0}(\w^\dagger)$.  Moreover,
$\U_k^\top\U=\boldsymbol0$ for every $k$, so
$\B(\w)\U=\boldsymbol0$ and
\begin{align*}
\A(\w)
&=\U\U^\top
+\U_\perp(\w)\bLa(\w)\U_\perp(\w)^\top,\\
\A(\w)^{-2}
&=\U\U^\top
+\U_\perp(\w)\bLa(\w)^{-2}\U_\perp(\w)^\top.
\end{align*}
Because $\bar\U_{k\perp}^\top\U=\boldsymbol0$, it follows that
\begin{align}\label{eq:M-A}
\M_k(\w)
=\bar\U_{k\perp}^\top\A(\w)^{-2}\bar\U_{k\perp}.
\end{align}
This representation avoids inverting
$\I-\U\U^\top-\B(\w)$, which is singular on
$\operatorname{span}(\U)$.

Write $\A:=\A(\w)$ and $\A':=\A(\w')$.  The identities
\begin{align*}
\A^{-1}-\A'^{-1}
&=\A^{-1}(\A'-\A)\A'^{-1},\\
\A^{-2}-\A'^{-2}
&=\A^{-1}(\A^{-1}-\A'^{-1})
+(\A^{-1}-\A'^{-1})\A'^{-1}
\end{align*}
and \eqref{eq:A-positive}, together with
$\op{\A-\A'}=\op{\B(\w)-\B(\w')}
\leq\op{\w-\w'}_1$, imply
\begin{align*}
\op{\A^{-1}-\A'^{-1}}
&\leq\theta_0^{-2}\op{\A-\A'}
\leq\theta_0^{-2}\op{\w-\w'}_1,\\
\op{\A^{-2}-\A'^{-2}}
&\leq2\theta_0^{-1}\op{\A^{-1}-\A'^{-1}}
\leq2\theta_0^{-3}\op{\w-\w'}_1.
\end{align*}
Since $\bar\U_{k\perp}$ has orthonormal columns,
\eqref{eq:M-A} proves \eqref{eq:M-k-lip}.

To control the trace, let
\begin{align*}
d_l&:=w_l-w_l',
&\bDelta&:=\sum_{l=1}^Kd_l\U_l\U_l^\top,\\
\w_t&:=\w'+t(\w-\w'),
&\A_t&:=\A(\w_t),
\end{align*}
and let
$\boldsymbol P_k:=\bar\U_{k\perp}\bar\U_{k\perp}^\top$.
Then $\boldsymbol P_k$ is an orthogonal projector, so
$\op{\boldsymbol P_k}=1$.
The set $\Omega_{\theta_0}(\w^\dagger)$ is convex, so
$\op{\A_t^{-1}}\leq\theta_0^{-1}$ for $t\in[0,1]$.
Since $\dot\A_t=-\bDelta$, differentiating
$\A_t^{-1}\A_t=\I$ and using the product rule give
\begin{align}\label{eq:A2-derivative}
\frac{\mathrm d}{\mathrm dt}\A_t^{-2}
=\A_t^{-1}\bDelta\A_t^{-2}
+\A_t^{-2}\bDelta\A_t^{-1}.
\end{align}
By \eqref{eq:M-A}, cyclicity of the trace, and trace duality,
\begin{align*}
\left|\frac{\mathrm d}{\mathrm dt}
\tr(\boldsymbol P_k\A_t^{-2})\right|
&\leq
\op{\A_t^{-2}\boldsymbol P_k\A_t^{-1}}\op{\bDelta}_*
+\op{\A_t^{-1}\boldsymbol P_k\A_t^{-2}}\op{\bDelta}_*\\
&\leq2\theta_0^{-3}\op{\bDelta}_*.
\end{align*}
Each $\U_l\U_l^\top$ is a rank-$r_l$ orthogonal projector. Hence
\begin{align*}
\op{\bDelta}_*
\leq\sum_{l=1}^K|d_l|\op{\U_l\U_l^\top}_*
=\sum_{l=1}^Kr_l|d_l|
\leq\bar r\op{\w-\w'}_1.
\end{align*}
Integrating \eqref{eq:A2-derivative} from $0$ to $1$ yields
\begin{align}\label{eq:trace-bound}
|\tr\{\M_k(\w)\}-\tr\{\M_k(\w')\}|
\leq2\theta_0^{-3}\bar r\op{\w-\w'}_1.
\end{align}
Equipped with \eqref{eq:theta-lip} and \eqref{eq:M-k-lip}, we can prove our
desired result. By definition, we get
\begin{align*}
|c_k(\w)-c_k(\w')|
&\leq\varepsilon_k^4
\left|\theta(\w)^{-2}-\theta(\w')^{-2}\right|\\
&\quad+n^{-1}\varepsilon_k^2\left(
|\tr\{\M_k(\w)\}-\tr\{\M_k(\w')\}|
+\bar r_k|\op{\M_k(\w)}-\op{\M_k(\w')}|
\right).
\end{align*}
The proof is completed by using \eqref{eq:theta-lip},
\eqref{eq:M-k-lip},
$|x^{-2}-y^{-2}|\leq2\theta_0^{-3}|x-y|$ for $x,y\geq\theta_0$,
the trace bound \eqref{eq:trace-bound}, and
$|\op{\M_k(\w)}-\op{\M_k(\w')}|
\leq\op{\M_k(\w)-\M_k(\w')}$.  Indeed,
\begin{align*}
|c_k(\w)-c_k(\w')|
&\leq2\theta_0^{-3}
\left\{\varepsilon_k^4+2n^{-1}\varepsilon_k^2\bar r\right\}
\op{\w-\w'}_1\\
&\leq\frac{L(\theta_0)}2\op{\w-\w'}_1,
\end{align*}
where $r_l\leq\bar r$ and $\bar r_k\leq\bar r$ were used.
\qed

\subsection{Proof of Lemma \ref{lem:diff}}
Denote $\B(\w):=\sum_{k=1}^Kw_k\U_k\U_k^\top$. The map
$\w\mapsto\B(\w)$ is affine in $\w$ and hence $C^1$. 
If $\B(\w^\dagger)\neq\boldsymbol0$, then its largest eigenvalue is
positive and simple by assumption (2).  Theorem~1 of
\cite{magnus1985differentiating} states that a simple eigenvalue of a real symmetric
matrix is a $C^\infty$ function of the matrix in a neighborhood of the
reference matrix.  Since $\w\mapsto\B(\w)$ is affine, it follows that
$\lambda_{\max}\{\B(\w)\}$, and hence
$\theta(\w)=1-\lambda_{\max}\{\B(\w)\}$, is $C^1$ on a neighborhood of
$\w^\dagger$.  The eigenvalue remains the largest one on a sufficiently
small neighborhood by its positive spectral gap and Weyl's inequality.
If $\B(\w^\dagger)=\boldsymbol0$, then the positivity of every
$w_k^\dagger$ and the positive semidefiniteness of
$\U_k\U_k^\top$ imply $r_k=0$ for every $k$.  Thus
$\B(\w)\equiv\boldsymbol0$ and $\theta(\w)\equiv1$.

Define $\A(\w):=\I-\B(\w)$.  Since
$\lambda_{\min}\{\A(\w^\dagger)\}=\theta(\w^\dagger)>0$,
$\A(\w)$ is positive definite on an open neighborhood $\Omega$ of
$\w^\dagger$ in $\RR^K$.  Matrix inversion is $C^\infty$ on the set of
nonsingular matrices. Explicitly,
\begin{align*}
D(\A^{-1})[\boldsymbol E]
=-\A^{-1}\boldsymbol E\A^{-1}.
\end{align*}
Therefore $\w\mapsto\A(\w)^{-2}$ is $C^1$ on $\Omega$.
By \eqref{eq:M-A}, $\w\mapsto\M_k(\w)$ and
$\w\mapsto\tr\{\M_k(\w)\}$ are also $C^1$ on $\Omega$.
Finally, assumption (3) and Theorem~1 of Magnus~\cite{magnus1985differentiating} imply,
after possibly shrinking $\Omega$, that
\begin{align*}
\w\longmapsto\op{\M_k(\w)}
=\lambda_{\max}\{\M_k(\w)\}
\end{align*}
is $C^1$ on $\Omega$ for every $k\in[K]$. Simplicity and Weyl's
inequality ensure that this local eigenvalue branch remains the largest one. The definitions of $c_k(\w)$ and
$J(\w)=\sum_{k=1}^Kw_k^2c_k(\w)$ then show that they are also $C^1$ on
$\Omega$.
\qed

\subsection{Proof of Lemma \ref{lemma:boundDeltak}}

Fix $l\geq1$ and $\s=(s_1,\ldots,s_{l+1})\in\SS_l$.
    We first bound the single term $\S_{k,l}(\s)$ in
    \eqref{eq:Deltak}.  Among the $l$ consecutive pairs
    $(s_j,s_{j+1})$, let
    \begin{align*}
        e_{11}
        &:=\#\{j:s_j>0,\ s_{j+1}>0\},\\
        e_{10}
        &:=\#\{j:\text{exactly one of }s_j,s_{j+1}\text{ is positive}\},\\
        e_{00}
        &:=\#\{j:s_j=s_{j+1}=0\}.
    \end{align*}
    Thus $e_{11}+e_{10}+e_{00}=l$.

    At every occurrence of $\bar\U_k^\top\bXi_k\bar\U_k$, insert
    $\bGa_k^{-1/2}$ on both sides.  At every occurrence of
    $\bar\U_k^\top\bXi_k\bar\U_{k\perp}$ or its transpose, insert
    $\bGa_k^{-1/2}$ on the side corresponding to $\bar\U_k$.  The remaining block is
    $\bar\U_{k\perp}^\top\bXi_k\bar\U_{k\perp}$.  If $s_j>0$, the
    factor $\bGa_k^{-s_j}$ in \eqref{eq:Deltak-xia} supplies one
    factor $\bGa_k^{-1/2}$ for each adjacent perturbation block.  This
    is always possible: an endpoint has one adjacent block and an
    interior index has two, while $s_j\geq1$.  All inverse powers left
    after these insertions have nonnegative exponents.

    By \eqref{defu}, on $\calE_k$ the three resulting block types satisfy
    \begin{align*}
    \op{\bGa_k^{-1/2}\bar\U_k^\top\bXi_k\bar\U_k
            \bGa_k^{-1/2}}
        &\leq C\varepsilon_k,\\
    \op{\bGa_k^{-1/2}\bar\U_k^\top\bXi_k\bar\U_{k\perp}}
        +\op{\bar\U_{k\perp}^\top\bXi_k\bar\U_k
            \bGa_k^{-1/2}}
        &\leq C\lambda_{k,\min}\varepsilon_k,\\
    \op{\bar\U_{k\perp}^\top\bXi_k\bar\U_{k\perp}}
        &\leq C\lambda_{k,\min}^2\varepsilon_k.
    \end{align*}
    The original inverse exponents sum to
    $s_1+\cdots+s_{l+1}=l$.  The inserted half powers use total exponent
    $(2e_{11}+e_{10})/2$.  Therefore all remaining inverse powers have
    product norm at most
    \begin{align*}
        \lambda_{k,\min}^{-2l+2e_{11}+e_{10}}.
    \end{align*}
    The $l$ perturbation blocks contribute at most
    \begin{align*}
        (C\varepsilon_k)^l
        \lambda_{k,\min}^{e_{10}+2e_{00}}.
    \end{align*}
    Multiplying the last two displays gives $(C\varepsilon_k)^l$, because
    \begin{align*}
    -2l+2e_{11}+e_{10}+e_{10}+2e_{00}=0.
    \end{align*}
    Hence every order-$l$ term in \eqref{eq:Deltak} satisfies
    \begin{align}\label{eq:single-local-monomial-bound}
        \op{\S_{k,l}(\s)}\,1(\calE_k)
        \leq (C\varepsilon_k)^l.
    \end{align}

    The number of elements of $\SS_l$ is
    $\binom{2l}{l}\leq4^l$.  Since \eqref{theta-max} and $\kappa_k\geq1$
    imply $\varepsilon_k\leq c$ for a sufficiently small absolute
    constant, \eqref{eq:single-local-monomial-bound} yields
    \begin{align*}
        \op{\bDelta_k}\,1(\calE_k)
        &\leq\sum_{l\geq1}(4C\varepsilon_k)^l
        \leq C\varepsilon_k.
    \end{align*}

    It remains to identify the first possible order of each compressed
    block.  The matrix
    $\bar\U_k^\top\S_{k,l}(\s)\bar\U_k$ can be nonzero only when
    $s_1>0$ and $s_{l+1}>0$.  These two inequalities cannot hold when
    $l=1$, because $s_1+s_2=1$.  Likewise,
    $\bar\U_{k\perp}^\top\S_{k,l}(\s)\bar\U_{k\perp}$ can be nonzero
    only when $s_1=s_{l+1}=0$, which is also impossible for $l=1$.
    Therefore both diagonal blocks start at order two, and
    \begin{align*}
        &\op{\bar\U_k^\top\bDelta_k\bar\U_k}\,1(\calE_k)
        +\op{\bar\U_{k\perp}^\top\bDelta_k\bar\U_{k\perp}}\,1(\calE_k)\\
        &\qquad\leq 2\sum_{l\geq2}(4C\varepsilon_k)^l
        \leq C\varepsilon_k^2.
    \end{align*}
    Finally,
    $\bar\U_k^\top\S_{k,l}(\s)\bar\U_{k\perp}$ can be nonzero only
    when $s_1>0$ and $s_{l+1}=0$; this is possible at order one.
    Hence
    \begin{align*}
        \op{\bar\U_k^\top\bDelta_k\bar\U_{k\perp}}\,1(\calE_k)
        \leq\sum_{l\geq1}(4C\varepsilon_k)^l
        \leq C\varepsilon_k.
    \end{align*}
    This proves all three claims.

\subsection{Proof of Lemma \ref{lemma:meanzero}}\label{sec:proof:lemma:meanzero}

Condition on $\calG$.  We use the original term-by-term symmetry argument
from the local projector expansion.  Condition further on
$\G_k=\bar\U_k^\top\E_k$.  Because $\E_k$ has independent Gaussian
entries and $[\bar\U_k,\bar\U_{k\perp}]$ is orthogonal, the matrix
$\H_k=\bar\U_{k\perp}^\top\E_k$ is independent of $\G_k$ and has the
same conditional distribution as $-\H_k$.

By \eqref{eq:UxiU}, under the replacement $\H_k\mapsto-\H_k$,
\begin{align*}
\bar\U_k^\top\bXi_k\bar\U_k
&\text{ is unchanged},\\
\bar\U_{k\perp}^\top\bXi_k\bar\U_{k\perp}
&\text{ is unchanged},\\
\bar\U_k^\top\bXi_k\bar\U_{k\perp}
&\text{ changes sign},
\end{align*}
and the transpose block changes sign as well.  The event $\calE_k$ is
unchanged: the two quantities involving $\H_k\bar\R_k$ and
$\G_k\H_k^\top$ may change sign, but their operator norms do not, and all
other quantities in \eqref{eq:local-good-event} are unchanged.

Now expand the cross block by \eqref{eq:Deltak}:
\begin{align*}
\bar\U_k^\top\bDelta_k\bar\U_{k\perp}\,1(\calE_k)
=\sum_{l\geq1}\sum_{\s\in\SS_l}
\bar\U_k^\top\S_{k,l}(\s)\bar\U_{k\perp}\,1(\calE_k).
\end{align*}
A summand can be nonzero only when $s_1>0$ and $s_{l+1}=0$.  Starting
from $s_1>0$ and ending at $s_{l+1}=0$, the indicator of the event
$s_j>0$ changes value an odd number of times as $j$ runs from $1$ to
$l$.  Equivalently, an odd number of the $l$ perturbation blocks in this
summand are of the form
$\bar\U_k^\top\bXi_k\bar\U_{k\perp}$ or its transpose.  Consequently,
for every such $l$ and $\s$,
\begin{align*}
\bar\U_k^\top\S_{k,l}(\s;\G_k,-\H_k)\bar\U_{k\perp}
=-\bar\U_k^\top\S_{k,l}(\s;\G_k,\H_k)\bar\U_{k\perp}.
\end{align*}
Multiplication by $1(\calE_k)$ preserves this oddness.  Hence each
summand has conditional expectation zero with respect to $\H_k$, given
$\calG$ and $\G_k$.

It remains only to justify summing these zero expectations.  The proof of
Lemma~\ref{lemma:boundDeltak} gives, on $\calE_k$,
\begin{align*}
\sum_{\s\in\SS_l}
\op{\bar\U_k^\top\S_{k,l}(\s)\bar\U_{k\perp}}
\leq (4C\varepsilon_k)^l.
\end{align*}
The standing perturbation condition makes
$\sum_{l\geq1}(4C\varepsilon_k)^l<\infty$.  Thus the series is absolutely
convergent on $\calE_k$, and conditional expectation may be interchanged
with the sum.  We obtain
\begin{align*}
\EE\bbrac{
\bar\U_k^\top\bDelta_k\bar\U_{k\perp}\,1(\calE_k)
\,\Bigm|\,\calG,\G_k}=0.
\end{align*}
Taking expectation over $\G_k$ proves
\begin{align*}
\EE_{\calG}\bbrac{
\bar\U_k^\top\bDelta_k\bar\U_{k\perp}\,1(\calE_k)}=0.
\end{align*}

\subsection{Proof of Lemma \ref{lemma:expectation-upper-bound}}

Condition on
$\calG=\sigma\brac{\U,\ebrac{\U_j,\V_j,\W_j}_{j=1}^{K}}$ throughout.
Recall that
\begin{align*}
\R_k
&=\U^\top\bar\U_k\bar\U_k^\top\bDelta_k
  \bar\U_k\bar\U_k^\top\U_\perp\bLa^{-1},\\
\bDelta_k
&=\bar\U_k\bar\U_k^\top-\tilde\U_k\tilde\U_k^\top.
\end{align*}
Set
\begin{align*}
\B_k
:=\EE_{\calG}\bbrac{
\bar\U_k^\top\bDelta_k\bar\U_k\,1(\calE_k)}.
\end{align*}
Because
\begin{align*}
\U^\top\bar\U_k
=\begin{bmatrix}\I_r&0\end{bmatrix},
\qquad
\bar\U_k^\top\U_\perp
=\begin{bmatrix}0\\ \U_k^\top\U_\perp\end{bmatrix},
\end{align*}
we have
\begin{align*}
\EE_{\calG}\bbrac{\R_k\,1(\calE_k)}
=\begin{bmatrix}\I_r&0\end{bmatrix}\B_k
 \begin{bmatrix}0\\ \U_k^\top\U_\perp\end{bmatrix}\bLa^{-1}.
\end{align*}
It is therefore enough to prove
\begin{align}\label{eq:Bk12-bound}
\op{[\B_k]_{12}}
\leq C\varepsilon_k^2
\brac{\delta_k(1-\delta_k^2)^{-1}\wedge1}.
\end{align}
Indeed, \eqref{eq:Bk12-bound},
$\op{\U_k^\top\U_\perp}\leq1$, and
$\op{\bLa^{-1}}\leq\theta^{-1}$ imply the stated bound.

Let
\begin{align*}
\Pi_k:=\begin{bmatrix}\I_r&0\\0&-\I_{r_k}\end{bmatrix}.
\end{align*}
For any matrix $\A\in\RR^{\bar r_k\times\bar r_k}$, write
\begin{align*}
\A_{\rm bd}:=\frac{\A+\Pi_k\A\Pi_k}{2},
\qquad
\A_{\rm od}:=\frac{\A-\Pi_k\A\Pi_k}{2}.
\end{align*}
The first matrix is block diagonal and the second is block off diagonal
with respect to the decomposition
$\bar\U_k=[\U,\U_k]$.  Lemma~\ref{lemma:block-offdiag-inverse-power}
therefore gives, for every positive half-integer $b$ that occurs below,
\begin{align}\label{eq:Gamma-offdiag-bound}
\op{(\bGa_k^{-b})_{\rm od}}
\leq C^{2b+1}\lambda_{k,\min}^{-2b}
\brac{\delta_k(1-\delta_k^2)^{-1}\wedge1}.
\end{align}

We next examine one term of the local projector expansion.  By
Lemma~\ref{lemma:boundDeltak}, the series in \eqref{eq:Deltak} is
absolutely convergent on $\calE_k$, so conditional expectation may be
interchanged with the sum.  The compression
$\bar\U_k^\top\S_{k,l}(\s)\bar\U_k$ can be nonzero only when
$s_1>0$ and $s_{l+1}>0$; hence $l\geq2$.  Define the four normalized
perturbation blocks
\begin{align*}
\calX_{k,11}
&:=\bGa_k^{-1/2}\bar\U_k^\top\bXi_k\bar\U_k\bGa_k^{-1/2},\\
\calX_{k,10}
&:=\bGa_k^{-1/2}\bar\U_k^\top\bXi_k\bar\U_{k\perp},\\
\calX_{k,01}
&:=\bar\U_{k\perp}^\top\bXi_k\bar\U_k\bGa_k^{-1/2},\\
\calX_{k,00}
&:=\bar\U_{k\perp}^\top\bXi_k\bar\U_{k\perp}.
\end{align*}
Equation \eqref{defu} gives, on $\calE_k$,
\begin{align}\label{eq:normalized-four-block-bounds}
\op{\calX_{k,11}}
&\leq C\varepsilon_k,\\
\op{\calX_{k,10}}+\op{\calX_{k,01}}
&\leq C\lambda_{k,\min}\varepsilon_k,\\
\op{\calX_{k,00}}
&\leq C\lambda_{k,\min}^2\varepsilon_k.
\end{align}
For each $j$ such that $s_j>0$, use one factor
$\bGa_k^{-1/2}$ for each adjacent perturbation block.  The remaining
factor at that index is $\bGa_k^{-(s_j-1/2)}$ at an endpoint and
$\bGa_k^{-(s_j-1)}$ at an interior index.  These exponents are
nonnegative.  Consequently, every term after left and right multiplication by $\bar\U_k^\top$ and $\bar\U_k$ can be written as
a product of the four blocks above and deterministic nonnegative
half-integer powers of $\bGa_k^{-1}$.

We now make the symmetry of these products explicit.  Define
\begin{align*}
\Q_k
:=\bar\R_k\bar\D_k^\top
  (\bar\D_k\bar\D_k^\top)^{-1/2},
\end{align*}
so that $\Q_k^\top\Q_k=\I_{\bar r_k}$, and extend $\Q_k$ to an
orthogonal matrix $[\Q_k,\Q_{k\perp}]$.  Write
\begin{align*}
\G_{k1}:=\G_k\Q_k,
\quad
\G_{k2}:=\G_k\Q_{k\perp},
\quad
\H_{k1}:=\H_k\Q_k,
\quad
\H_{k2}:=\H_k\Q_{k\perp}.
\end{align*}
Since
$\bar\D_k\bar\R_k^\top=\mat{\V_k^\top\\\W_k^\top}$ and
$\bGa_k=\bar\D_k\bar\D_k^\top$, direct multiplication gives
\begin{align}\label{eq:normalized-block-explicit}
\calX_{k,11}
&=\G_{k1}^\top\bGa_k^{-1/2}
  +\bGa_k^{-1/2}\G_{k1}
  +\bGa_k^{-1/2}
   (\G_k\G_k^\top-d_k\sigma_k^2\I)
   \bGa_k^{-1/2},\\
\calX_{k,10}
&=\H_{k1}^\top+\bGa_k^{-1/2}\G_k\H_k^\top,\\
\calX_{k,01}
&=\calX_{k,10}^\top,\\
\calX_{k,00}
&=\H_k\H_k^\top-d_k\sigma_k^2\I.
\end{align}
The centered covariance matrices in these formulas are kept intact,
because those are the matrices controlled by $\calE_k$.

Expand one such term only in the following two finite ways:
first use the sums in \eqref{eq:normalized-block-explicit}, and then replace
every positive deterministic inverse power by the sum of its block-diagonal
and block-off-diagonal parts.  There are at most $C^l$ resulting terms for
an order-$l$ term.

Consider the term in which every deterministic inverse power is replaced
by its block-diagonal part.  The transformation
\begin{align}\label{eq:noise-block-involution}
(\G_{k1},\G_{k2},\H_{k1},\H_{k2})
\longmapsto
(\Pi_k\G_{k1}\Pi_k,\Pi_k\G_{k2},
 \H_{k1}\Pi_k,\H_{k2})
\end{align}
preserves their joint Gaussian distribution.  To verify that it also
preserves $\calE_k$, let
\begin{align*}
\boldsymbol{\Omega}_k
:=\Q_k\Pi_k\Q_k^\top+\Q_{k\perp}\Q_{k\perp}^\top.
\end{align*}
This is an orthogonal matrix, and \eqref{eq:noise-block-involution}
corresponds in the original coordinates to
\begin{align*}
\G_k\longmapsto\Pi_k\G_k\boldsymbol{\Omega}_k,
\qquad
\H_k\longmapsto\H_k\boldsymbol{\Omega}_k.
\end{align*}
The matrices $\Q_k$ and $\bar\R_k$ have the same column space and both
have orthonormal columns, so $\Q_k^\top\bar\R_k$ is orthogonal.  Hence
\begin{align*}
\op{\Pi_k\G_k\boldsymbol{\Omega}_k\bar\R_k}
&=\op{\G_k\bar\R_k},
&
\op{\H_k\boldsymbol{\Omega}_k\bar\R_k}
&=\op{\H_k\bar\R_k}.
\end{align*}
Moreover,
\begin{align*}
(\Pi_k\G_k\boldsymbol{\Omega}_k)(\Pi_k\G_k\boldsymbol{\Omega}_k)^\top
&=\Pi_k\G_k\G_k^\top\Pi_k,\\
(\Pi_k\G_k\boldsymbol{\Omega}_k)(\H_k\boldsymbol{\Omega}_k)^\top
&=\Pi_k\G_k\H_k^\top,\\
(\H_k\boldsymbol{\Omega}_k)(\H_k\boldsymbol{\Omega}_k)^\top
&=\H_k\H_k^\top.
\end{align*}
Thus every norm in \eqref{eq:local-good-event} is unchanged.

Every block-diagonal deterministic factor commutes with $\Pi_k$.  From
\eqref{eq:normalized-block-explicit}, the transformation
\eqref{eq:noise-block-involution} sends
\begin{align*}
\calX_{k,11}&\longmapsto\Pi_k\calX_{k,11}\Pi_k,
&
\calX_{k,10}&\longmapsto\Pi_k\calX_{k,10},\\
\calX_{k,01}&\longmapsto\calX_{k,01}\Pi_k,
&
\calX_{k,00}&\longmapsto\calX_{k,00}
\end{align*}
for the term under consideration. Since $s_1>0$ and $s_{l+1}>0$, the transformed product is conjugated by $\Pi_k$.  Its
conditional expectation, including the invariant indicator $1(\calE_k)$,
therefore satisfies
\begin{align*}
\EE_{\calG}\bbrac{\text{term}\cdot1(\calE_k)}
=\Pi_k\EE_{\calG}\bbrac{\text{term}\cdot1(\calE_k)}\Pi_k.
\end{align*}
This expectation is block diagonal, so its $(1,2)$ block is zero.

Every other term in the finite expansion contains at least one
block-off-diagonal deterministic inverse power.  Bound one such factor by
\eqref{eq:Gamma-offdiag-bound}; bound every other factor
$\bGa_k^{-b}$ by $\lambda_{k,\min}^{-2b}$.  We now verify the remaining
powers explicitly.  Let $e_{11}$, $e_{10}$, and $e_{00}$ be the numbers
of consecutive index pairs of the three types used in the proof of
Lemma~\ref{lemma:boundDeltak}.  The normalized perturbation blocks in
\eqref{eq:normalized-four-block-bounds} contribute at most
\begin{align*}
(C\varepsilon_k)^l
\lambda_{k,\min}^{e_{10}+2e_{00}}.
\end{align*}
The inverse powers remaining after normalization contribute at most
\begin{align*}
\lambda_{k,\min}^{-2l+2e_{11}+e_{10}}.
\end{align*}
These powers cancel because
$e_{11}+e_{10}+e_{00}=l$.  The one factor bounded by
\eqref{eq:Gamma-offdiag-bound} contributes the additional multiplier
$\delta_k(1-\delta_k^2)^{-1}\wedge1$.  Each expanded term is therefore
bounded on $\calE_k$ by
\begin{align*}
C^l\varepsilon_k^l
\brac{\delta_k(1-\delta_k^2)^{-1}\wedge1}.
\end{align*}
The separate summands in \eqref{eq:normalized-block-explicit} obey the
same bounds as their grouped blocks: this follows directly from
$\op{\G_{k1}}=\op{\G_k\bar\R_k}$,
$\op{\H_{k1}}=\op{\H_k\bar\R_k}$, and the remaining bounds in
\eqref{eq:local-good-event}.  Thus no cancellation inside a centered
covariance block has been discarded in the preceding estimate.

It follows that, for every $l\geq2$ and $\s\in\SS_l$ with
$s_1,s_{l+1}>0$,
\begin{align*}
\op{\big[
\EE_{\calG}\{\bar\U_k^\top\S_{k,l}(\s)\bar\U_k\,1(\calE_k)\}
\big]_{12}}
\leq C^l\varepsilon_k^l
\brac{\delta_k(1-\delta_k^2)^{-1}\wedge1}.
\end{align*}
Finally, $|\SS_l|\leq4^l$, and the standing perturbation condition makes
$\varepsilon_k$ sufficiently small.  Therefore
\begin{align*}
\op{[\B_k]_{12}}
&\leq
\brac{\delta_k(1-\delta_k^2)^{-1}\wedge1}
\sum_{l\geq2}(4C\varepsilon_k)^l\\
&\leq C\varepsilon_k^2
\brac{\delta_k(1-\delta_k^2)^{-1}\wedge1}.
\end{align*}
This proves \eqref{eq:Bk12-bound}, and hence
\begin{align*}
\op{\EE_{\calG}\bbrac{\R_k\,1(\calE_k)}}
\leq C\varepsilon_k^2
\brac{\delta_k(1-\delta_k^2)^{-1}\wedge1}\theta^{-1}.
\end{align*}

\subsection{Proof of Lemma \ref{lemma:concentrationR:general}}

Recall $\R_k = \U^\top\bar\U_k\bar\U_k^\top\bDelta_k\bar\U_k\bar\U_k^\top\U_{\perp}\bLa^{-1}$. 
Expanding $\bDelta_k$ (see \eqref{eq:Deltak}) and we have 
\begin{align*} 
\sum_{k=1}^{K}w_k\R_k\cdot 1(\calE_k)= \sum_{k=1}^{K}\underbrace{w_k\U^\top\bar\U_{k}\sum_{l\geq 2}\sum_{\s\in\SS_l}\underline{\bar\U_{k}^\top\S_{k,l}(\s)\bar\U_k}\bar\U_k^\top\U_{\perp}\bLa^{-1}\cdot 1(\calE_k)}_{:=\W_k}.
\end{align*}
Here we use the fact when $l=1$, $\bar\U_{k}^\top\S_{k,1}(\s)\bar\U_k = 0$. 
Condition on $\calG$.  Let $\bar\W_k$ be the self-adjoint dilation of
$\W_k$,
\begin{align*}
    \bar\W_k = \mat{0&\W_k\\ \W_k^\top&0}.
\end{align*}
The termwise local projector bound and $|\SS_l|\leq4^l$ imply, for a
sufficiently small local perturbation,
\begin{align*}
    \op{\W_k}
    &\leq w_k\theta^{-1}
       \sum_{l\geq2}(4C\varepsilon_k)^l
     \leq Cw_k\theta^{-1}\varepsilon_k^2.
\end{align*}
Consequently,
\begin{align*}
\op{\bar\W_k-\EE_{\calG}\bar\W_k}
&\leq Cw_k\theta^{-1}\varepsilon_k^2,\\
(\bar\W_k-\EE_{\calG}\bar\W_k)^2
&\preceq Cw_k^2\theta^{-2}\varepsilon_k^4\I.
\end{align*}
The centered
noise-dependent dilations are independent across $k$ conditionally on
$\calG$.  Matrix Hoeffding therefore gives
\begin{align*}
    \PP\bigg(\op{\sum_{k=1}^{K}(\bar\W_k- \EE_{\calG}\bar\W_k)} \geq C\sqrt{\sum_{k=1}^{K}w_k^2\varepsilon_k^{4}\theta^{-2}}\sqrt{\log n}\,\Bigm|\,\calG \bigg)\leq n^{-10}.
\end{align*}
Since the norm of a rectangular matrix equals the norm of its self-adjoint
dilation, this proves the claim.

\subsection{Proof of Lemma \ref{lemma:concentrationUDU}}
We state a more general version of Lemma \ref{lemma:concentrationUDU}. 
\begin{lemma}\label{lemma:concentration-general}
Let $\N_k, \O_k$, $k\in[K]$, be $\calG$-measurable matrices of compatible
dimensions, with the row dimension of $\O_k$ and the column dimension of
$\N_k$ at most $n$, and suppose $\op{\O_k}\leq 1$. Then, for every
$p\in(0,1)$, conditionally on $\calG$, with probability at least $1-p$,
\begin{align*}
&\op{\sum_{k=1}^{K}\O_k\cdot\underline{\bar\U_{k}^\top\bDelta_{k}\bar\U_{k\perp}}\cdot\N_k\cdot 1(\calE_k)}\\
&\hspace{3cm}\leq C\sqrt{\sum_{k=1}^{K}n^{-1}\big(\tr(\N_k^\top\N_k)+\bar r_k\op{\N_k^\top\N_k}\big)\varepsilon_k^{2}}\sqrt{\log (n\vee p^{-1})}.
\end{align*}
\end{lemma}

\begin{proof}
Condition throughout on $\calG$.  By Lemma \ref{lemma:meanzero},
$\EE_{\calG}\bbrac{\bar\U_{k}^\top\bDelta_{k}\bar\U_{k\perp}\cdot 1(\calE_k)} = 0$.
We have 
\begin{align*} 
\sum_{k=1}^{K}\O_k\cdot\underline{\bar\U_{k}^\top\bDelta_{k}\bar\U_{k\perp}}\cdot\N_k\cdot 1(\calE_k)= \sum_{k=1}^{K}\O_k\sum_{l\geq 1}\sum_{\s\in\SS_l}\underline{\bar\U_{k}^\top\S_{k,l}(\s)\bar\U_{k\perp}}\N_k\cdot 1(\calE_k).
\end{align*}
Now we consider
\begin{align*}
    \op{ \sum_{k=1}^{K}\underbrace{\O_k{\underline{\bar\U_{k}^\top\S_{k,l}(\s)\bar\U_{k\perp}}}\N_k\cdot 1(\calE_k)}_{:=\W_k}}
\end{align*}
for fixed $l\geq 1$ and $\s\in\SS_l$. 
Let $\bar\W_k$ be the symmetric version of $\W_k$:
\begin{align*}
    \bar\W_k = \mat{0&\W_k\\ \W_k^\top&0}. 
\end{align*}
Then
\begin{align*}
    \EE\exp(\lambda\bar\W_k) = \EE\sum_{p \geq 0}\frac{\lambda^p\bar\W_k^p}{p!}.
\end{align*}
{For a nonzero summand of $\bar\U_k^\top\S_{k,l}(\s)\bar\U_{k\perp}$, $s_1>0$ and $s_{l+1}=0$.
Therefore the number of indices $j\in[l]$ for which exactly one of
$s_j,s_{j+1}$ is positive is odd. By \eqref{eq:UxiU}, precisely those
perturbation blocks change sign under $\H_k\mapsto-\H_k$; all other
blocks and $1(\calE_k)$ are unchanged. Hence
$\W_k\mapsto-\W_k$, so $\W_k\stackrel{d}{=}-\W_k$ and all its odd
moments vanish.}
Since the odd moments of $\bar\W_k$ vanish, we have 
\begin{align*}
    \EE\exp(\lambda\bar\W_k) = \EE\sum_{p \geq 0}\frac{\lambda^{2p}\bar\W_k^{2p}}{(2p)!}. 
\end{align*}
Notice 
\begin{align*}
    \bar\W_k^{2} = \mat{\W_k\W_k^\top &0\\ 0& \W_k^\top\W_k}\leq {\op{\W_k\W_k^\top}}\I. 
\end{align*}
Here $\A\leq \B$ means $\B-\A$ is SPSD. And therefore 
\begin{align*}
    \EE \bar\W_k^{2p} \leq \EE{\op{\W_k\W_k^\top}}^p\I.
\end{align*}
Next we bound the $\psi_1$ norm of ${\op{\W_k\W_k^\top}}$. 
To write the factorization explicitly, let
\begin{align*}
    j_0:=\max\{j\in[l]:s_j>0\},
    \qquad
    a:=l-j_0.
\end{align*}
Because $s_{l+1}=0$, we have
$s_{j_0}>0$ and $s_{j_0+1}=\cdots=s_{l+1}=0$.  The perturbation block
between the $j_0$th and $(j_0+1)$st factors is therefore
\begin{align*}
\Y_k
:=(\bar\D_k\bar\D_k^\top)^{-1/2}
  (\bar\D_k\bar\R_k^\top+\G_k)\H_k^\top,
\end{align*}
and each of the following $a$ perturbation blocks equals
\begin{align*}
\Z_k:=\H_k\H_k^\top-d_k\sigma_k^2\I.
\end{align*}
Consequently, up to its deterministic sign, the term has the form
$\B_{k,l,\s,a}^\top\Y_k\Z_k^a$.  Applying the explicit term-by-term
calculation in the proof of Lemma~\ref{lemma:boundDeltak} to the first
$l-a-1$ perturbation blocks, while leaving $\Y_k\Z_k^a$ unchanged,
gives on $\calE_k$
\begin{align*}
\op{\B_{k,l,\s,a}}
\leq C^l\varepsilon_k^{l-a-1}
\lambda_{k,\min}^{-2a-1}.
\end{align*}
The factor $C^l$ is included in the order-$l$ constants below.
Since $\op{\O_k}\leq 1$, 
\begin{align*}
    {\op{\W_k\W_k^\top}}
    &\leq \varepsilon_k^{2l-2a-2}\lambda_{k,\min}^{-4a-2}
    \op{\Y_k\Z_k^a{\N_k\N_k^\top}\Z_k^a\Y_k^\top}\cdot 1(\calE_k).
\end{align*}
{
\begin{lemma}\label{lemma:psione-norm}
For every integer $a\geq0$ and every SPSD matrix $\M_0$,
\begin{align*}
&\lambda_{k,\min}^{-4a-2}
\psione{
\op{
\Y_k\Z_k^a\M_0\Z_k^a\Y_k^\top
\cdot 1(\calE_k)
}
}\\
&\qquad\leq
C^{a+1}n^{-1}\varepsilon_k^{2a+2}
\bigl(\tr(\M_0)+\bar r_k\op{\M_0}\bigr),
\end{align*}
where $C>1$ is an absolute constant.
\end{lemma}
}
{Apply Lemma \ref{lemma:psione-norm} with
$\M_0=\N_k\N_k^\top$. Since
$\tr(\N_k\N_k^\top)=\tr(\N_k^\top\N_k)$ and
$\op{\N_k\N_k^\top}=\op{\N_k^\top\N_k}$,} 
{
\begin{align*}
    \psione{{\op{\W_k\W_k^\top}}}
    \leq
    C^{l+1}n^{-1}\varepsilon_k^{2l}
    \bigl(
      \tr(\N_k^\top\N_k)
      +\bar r_k\op{\N_k^\top\N_k}
    \bigr).
\end{align*}
}
And thus using the moments bound for sub-exponential random variables, we have 
\begin{align*}
    \EE \bar\W_k^{2p} \leq \EE{\op{\W_k\W_k^\top}}^p\I\leq {\big[C^{l+1}n^{-1}\varepsilon_k^{2l}\big(\tr(\N_k^\top\N_k)+\bar r_k\op{\N_k^\top\N_k}\big)\big]^p(p!)\I}. 
\end{align*}
Therefore 
\begin{align*}
    \EE\exp(\lambda\bar\W_k) &= \EE\sum_{p \geq 0}\frac{\lambda^{2p}\bar\W_k^{2p}}{(2p)!}\leq \sum_{p \geq 0}\frac{\lambda^{2p}{\big[C^{l+1}n^{-1}\varepsilon_k^{2l}\big(\tr(\N_k^\top\N_k)+\bar r_k\op{\N_k^\top\N_k}\big)\big]^p}}{p!}\I\\
    &= \exp\brac{\frac{1}{2}\lambda^2 \V_k},
\end{align*}
where ${\V_k = 2C^{l+1}n^{-1}\varepsilon_k^{2l}\big(\tr(\N_k^\top\N_k)+\bar r_k\op{\N_k^\top\N_k}\big)\I}$.
The dilations are independent across $k$. Applying the matrix Laplace bound to both signs of the self-adjoint sum and taking a union bound gives
\begin{align*}
    \PP\bigg(\op{\sum_{k=1}^{K}\bar\W_k} \geq \sigma_l\sqrt{\log(2n)+ \log p_l^{-1}}
    \bigg)
    \leq p_l
\end{align*}
where ${\sigma_l^2 \leq C^{l+1}\sum_{k=1}^{K}n^{-1}\varepsilon_k^{2l}\big(\tr(\N_k^\top\N_k)+\bar r_k\op{\N_k^\top\N_k}\big)}$ and $p_l>0$ is to be determined. 

Now taking union bound over all $l\geq 1$, and $\s\in\SS_l$, we get with probability exceeding $1 - \sum_{l\geq 1}4^lp_l$, 
\begin{align*}
\op{\sum_{k=1}^{K}\O_k\underline{\bar\U_{k}^\top\bDelta_k\bar\U_{k\perp}}\N_k\cdot 1(\calE_k)} \leq \sum_{l\geq 1}4^l\sigma_l\sqrt{\log(2n)+ \log p_l^{-1}}.
\end{align*}
Set $p_l:=p\,2^{-l}/|\SS_l|$. Then
$\sum_{l\geq1}|\SS_l|p_l=p$ and
$\log p_l^{-1}\leq\log p^{-1}+Cl$.  The resulting factor $\sqrt l$ is absorbed into the exponential-in-$l$ constants. We conclude with probability exceeding $1-p$, 
\begin{align*}
&\quad\op{\sum_{k=1}^{K}\O_k\underline{\bar\U_{k}^\top\bDelta_k\bar\U_{k\perp}}\N_k\cdot 1(\calE_k)} \\
&\leq \sum_{l\geq 1}C^l\sqrt{\sum_{k=1}^{K}n^{-1}\big(\tr(\N_k^\top\N_k)+\bar r_k\op{\N_k^\top\N_k}\big)\varepsilon_k^{2l}}\sqrt{\log n+\log p^{-1}}\\
&\leq \sum_{l\geq 1}\frac{1}{2^l}\sqrt{\sum_{k=1}^{K}n^{-1}\big(\tr(\N_k^\top\N_k)+\bar r_k\op{\N_k^\top\N_k}\big)\varepsilon_k^{2}}\sqrt{\log n+\log p^{-1}}\\
&\leq 2\sqrt{\sum_{k=1}^{K}n^{-1}\big(\tr(\N_k^\top\N_k)+\bar r_k\op{\N_k^\top\N_k}\big)\varepsilon_k^{2}}\sqrt{\log n+\log p^{-1}},
\end{align*}
where we use the fact $\varepsilon_k\leq c_0$ for some small absolute constant. 
\end{proof}

\subsection{Proof of Lemma \ref{lemma:psione-norm}}
{
Define
\begin{align*}
    \Q_k
    &:=
    \bar\R_k\bar\D_k^\top
    (\bar\D_k\bar\D_k^\top)^{-1/2}
    \in\RR^{d_k\times\bar r_k}.
\end{align*}
Then $\Q_k^\top\Q_k=\I_{\bar r_k}$. Extend $\Q_k$ to an
orthogonal matrix $\mat{\Q_k&\Q_{k,\perp}}$, and set
\begin{align*}
    \H_{k,1}:=\H_k\Q_k
    \in\RR^{(n-\bar r_k)\times\bar r_k},
    \qquad
    \H_{k,2}:=\H_k\Q_{k,\perp}
    \in\RR^{(n-\bar r_k)\times(d_k-\bar r_k)}.
\end{align*}
By rotational invariance, $\H_{k,1}$ and $\H_{k,2}$ are independent
Gaussian matrices with i.i.d.\ $N(0,\sigma_k^2)$ entries. Moreover,
\begin{align*}
    \Y_k
    &=
    \H_{k,1}^\top
    +(\bar\D_k\bar\D_k^\top)^{-1/2}
      \G_k\H_k^\top.
\end{align*}
Therefore,
\begin{align}
&\op{\Y_k\Z_k^a\M_0\Z_k^a\Y_k^\top}\cdot1(\calE_k)
\notag\\
&\quad\leq
2\op{\H_{k,1}^\top\Z_k^a\M_0\Z_k^a\H_{k,1}}\cdot1(\calE_k)
\notag\\
&\qquad+
2\lambda_{k,\min}^{-2}
\op{
\G_k\H_k^\top\Z_k^a\M_0\Z_k^a\H_k\G_k^\top
}\cdot1(\calE_k).
\label{eq:psione-Y-split}
\end{align}

Define the $\sigma(\H_k)$-measurable event
\begin{align*}
    \calF_{k,H}
    :=\left\{
      \op{\Z_k}\leq C_0(\sqrt{nd_k}+n)\sigma_k^2,\quad
      \op{\H_{k,1}}\leq C\sqrt n\,\sigma_k
    \right\}.
\end{align*}
Because $\Q_k=\bar\R_k\O_k$ for an orthogonal matrix $\O_k$, the
definition of $\calE_{k,1}$ and the bound on $\Z_k$ in
$\calE_{k,2}$ imply $\calE_k\subseteq\calF_{k,H}$. Thus, whenever
we condition on $\H_k$, we may replace $1(\calE_k)$ by
$1(\calF_{k,H})$ before applying a Gaussian Orlicz-norm bound. In
particular, on $\calF_{k,H}$,
\begin{align*}
    \op{\H_{k,1}}
    =\op{\H_k\bar\R_k}
    \leq C\sqrt n\,\sigma_k.
\end{align*}
Uniform conditional $\psi_1$ bounds below imply the corresponding
unconditional bounds by averaging the conditional exponential moments.

We first control the second term in \eqref{eq:psione-Y-split}.
Conditionally on $\H_k$, monotonicity, the inclusion
$\calE_k\subseteq\calF_{k,H}$, and
Lemma~\ref{lemma:psione-operatornorm} applied to $\G_k^\top$ give
\begin{align*}
&\left\|
\op{
\G_k\H_k^\top\Z_k^a\M_0\Z_k^a\H_k\G_k^\top
}\cdot1(\calE_k)
\right\|_{\psi_1\mid\H_k}\\
&\quad\leq
1(\calF_{k,H})C\sigma_k^2
\left\{
\tr(\H_k^\top\Z_k^a\M_0\Z_k^a\H_k)
+\bar r_k\op{\H_k^\top\Z_k^a\M_0\Z_k^a\H_k}
\right\}.
\end{align*}
On $\calF_{k,H}$,
\begin{align*}
    \op{\H_k}^2
    &=\op{\Z_k+d_k\sigma_k^2\I}
      \leq C(n\vee d_k)\sigma_k^2,
\end{align*}
so the last conditional norm is bounded uniformly in $\H_k$.
Averaging its conditional exponential-moment bound therefore gives
\begin{align}
&\psione{
\op{
\G_k\H_k^\top\Z_k^a\M_0\Z_k^a\H_k\G_k^\top
}\cdot1(\calE_k)
}
\notag\\
&\quad\leq
C^{a+1}(n\vee d_k)\sqbrac{(\sqrt{nd_k}+n)\sigma_k^2}^{2a}\sigma_k^4
\bigl(\tr(\M_0)+\bar r_k\op{\M_0}\bigr).
\label{eq:psione-G-term}
\end{align}

It remains to control the first term in
\eqref{eq:psione-Y-split}. Define
\begin{align*}
    \bDelta_1
    &:=
    \H_{k,1}\H_{k,1}^\top,\\
    \bDelta_2
    &:=
    \H_{k,2}\H_{k,2}^\top
    -(d_k-\bar r_k)\sigma_k^2\I_{n-\bar r_k},\\
    \bDelta_3
    &:=
    -\bar r_k\sigma_k^2\I_{n-\bar r_k}.
\end{align*}
Then
\begin{align*}
    \Z_k=\bDelta_1+\bDelta_2+\bDelta_3.
\end{align*}
This identity includes the factors $\sigma_k^2$ in both centered
Wishart components. On $\calE_k$,
\begin{align*}
    \max_{j\in[3]}\op{\bDelta_j}\leq C(\sqrt{nd_k}+n)\sigma_k^2.
\end{align*}
Indeed, the claim is immediate for $\bDelta_1$ and $\bDelta_3$, and
for $\bDelta_2$ it follows from
$\bDelta_2=\Z_k-\bDelta_1-\bDelta_3$.
Let
\begin{align*}
    \calF_{k,2}:=\left\{\op{\bDelta_2}\leq C(\sqrt{nd_k}+n)\sigma_k^2\right\}.
\end{align*}
Then $\calE_k\subseteq\calF_{k,2}$ and $\calF_{k,2}$ is
$\sigma(\H_{k,2})$-measurable. Since the matrix inside the norm is SPSD,
\begin{align*}
&\op{
\H_{k,1}^\top\Z_k^a\M_0\Z_k^a\H_{k,1}
}
\leq
\tr\brac{
\H_{k,1}^\top\Z_k^a\M_0\Z_k^a\H_{k,1}
}
=\op{\M_0^{1/2}\Z_k^a\H_{k,1}}_{\rm F}^2.
\end{align*}
Expanding $\Z_k^a$ and using
$\op{\sum_{\mathbf i}\X_{\mathbf i}}_{\rm F}^2
\leq3^a\sum_{\mathbf i}\op{\X_{\mathbf i}}_{\rm F}^2$ yields
\begin{align}
&\tr\brac{
\H_{k,1}^\top\Z_k^a\M_0\Z_k^a\H_{k,1}
}
\notag\\
&\quad\leq
3^a
\sum_{i_1,\ldots,i_a\in[3]}
\tr\brac{
\H_{k,1}^\top
\bDelta_{i_1}\cdots\bDelta_{i_a}
\M_0
\bDelta_{i_a}\cdots\bDelta_{i_1}
\H_{k,1}
}.
\label{eq:word-expansion}
\end{align}

Consider one word in \eqref{eq:word-expansion}. If the word contains
no $\bDelta_1$, its nonscalar part depends only on $\H_{k,2}$ and is
therefore independent of $\H_{k,1}$. Conditioning on $\H_{k,2}$, using
$1(\calE_k)\leq1(\calF_{k,2})$, and applying
Lemma~\ref{lemma:psione-operatornorm} gives a uniform conditional
$\psi_1$ bound. Averaging the conditional exponential moments gives
\begin{align*}
&\psione{
\op{
\H_{k,1}^\top
\bDelta_{i_1}\cdots\bDelta_{i_a}
\M_0
\bDelta_{i_a}\cdots\bDelta_{i_1}
\H_{k,1}
}\cdot1(\calE_k)
}\\
&\quad\leq
C^{a+1}\big\{(\sqrt{nd_k}+n)\sigma_k^2\big\}^{2a}\sigma_k^2
\bigl(\tr(\M_0)+\bar r_k\op{\M_0}\bigr).
\end{align*}

If the word contains $\bDelta_1$, let its last occurrence before
$\M_0$ be displayed as
\begin{align*}
    \bDelta_{i_1}\cdots\bDelta_{i_a}
    =\A\,\bDelta_1\,\B,
\end{align*}
where $\B$ contains only $\bDelta_2$ and the scalar matrix
$\bDelta_3$. Thus, up to a scalar factor, $\B=\bDelta_2^b$ for some
$b\geq0$. The corresponding matrix can be written as
\begin{align*}
&\bigl(\H_{k,1}^\top\A\H_{k,1}\bigr)
\bigl(
\H_{k,1}^\top\B\M_0\B^\top\H_{k,1}
\bigr)
\bigl(\H_{k,1}^\top\A\H_{k,1}\bigr)^\top.
\end{align*}
On $\calE_k$,
\begin{align*}
    \op{\H_{k,1}^\top\A\H_{k,1}}
    \leq
    C^{a+1}n\sigma_k^2\sqbrac{(\sqrt{nd_k}+n)\sigma_k^2}^{a-b-1}.
\end{align*}
After multiplying the factorization by $1(\calE_k)$, this eventwise
bound reduces the remaining random term to the middle factor times
$1(\calE_k)$. We then use
$1(\calE_k)\leq1(\calF_{k,2})$ and condition on $\H_{k,2}$.
Lemma~\ref{lemma:psione-operatornorm} controls the middle factor
uniformly on $\calF_{k,2}$. Averaging the resulting conditional
exponential-moment bound shows that its contribution is bounded by
\begin{align*}
&C^{a+1}n^2\sigma_k^6\sqbrac{(\sqrt{nd_k}+n)\sigma_k^2}^{2a-2}
\bigl(\tr(\M_0)+\bar r_k\op{\M_0}\bigr)\\
&\qquad\leq
C^{a+1}\sqbrac{(\sqrt{nd_k}+n)\sigma_k^2}^{2a}\sigma_k^2
\bigl(\tr(\M_0)+\bar r_k\op{\M_0}\bigr),
\end{align*}
where the last inequality uses $(\sqrt{nd_k}+n)\sigma_k^2\geq n\sigma_k^2$.
After summing the at most $3^a$ words in
\eqref{eq:word-expansion}, we obtain
\begin{align}
&\psione{
\op{
\H_{k,1}^\top\Z_k^a\M_0\Z_k^a\H_{k,1}
}\cdot1(\calE_k)
}
\notag\\
&\quad\leq
C^{a+1}\sqbrac{(\sqrt{nd_k}+n)\sigma_k^2}^{2a}\sigma_k^2
\bigl(\tr(\M_0)+\bar r_k\op{\M_0}\bigr).
\label{eq:psione-H1-term}
\end{align}
Combining \eqref{eq:psione-Y-split},
\eqref{eq:psione-G-term}, and
\eqref{eq:psione-H1-term}, and using
$n\sigma_k^2/\lambda_{k,\min}^2\lesssim1$, gives
\begin{align*}
&\psione{
\op{
\Y_k\Z_k^a\M_0\Z_k^a\Y_k^\top
\cdot1(\calE_k)
}
}\\
&\quad\leq
C^{a+1}\sqbrac{(\sqrt{nd_k}+n)\sigma_k^2}^{2a}\sigma_k^2
\bigl(\tr(\M_0)+\bar r_k\op{\M_0}\bigr)
\brac{1+d_k\snr_k^{-2}}.
\end{align*}
Finally, the standing perturbation condition gives
$n\sigma_k^2\lambda_{k,\min}^{-2}\lesssim1$.  Therefore
\begin{align*}
\lambda_{k,\min}^{-4}(\sqrt{nd_k}+n)^2\sigma_k^4
&\leq C(n^2+nd_k)\sigma_k^4\lambda_{k,\min}^{-4}\\
&\leq C\sqbrac{n\sigma_k^2\lambda_{k,\min}^{-2}+nd_k\sigma_k^4\lambda_{k,\min}^{-4}}
\leq C\varepsilon_k^2,
\end{align*}
and
\begin{align*}
\lambda_{k,\min}^{-2}\sigma_k^2
\brac{1+d_k\snr_k^{-2}}
&=n^{-1}\sqbrac{n\sigma_k^2\lambda_{k,\min}^{-2}+nd_k\sigma_k^4\lambda_{k,\min}^{-4}}\leq n^{-1}\varepsilon_k^2.
\end{align*}
The claimed bound follows.
}
\section{Technical Lemmas}\label{sec:pf:technicallemmas}
\begin{lemma}[Block-off-diagonal inverse-power bound]\label{lemma:block-offdiag-inverse-power}
Let $\V\in\RR^{d\times r_1}$ and $\W\in\RR^{d\times r_2}$, and suppose $\V$, $\W$, and $\begin{bmatrix}\V&\W\end{bmatrix}$ all have full column rank. Define
\begin{align*}
    \bGa:=\begin{bmatrix}
        \V^\top\V & \V^\top\W\\
        \W^\top\V & \W^\top\W
    \end{bmatrix},
    \qquad
    \delta:=\op{(\V^\top\V)^{-1/2}\V^\top\W(\W^\top\W)^{-1/2}},
\end{align*}
and let $\lambda$ be the smallest singular value of $\begin{bmatrix}\V&\W\end{bmatrix}$. Let $\Pi:=\begin{bmatrix}\I_{r_1}&0\\0&-\I_{r_2}\end{bmatrix}$. For any $\A\in\RR^{(r_1+r_2)\times(r_1+r_2)}$, define
\begin{align*}
    \A_{\mathrm{bd}}:=\frac{\A+\Pi\A\Pi}{2},
    \qquad
    \A_{\mathrm{od}}:=\frac{\A-\Pi\A\Pi}{2}.
\end{align*}
Then $\A_{\mathrm{bd}}$ is block diagonal and $\A_{\mathrm{od}}$ is block off-diagonal. Moreover, for any fixed $a>0$,
\begin{align*}
    \op{(\bGa^{-a})_{\mathrm{od}}}
    \le C_a\lambda^{-2a}\brac{\delta(1-\delta^2)^{-1}\wedge 1},
\end{align*}
where $C_a>0$ depends only on $a$. For the half-integer powers used in the projector expansion, one may take $C_a\leq C^{2a+1}$ for an absolute constant $C>1$.
\end{lemma}
\begin{proof}
Let
\begin{align*}
    \A_0:=\V^\top\V,
    \qquad
    \B_0:=\W^\top\W,
    \qquad
    \C_0:=\V^\top\W.
\end{align*}
For any $t\ge 0$, set $\A_t:=\A_0+t\I_{r_1}$, $\B_t:=\B_0+t\I_{r_2}$, and $\T_t:=\A_t^{-1/2}\C_0\B_t^{-1/2}$. Since $\A_t\succeq \A_0$ and $\B_t\succeq \B_0$, we have $\op{\T_t}\le \op{\A_0^{-1/2}\C_0\B_0^{-1/2}}=\delta$. By the Schur complement formula,
\begin{align*}
    [(\bGa+t\I)^{-1}]_{12}
    =-\A_t^{-1/2}\T_t(\I-\T_t^\top\T_t)^{-1}\B_t^{-1/2}.
\end{align*}
Since the smallest eigenvalue of $\bGa$ is $\lambda^2$, the principal submatrices $\A_0$ and $\B_0$ have smallest eigenvalues lower bounded by $\lambda^2$. Hence
\begin{align*}
    \op{[(\bGa+t\I)^{-1}]_{12}}
    \le \frac{\delta}{1-\delta^2}(t+\lambda^2)^{-1}.
\end{align*}
The trivial bound $\op{[(\bGa+t\I)^{-1}]_{12}}\le \op{(\bGa+t\I)^{-1}}\le (t+\lambda^2)^{-1}$ further gives
\begin{align}\label{eq:resolvent-offdiag-bound}
    \op{[(\bGa+t\I)^{-1}]_{12}}
    \le \brac{\delta(1-\delta^2)^{-1}\wedge1}(t+\lambda^2)^{-1}.
\end{align}
In particular, taking $t=0$ proves the desired bound for $a=1$.

For $0<a<1$, the resolvent representation of fractional powers gives
\begin{align*}
    \bGa^{-a}=c_a\int_0^\infty t^{-a}(\bGa+t\I)^{-1}dt.
\end{align*}
Combining this with \eqref{eq:resolvent-offdiag-bound}, we obtain
\begin{align*}
    \op{[\bGa^{-a}]_{12}}
    \le C_a\brac{\delta(1-\delta^2)^{-1}\wedge1}\int_0^\infty t^{-a}(t+\lambda^2)^{-1}dt
    \le C_a\lambda^{-2a}\brac{\delta(1-\delta^2)^{-1}\wedge1}.
\end{align*}
The same bound holds for the $(2,1)$ block by symmetry, and hence for $(\bGa^{-a})_{\mathrm{od}}$.

For a general fixed $a>0$, write $a=m+b$ with $m\in\NN\cup\{0\}$ and $b\in(0,1]$. Expanding
\begin{align*}
    \bGa^{-a}=(\bGa^{-1})^m\bGa^{-b}
\end{align*}
after decomposing each factor into its block-diagonal and block-off-diagonal parts, every term contributing to the block-off-diagonal part contains at least one block-off-diagonal factor. Applying the preceding bound to one such factor and the ordinary bound $\op{\bGa^{-c}}\le \lambda^{-2c}$ to all remaining factors gives
\begin{align*}
    \op{(\bGa^{-a})_{\mathrm{od}}}
    \le C_a\lambda^{-2a}\brac{\delta(1-\delta^2)^{-1}\wedge1},
\end{align*}
where the constant absorbs the finite number of product terms and depends only on $a$.  If $a$ is a positive half-integer, then $b\in\{1/2,1\}$ and the expansion contains at most $2^{m+1}$ product terms.  The constants in the two base cases are absolute, which yields $C_a\leq C^{2a+1}$.
\end{proof}

\begin{lemma}\label{lemma:prob-psione}
    If we have for $A,B,C>0$, 
    \begin{align*}
        \PP(|X|\geq A(t+C)+B\sqrt{t+C} )\leq 2\exp(-t),
    \end{align*}
    for all $t>0$, then we have
    \begin{align*}
        \psione{X}\leq C_0\brac{AC + A + \frac{B^2}{A}}. 
    \end{align*}
    for some absolute constant $C_0>0$. 
\end{lemma}
\begin{proof}
For every $p\geq1$, integrate the tail probability using
\begin{align*}
\EE|X|^p
=p\int_0^\infty s^{p-1}\PP(|X|\geq s)\,ds.
\end{align*}
Split the integral at $AC+B\sqrt C$. On the remaining part, use the
change of variable $s=A(t+C)+B\sqrt{t+C}$. The assumed tail bound gives
\begin{align*}
\EE|X|^p
&\leq (AC+B\sqrt C)^p\\
&\quad+2p\int_0^\infty
\sqbrac{A(t+C)+B\sqrt{t+C}}^{p-1}
\left(A+\frac{B}{2\sqrt{t+C}}\right)e^{-t}\,dt.
\end{align*}
Using $(x+y)^p\leq2^{p-1}(x^p+y^p)$ and the standard bounds for the
Gamma integrals $\int_0^\infty t^q e^{-t}dt$, the last display implies
\begin{align*}
    \bigl(\EE|X|^p\bigr)^{1/p}
    \leq C_1\sqbrac{A(p+C)+B\sqrt{p+C}}.
\end{align*}
The standard moment characterization
$\psione{X}\asymp\sup_{p\geq1}p^{-1}(\EE|X|^p)^{1/p}$ now gives
\begin{align*}
\psione{X}
&\leq C_2\sqbrac{A+AC+B+B\sqrt C}\\
&\leq C_0\sqbrac{AC+A+\frac{B^2}{A}}.
\end{align*}
The last step follows twice from Young's inequality:
$B\leq(A+B^2/A)/2$ and
$B\sqrt C\leq(AC+B^2/A)/2$.
\end{proof}

\begin{lemma}\label{lemma:psione-operatornorm}
Let $\G\in\RR^{m\times r}$ have independent mean-zero sub-Gaussian entries
with $\max_{ij}\psitwo{\G_{ij}}\leq1$. Then, for every SPSD matrix
$\M_0\in\RR^{m\times m}$,
\begin{align*}
    \psione{\op{\G^\top\M_0\G}}
    \leq C_1\bigl\{\tr(\M_0)+r\op{\M_0}\bigr\},
\end{align*}
where $C_1>0$ is an absolute constant.
\end{lemma}
\begin{proof}
If $\M_0=0$, the claim is immediate.  For a fixed
$\x\in\mathbb S^{r-1}$, the coordinates of $\G\x$ are independent,
mean-zero, and have uniformly bounded sub-Gaussian norms.  The
Hanson--Wright inequality therefore gives, for every $t>0$,
\begin{align*}
\PP\bigg(
  \x^\top\G^\top\M_0\G\x
  \geq C\tr(\M_0)
  +C\fro{\M_0}\sqrt t
  +C\op{\M_0}t
\bigg)
\leq2e^{-t}.
\end{align*}
Let $\mathcal N$ be a $1/4$-net of $\mathbb S^{r-1}$ with
$|\mathcal N|\leq9^r$. Since $\G^\top\M_0\G$ is SPSD,
\begin{align*}
\op{\G^\top\M_0\G}
\leq2\max_{\x\in\mathcal N}
\x^\top\G^\top\M_0\G\x.
\end{align*}
A union bound, with $t$ replaced by $t+Cr$, yields
\begin{align*}
\PP\bigg(
\op{\G^\top\M_0\G}
\geq C\tr(\M_0)
+C\fro{\M_0}\sqrt{t+r}
+C\op{\M_0}(t+r)
\bigg)
\leq2e^{-t}.
\end{align*}
Applying Lemma~\ref{lemma:prob-psione} to the centered tail part and adding
the deterministic term gives
\begin{align*}
\psione{\op{\G^\top\M_0\G}}
&\leq C\brac{
\tr(\M_0)+r\op{\M_0}
+\frac{\fro{\M_0}^2}{\op{\M_0}}
}\\
&\leq C_1\brac{\tr(\M_0)+r\op{\M_0}},
\end{align*}
where the last inequality uses
$\fro{\M_0}^2\leq\op{\M_0}\tr(\M_0)$.
\end{proof}

\begin{lemma}\label{lemma:psitwo-operatornorm}
    Let $\X\in\RR^{m\times n}$ be a random matrix whose entries $\X_{ij}$ are independent mean-zero sub-gaussian random variables. Then we have 
    \begin{align*}
        \psitwo{\op{\X}} \leq c\cdot\max_{ij}\psitwo{\X_{ij}}\sqrt{m+n}. 
    \end{align*}
\end{lemma}
\begin{proof}
    From \cite[Theorem 4.4.5]{vershynin2018high}, we have 
    \begin{align*}
        \PP\big(\op{\X}\geq c\cdot\max_{ij}\psitwo{\X_{ij}}(\sqrt{m+n}+t)\big)\leq 2e^{-t^2}. 
    \end{align*}
    Together with Lemma \ref{lemma:prob-psitwo}, we finalize the proof. 
\end{proof}

\begin{lemma}\label{lemma:prob-psitwo}
    If we have 
    \begin{align*}
        \PP(|X|\geq C_0(K+t))\leq 2\exp(-t^2),
    \end{align*}
    for all $t>0$ and $K\geq1$, then we have 
    \begin{align*}
        \psitwo{X}\leq 3C_0K. 
    \end{align*}
\end{lemma}
\begin{proof}
We will show $\PP(|X|\geq s)\leq 2\exp(-s^2/(3C_0K)^2)$ for all $s>0$. When $0<s\leq 2C_0K$, we have 
\begin{align*}
    \PP(|X|\geq s)\leq 1 \leq 2\exp(-(2C_0K)^2/(3C_0K)^2) \leq 2\exp(-s^2/(3C_0K)^2).
\end{align*}
When $s>2C_0K$, we have 
\begin{align*}
    \frac{s}{C_0} - K \geq \frac{s}{C_0} - \frac{s}{2C_0} = \frac{s}{2C_0}. 
\end{align*}
And thus $(\frac{s}{C_0} - K)^2\geq \frac{s^2}{(2C_0)^2}$. So we conclude
\begin{align*}
    \PP(|X|\geq s) &= \PP\brac{|X|\geq C_0\brac{K +\frac{s}{C_0} - K}}\\
    &\leq 2\exp\Big(-\brac{\frac{s}{C_0} - K}^2\Big)\leq 2\exp\Big(-\frac{s^2}{(2C_0)^2}\Big)\leq 2\exp\Big(-\frac{s^2}{(3C_0K)^2}\Big),
\end{align*}
where in the last inequality we use $K\geq 1$. 
\end{proof}

\section{Deferred Details from the Main Text}\label{sec:rate-comparison}
In the main text, to facilitate a clear visual comparison, we focused on the proportional scaling regime where $n \asymp d$. In this section, we provide the algebraic comparison for the general setting, showing that, in the perturbative regime considered in this paper, our bound improves the corresponding terms in \cite{yang2025estimating}, up to constants and logarithmic factors.
From \cite[Theorem 1]{yang2025estimating}, they provided the following upper bound:
\begin{align*}
 \snr^{-1}\sqrt{\frac{n}{K} + \frac{1}{K\theta}}\log^{5/2}(n+d)  + \snr^{-2}\bigg(\frac{1}{\theta}+\frac{1}{K\theta^2}\bigg)(\sqrt{nd} + n)\log^{5/2}(n+d). 
\end{align*}
And our bound gives: 
\begin{align*}
   \underbrace{\varepsilon \sqrt{\frac{\log n}{K} + \frac{\log n}{Kn\theta}} }_{\textrm{first-order}} \hspace{0.1cm}+\hspace{0.1cm}\underbrace{\varepsilon^2\sqrt{\frac{\log n}{K\theta^2}} \hspace{0.1cm}+\hspace{0.1cm}\varepsilon^2\cdot \frac{\sum_{k=1}^{K}\big(\delta_k(1-\delta_k^2)^{-1}\wedge 1\big)}{K\theta}}_{\textrm{second-order}}.
\end{align*}
Substituting
$\varepsilon=\sqrt n\,\snr^{-1}+\sqrt{nd}\,\snr^{-2}$ directly into
the first-order term gives
\begin{align*}
&\varepsilon\sqrt{\frac{\log n}{K}+\frac{\log n}{Kn\theta}}=
\snr^{-1}\sqrt{\frac{n\log n}{K}+\frac{\log n}{K\theta}}
+\snr^{-2}\sqrt{\frac{nd\log n}{K}+\frac{d\log n}{K\theta}}.
\end{align*}
The first summand is bounded by the first term of the Yang--Ma bound,
up to its larger logarithmic factor. For the second summand, since
$K\geq1$, $0<\theta\leq1$, and $n\geq1$,
\begin{align*}
\sqrt{\frac{nd}{K}+\frac{d}{K\theta}}
&\leq \sqrt{nd}+\sqrt{d/\theta}
 \leq 2\theta^{-1}(\sqrt{nd}+n).
\end{align*}
It is therefore bounded by the second term of the Yang--Ma bound, again
with a smaller logarithmic factor.

For the second-order part, use $\frac1K\sum_{k=1}^K
\brac{\delta_k(1-\delta_k^2)^{-1}\wedge1}\leq1$
to obtain
\begin{align*}
&\varepsilon^2\theta^{-1}
\left\{\sqrt{\frac{\log n}{K}}
+\frac1K\sum_{k=1}^K
\brac{\delta_k(1-\delta_k^2)^{-1}\wedge1}\right\}\leq
C\theta^{-1}\varepsilon^2\sqrt{\log(n+d)}.
\end{align*}
Moreover,
\begin{align*}
\varepsilon^2
&\leq2n\snr^{-2}+2nd\snr^{-4}\le C\snr^{-2}(n+\sqrt{nd}),
\end{align*}
where the last inequality uses the standing perturbative condition
$\sqrt{nd}\,\snr^{-2}\leq c_0\leq1$. Thus the second-order part is
also bounded by the second term of the Yang--Ma bound, up to constants
and with a smaller logarithmic factor. This proves the claimed
algebraic comparison in the general $n,d$ regime.

\end{document}